\documentclass[11pt,a4paper]{article}
\usepackage[utf8]{inputenc}
\usepackage[english]{babel}
\usepackage[T1]{fontenc} % -> Im PDF kann man Sonderzeichen wie bei Tur\'an kopieren
\usepackage{lmodern} %sinnvoll, Schrift nicht mehr pixlig
\usepackage{xcolor} %Paket fuer Farbe
\usepackage{color}
\usepackage{amssymb} %fuer natuerliche-zahlen-symbol
\usepackage{mathrsfs} %fuer geschwungenes L
\usepackage{amsmath}
\usepackage{amsthm}%fuer Theoreme, proof
\usepackage{amsopn}
\usepackage{enumitem}%for enumerate/itemize without indentation -> [wide=0pt] or [leftmargin=*] -> for [label=\textbf{S.\arabic*}]
\usepackage{graphicx} % für \bigcupdot notwendig
\usepackage{subcaption}
\usepackage{xspace} %bei Abkuerzung von d.h. \xspace in newcommand
\usepackage{mathtools} %noetig fuer declare paired delimiter -> gaussklammer so definiert
\usepackage{scalerel,stackengine} % for the newcommand of the of fourier hat
\usepackage{url}
\usepackage{hyperref}
\usepackage{doi}
\usepackage{tikz}
\usepackage{tikz-cd} %commuting diagram
\usepackage{nomencl} % list of notation
\usepackage[margin=3cm]{geometry}
\usepackage{pictexwd,dcpic}

\theoremstyle{plain}

\newtheorem{Theorem}{Theorem}[section]
\newtheorem{Lemma}[Theorem]{Lemma}
\newtheorem{Corollary}[Theorem]{Corollary}
\newtheorem{Proposition}[Theorem]{Proposition}

\theoremstyle{definition}

\newtheorem{Definition}[Theorem]{Definition}
\newtheorem{Example}[Theorem]{Example}

\theoremstyle{remark}

\newtheorem{Remark}[Theorem]{Remark}
\makeatletter
\newtheoremstyle{case}{3pt}{3pt}{\addtolength{\@totalleftmargin}{2.0em}\addtolength{\linewidth}{-1.5em}\parshape 1 1.5em \linewidth}{}{\kern-1.5em\itshape}{.}{ }{\thmnote{#3}}
\makeatother
\theoremstyle{case}
\newtheorem*{case}{}

\makeatletter
\newtheoremstyle{caseinside}{3pt}{3pt}{\addtolength{\@totalleftmargin}{4.0em}\addtolength{\linewidth}{-3.0em}\parshape 1 3.0em \linewidth}{}{\kern-1.5em\itshape}{.}{ }{\thmnote{#3}}
\makeatother
\theoremstyle{caseinside}
\newtheorem*{caseinside}{}

\allowdisplaybreaks[3]

\DeclareMathOperator{\aff}{aff} % affine hull

\DeclareMathOperator{\diag}{diag}
\DeclareMathOperator{\dist}{dist}
\DeclareMathOperator{\E}{E}     %Euclidean group
\DeclareMathOperator{\GL}{\mathrm{GL}}
\DeclareMathOperator{\rank}{rank}

\DeclareMathOperator{\Skew}{Skew}
\DeclareMathOperator{\SO}{\mathrm{SO}}
\DeclareMathOperator{\spano}{span}  % \span is a primitive command used in \multicolumn.

\DeclareMathOperator{\Sym}{\mathrm{Sym}} % symmetric
\DeclareMathOperator{\Trans}{Trans} % group of translations \trans(d)

\renewcommand{\O}{\operatorname{{\mathrm O}}} %orthogonal group
\DeclarePairedDelimiter{\floor}{\lfloor}{\rfloor} %\floor[\big]{x} \floor[\Big]{x} \floor[\bigg]{x} \floor[\Bigg]{x}
\DeclarePairedDelimiter{\abs}{\lvert}{\rvert}
\DeclarePairedDelimiter{\angles}{\langle}{\rangle}
\DeclarePairedDelimiter{\norm}{\lVert}{\rVert}
\DeclarePairedDelimiter{\parens}{\lparen}{\rparen}
\DeclarePairedDelimiter{\bracks}{\lbrack}{\rbrack}
\DeclarePairedDelimiter{\braces}{\lbrace}{\rbrace}
\DeclarePairedDelimiter{\ceil}{\lceil}{\rceil}
\DeclarePairedDelimiterX{\iso}[2]{(}{)}{#1,#2} % (A,b) is better than (A|b) since E(d)=O(d)\ltimes\R^d, but for examples like (0&0\\1&1|0\\1) is (A|b) better than (A,b)
\DeclarePairedDelimiterX{\set}[2]{\{}{\}}{#1\,\delimsize|\,#2}
\DeclarePairedDelimiterX{\scalar}[2]{\langle}{\rangle}{#1,#2}
\makeatletter
\newcommand{\@absonestar}[1]{\abs*{#1}_1}
\newcommand{\@absonenostar}[2][]{\abs[#1]{#2}_1}
\newcommand{\absone}{\@ifstar\@absonestar\@absonenostar}
\makeatother
\makeatletter
\newcommand{\@zeronormstar}[2]{\norm*{#1}_{#2,0}}
\newcommand{\@zeronormnostar}[3][]{\norm[#1]{#2}_{#3,0}}
\newcommand{\zeronorm}{\@ifstar\@zeronormstar\@zeronormnostar}
\makeatother
\makeatletter
\newcommand{\@nablanormstar}[2]{\norm*{#1}_{#2,\nabla}}
\newcommand{\@nablanormnostar}[3][]{\norm[#1]{#2}_{#3,\nabla}}
\newcommand{\nablanorm}{\@ifstar\@nablanormstar\@nablanormnostar}
\makeatother
\makeatletter
\newcommand{\@nablazeronormstar}[2]{\norm*{#1}_{#2,0,\nabla,0}}
\newcommand{\@nablazeronormnostar}[3][]{\norm[#1]{#2}_{#3,\nabla,0}}
\newcommand{\nablazeronorm}{\@ifstar\@nablazeronormstar\@nablazeronormnostar}
\makeatother
\makeatletter
\newcommand{\@newnormstar}[2]{\norm*{#1}_{#2,0,0}}
\newcommand{\@newnormnostar}[3][]{\norm[#1]{#2}_{#3,0,0}}
\newcommand{\newnorm}{\@ifstar\@newnormstar\@newnormnostar}
\makeatother
\makeatletter
\newcommand{\@newnablanormstar}[2]{\norm*{#1}_{#2,\nabla,0,0}}
\newcommand{\@newnablanormnostar}[3][]{\norm[#1]{#2}_{#3,\nabla,0,0}}
\newcommand{\newnablanorm}{\@ifstar\@newnablanormstar\@newnablanormnostar}
\makeatother
\makeatletter
\providecommand*{\cupdot}{%
  \mathbin{%
    \mathpalette\@cupdot{}%
  }%
}
\newcommand*{\@cupdot}[2]{%
  \ooalign{%
    $\m@th#1\cup$\cr
    \sbox0{$#1\cup$}%
    \dimen@=\ht0 %
    \sbox0{$\m@th#1\cdot$}%
    \advance\dimen@ by -\ht0 %
    \dimen@=.5\dimen@
    \hidewidth\raise\dimen@\box0\hidewidth
  }%
}
\providecommand*{\bigcupdot}{%
  \mathop{%
    \vphantom{\bigcup}%
    \mathpalette\@bigcupdot{}%
  }%
}
\newcommand*{\@bigcupdot}[2]{%
  \ooalign{%
    $\m@th#1\bigcup$\cr
    \sbox0{$#1\bigcup$}%
    \dimen@=\ht0 %
    \advance\dimen@ by -\dp0 %
    \sbox0{\scalebox{2}{$\m@th#1\cdot$}}%
    \advance\dimen@ by -\ht0 %
    \dimen@=.5\dimen@
    \hidewidth\raise\dimen@\box0\hidewidth
  }%
}
\makeatother

\stackMath
\newcommand\fourier[1]{%
\savestack{\tmpbox}{\stretchto{%
  \scaleto{%
    \scalerel*[\widthof{\ensuremath{#1}}]{\kern-.6pt\bigwedge\kern-.6pt}%
    {\rule[-\textheight/2]{1ex}{\textheight}}%WIDTH-LIMITED BIG WEDGE
  }{\textheight}% 
}{0.5ex}}%
\stackon[1pt]{#1}{\tmpbox}%
}
\par

\newcommand\+{\mkern2mu} %spacing for \exists -> https://tex.stackexchange.com/questions/101018/should-using-spaces-in-math-mode-be-a-common-thing
\newcommand{\conj}[1]{\overline{#1}}
\newcommand{\dual}[1]{\widehat{#1}}
\newcommand{\fdot}{{{}\cdot{}}} %spacing of $f(\cdot)$ in not equal to that of $f({}\cdot{})$, also $\lvert\cdot\rvert$ vs $\lvert{}\cdot{}\rvert$ -> good for norm and abs
\newcommand{\gdot}{\cdot} %group action
\newcommand{\gplus}[2]{#1\oplus #2}

\newcommand{\euler}{\mathrm e}
\newcommand{\iu}{\mathrm i}
\newcommand{\C}{{\mathbb C}}
\newcommand{\N}{{\mathbb N}}
\newcommand{\Q}{{\mathbb Q}}
\newcommand{\R}{{\mathbb R}}
\newcommand{\Z}{{\mathbb Z}}
\renewcommand{\epsilon}{\varepsilon} %Bernd benutzt varepsilon
\renewcommand{\phi}{\varphi} %Bernd benutzt varphi (cut-off function)

\newcommand{\A}{\mathcal A}
\newcommand{\F}{\mathcal F} %'big' finite group
\newcommand{\G}{\mathcal G}

\newcommand{\id}{{id}} %identity of a group
\newcommand{\MM}{M_0} % M_0=\set{n\in\N}{\T^m is a normal subgroup of \G}
\newcommand{\CC}{\mathcal C} % Cuboid
\newcommand{\EE}{\mathfrak E} % irreducble representations
\newcommand{\RR}{\mathcal R} % range (derivative)
\newcommand{\SG}{\mathcal S} %space group in R^{d_1}
\newcommand{\T}{\mathcal T} %'big' translation group
\newcommand{\daff}{{d_{\mathrm{aff}}}} %dimension of the objective structure
\newcommand{\rot}{L} % linear component
\newcommand{\trans}{\tau} % translational component
\newcommand{\transtwo}{\tau_2}
\newcommand{\proj}[1]{\bar{#1}}

\newcommand{\UIso}{U_{\mathrm{iso},0,0}}
\newcommand{\Uiso}[1]{U_{\mathrm{iso}}(#1)}
\newcommand{\Viso}[1]{V_{\mathrm{iso}}(#1)}

\newcommand{\zeroUiso}[1]{U_{\mathrm{iso},0}(#1)}
\newcommand{\Unewiso}[1]{U_{\mathrm{iso},0,0}(#1)}

\newcommand{\URot}{U_{\mathrm{rot},0,0}}
\newcommand{\Urot}[1]{U_{\mathrm{rot}}(#1)}
\newcommand{\zeroUrot}[1]{U_{\mathrm{rot},0}(#1)}
\newcommand{\Unewrot}[1]{U_{\mathrm{rot},0,0}(#1)}
\newcommand{\UTrans}{U_{\mathrm{trans}}}
\newcommand{\Utrans}[1]{U_{\mathrm{trans}}(#1)}
\newcommand{\UPer}{U_{\mathrm{per}}}
\newcommand{\UPerC}{U_{\mathrm{per},\C}}
\newcommand{\Per}{L_\mathrm{per}^\infty}
\newcommand{\semip}[2][]{p_{#2}^{#1}}

\newcommand{\newsemip}[2][]{p_{0,#2}^{#1}}
\newcommand{\semiq}[3][]{q_{#2,#3}^{#1}}

\newcommand{\newsemiq}[3][]{q_{0,#2,#3}^{#1}}
\newcommand{\xzeroone}{x_{0,1}} % x_0 = (x_{0,1},x_{0,2}) \in \R^{d_1,d_2}
\newcommand{\xzerotwo}{x_{0,2}}

\newcommand{\eg}{\mbox{e.\,g.}\xspace}
\newcommand{\ie}{\mbox{i.\,e.}\xspace}
\begin{document}
\begin{center}
\begin{Large}
Korn type Inequalities for Objective Structures
\end{Large}
\\[0.5cm]
\begin{large}
Bernd Schmidt\footnote{Universität Augsburg, Germany, {\tt bernd.schmidt@math.uni-augsburg.de}} and 
Martin Steinbach\footnote{Universität Augsburg, Germany, {\tt steinbachmartin@gmx.de}} 
\end{large}
\\[0.5cm]
\today
\\[1cm]
\end{center}

\begin{abstract}
We establish discrete Korn type inequalities for particle systems within the general class of objective structures that represents a far reaching generalization of crystal lattice structures. For space filling configurations whose symmetry group is a general space group we obtain a full discrete Korn inequality. For systems with non-trivial codimension our results provide an intrinsic rigidity estimate within the extended dimensions of the structure. As their continuum counterparts in elasticity theory, such estimates are at the core of energy estimates and, hence, a stability analysis for a wide class of atomistic particle systems. 
\end{abstract}



\tableofcontents

%-------------------------------------------------------------------
%-------------------------------------------------------------------
\section{Introduction}

%
%
%$\Vrot{\RR}$
%
%$\VrotC{\RR}$
%
%$\zeroVrot{\RR}$
%
%$\zeroVrotC{\RR}$
%
%$\Urot{\RR}$
%
%$\zeroUrot{\RR}$
%
%$\semip{\RR}$
%
%$\zerosemip{\RR}$
%
%$\semiq{\RR_1}{\RR_2}{\RR_3}$
%
%$\zerosemiq{\RR_1}{\RR_2}{\RR_3}$
%

The classical Korn inequality provides a quantitative rigidity estimate for $H^1$ functions in terms of their symmetrized gradient: If $\Omega \subset \R^d$ is bounded, connected and sufficiently regular (\eg, Lipschitz), then for all $u \in H^1(\Omega,\R^d)$ 
$$ \min \set[\big]{\norm{\nabla u - A}_{L^2(\Omega)}}{A \in \Skew(d)} 
   \le C \norm{(\nabla u)^T + \nabla u}_{L^2(\Omega)}, $$
cf., \eg, \cite{Ciarlet:88}. This inequality is of paramount importance in linear elasticity theory since the elastic energy of an infinitesimal displacement $u \colon \Omega \to \R^d$ dominates the $L^2$ norm of the symmetrized gradient $\frac12((\nabla u)^T + \nabla u)$ but not the full gradient $\nabla u$. As a consequence, the elastic energy controls the deviation of $\nabla u$ from a single skew symmetric matrix $A$ and hence the deviation of $u$ from an infinitesimal rigid motion of the form $x \mapsto Ax + c$. An immediate corollary is the corresponding qualitative rigidity result which states that $(\nabla u)^T + \nabla u = 0$ a.e.\ on $\Omega$ implies that $u(x) = Ax + c$ for some $A \in \Skew(d)$, $c \in \R^d$. 

For our purposes it turns out to be useful to re-write Korn's inequality in terms of projection-induced seminorms as follows. Denoting by $\pi_{\rm rot} \colon \R^{d \times d} \to \R^{d \times d}$, $\pi_{\rm rot} M = \frac12 (M^T + M)$ the orthogonal projection of $d \times d$ matrices onto their symmetric part (whose kernel is the set of infinitesimal rotations $\Skew(d)$) and by $\Pi_{\rm rot} \colon L^2(\Omega, \R^{d \times d}) \to L^2(\Omega, \R^{d \times d})$, $F \mapsto \Pi_{\rm rot} F$ the orthogonal projection whose kernel is the set of constant linearized rotations $\set{x \mapsto A}{A \in \Skew(d)}$, Korn's inequality reads 
$$ \norm{\Pi_{\rm rot} \nabla u}_{L^2(\Omega)} 
   \le C \norm{\pi_{\rm rot} \nabla u}_{L^2(\Omega)}. $$
In terms of $\Pi_{\rm iso} \colon L^2(\Omega, \R^d) \to L^2(\Omega, \R^d)$, $u \mapsto \Pi_{\rm iso} u$, the orthogonal projection whose kernel is the set of linearized isometries $\set{x \mapsto Ax + c}{A \in \Skew(d),\ c \in \R^d}$, it can also be rephrased as %\footnote{It is clear that $\norm{\Pi_{\rm rot} \nabla u}_{L^2} \le \norm{\nabla \Pi_{\rm iso} u}_{L^2}$. Conversely, if $\Pi_{\rm iso} u = u - A \fdot - c$ and $\Pi_{\rm rot} \nabla u = \nabla u - A'$, then Poincar{\'e}'s inequality gives $\norm{u - A \fdot - c}_{L^2} \le \norm{u - A' \fdot - c'}_{L^2} \le C \norm{\nabla u - A'}_{L^2}$ for some $c' \in \R^d$ and hence $\norm{(A'-A) \fdot + c'-c}_{L^2} \le C \norm{\nabla u - A'}_{L^2}$, which implies $\norm{A'-A} \le C \norm{\nabla u - A'}_{L^2}$. Thus, $\norm{\nabla u - A}_{L^2} \le \norm{A' - A}_{L^2} + \norm{\nabla u - A'}_{L^2} \le C \norm{\nabla u - A'}_{L^2}$. } 
$$ \norm{\nabla \Pi_{\rm iso} u}_{L^2(\Omega)} 
   \le C \norm{\pi_{\rm rot} \nabla u}_{L^2(\Omega)} $$ 
(see \eqref{eq:PirotD-DPiiso-equiv} below). 
In particular, on $H^1_0(\Omega)$ or $H^1_{\rm per}(\Omega)$ (in case $\Omega$ is a cuboid) one even has 
$$ \norm{\nabla u}_{L^2(\Omega)} 
   \le C \norm{\pi_{\rm rot} \nabla u}_{L^2(\Omega)}. $$  
The reverse estimates being trivial, an equivalent form is to say that the seminorms $\norm{\nabla \Pi_{\rm iso} \fdot}_{L^2(\Omega)}$ and $\norm{\pi_{\rm rot} \nabla \fdot}_{L^2(\Omega)}$, respectively, $\norm{\nabla \fdot}_{L^2(\Omega)}$ and $\norm{\pi_{\rm rot} \nabla \fdot}_{L^2(\Omega)}$ are equivalent. 

In fact, numerous generalizations of Korn's basic inequality have been established, in particular, over the last years, including settings in more general function spaces (Orlicz spaces, functions of bounded variation), estimates for incompatible fields (that cannot be written as a gradient), and nonlinear rigidity inequalities. For a comprehensive summary, we refer the reader to the recent paper \cite{LewintanMuellerNeff:21} and the references cited therein. 

In more direct connection with the subject of the present contribution are discretized versions of the continuum Korn inequality that, motivated by the analysis of numerical approximation schemes, have been obtained in various settings. 
By way of example we mention \cite{Brenner:04,MardalWinther:06,DiPietroNicaise:13,CarstensenSchedensack:15,LiMingShi:17,BottiDiPietroGuglielmana:19,LiMingWang:21}. 
More recently, discrete versions of the Korn inequality have been developed that apply to systems of interacting particles and provide rigidity estimates for crystals in terms of their configurational energy. Such estimates are at the basis of the stability analysis of lattice systems: 
If a configuration is a critical point of the configurational energy, \ie, the forces within the particle system are in balance, one is interested in criteria that guarantee that such a configuration is stable, see, \eg, \cite{E2007,Hudson2011,Ortner2013,Braun2016}. 
It bears emphasis that, in comparison to pure continuum models, such atomistic systems are considerably more delicate as not only continuum (and hence long wave length) perturbations but also possible disorder at the atomistic scale has to be taken into account. 
From a technical point of view this amounts to additional degrees of freedom in possibly high dimensional discrete gradients (cp.\ \cite{FrieseckeTheil:02,ContiDolzmannKirchheimMueller:06,Schmidt:06,Schmidt:09}) that need to be controlled in terms of energy estimates so that eventually (a suitable version of) a Cauchy-Born rule can be established.

There are two principal features that are at the core of a discrete Korn inequality for a lattice system (cf.\ \cite{Hudson2011,Braun2016}): 1.\ Periodicity: The periodic arrangement of particles allows for the application of Fourier transform methods to establish `phonon stability'; and 2.\ Exhaustion of the full space: In bulk systems there are no soft modes due to buckling type deformations.

The central aim of the present contribution is to investigate the validity of Korn type inequalities beyond the periodic setting and, to some extend, also beyond the bulk regime. It lies at the heart our endeavor to examine the stability behavior of such generalized structures, cf.\ \cite{SchmidtSteinbach:21a,SchmidtSteinbach:21c,Steinbach2021}. The main motivation for such an analysis are possible applications to {\em objective structures}. These particle systems, introduced by James in \cite{James2006}, constitute a far reaching generalization of lattice systems and have been successfully applied to a remarkable number of important structures, ranging from biology (to describe parts of viruses) to nanoscience (to model carbon nanotubes), see, \eg,  \cite{FalkJames06,DumitricaJames07,DayalJames10,FengPlucinskyJames19}. They are characterized by the fact that, up to rigid motions of the surrounding space, any two points ``see'' an identical environment of other points. (In a lattice this would be true even up to translations.) As a consequence, objective structures correspond to orbits of a single point under the action of a general discrete group of Euclidean isometries, cf.\ \cite{James2006,Juestel2014}. As the symmetry of these objects in general is considerably more complex than that of a lattice, the adaption of methods and results on lattices has only been achieved in a few cases so far. As notable examples we mention an algorithm for solving the Kohn-Sham equations for clusters \cite{BanerjeeElliottJames:15} and the X-ray analysis of helical structures set forth in \cite{FrieseckeJamesJuestel16}. 

Within an appropriate coordinate system for an objective structure, such a group might be assumed to embed into a subgroup of $\O(d_1) \oplus \mathcal{S}$ for a crystallographic spacegroup $\mathcal{S}$ acting on $\R^{d_2}$, where $d_1+d_2=d$, with surjective projection onto $\mathcal{S}$. In particular, for bulk structures with $d_2 = d$ the particles invade the whole space $\R^d$, whereas lower dimensional structures invade a tubular neighborhood of $\{0\} \times \R^{d_2}$. 

A major difficulty in obtaining Korn type inequalities then results from the general structure and the non-commutativity of these groups. Whereas in principle a Fourier transform is defined on their dual spaces, the consideration of periodic mappings with significant ``long wave-length'' contributions turns out non-trivial. Yet, uniform estimates on such quantities that are stable in the limit of infinitely large periodicity (corresponding to infinitely many particles, respectively, vanishing interparticle distances in a rescaled set-up) are essential for a discrete Korn inequality to hold. However, as objective structures need not be periodic, even the definition of quantities that can serve the role of a wave vector is not obvious. 

In \cite{SchmidtSteinbach:21a}, by exploiting the special structure of discrete subgroups of the group of Euclidean isometries on $\R^d$, we provided an efficient and extensive description of the dual space of a general discrete group of Euclidean isometries. In particular, we identified a finite union of convex `wave vector domains' reflecting the existence of an underlying part of translational type of finite index. This structure is indeed tailor-made for our investigations on Korn inequalities.  Due to the discrete nature of the underlying particle system, we consider finite difference stencils of (finite) interaction range and associate to them suitable seminorms measuring the (local) distances to the set of infinitesimal rigid motions and certain subsets thereof, respectively, in terms of $\ell^2$ norms of projections onto these sets. Our main results are then formulated in terms of such seminorms and state generic conditions for their equivalence, the main result being Theorem~\ref{Theorem:EquivalenceAll}. At the core of our proof lies the technical Lemma~\ref{Lemma:helpTheorem} in which we utilize a classical minimax theorem of Tur\'{a}n on generalized power sums in order to obtain control on a general skew symmetric matrix in terms of certain oscillatory perturbations.  

In more detail, for a given interaction range $\RR$ we consider the three seminorms $\norm \fdot_\RR$, $\norm \fdot_{\RR,0}$, $\norm \fdot_{\RR,0,0}$. Roughly speaking, $\norm \fdot_\RR$ measures the local distances from the set of all infinitesimal rigid motions, characterized by generic skew symmetric matrices $$S = {\footnotesize \parens[\bigg]{\begin{matrix}S_1&S_2\\-S_2^T&S_3\end{matrix}} } \in \Skew(d),$$ where $S_1\in\Skew(d_1),S_2\in\R^{d_1\times d_2},S_3\in\Skew(d_2)$. $\norm \fdot_{\RR,0}$ measures the local distances to those rigid motions that fix $\{0\} \times \R^{d_2}$ intrinsically, corresponding to $S \in\Skew(d)$ as above with $S_3 = 0$, and $\norm \fdot_{\RR,0,0}$ measures the local distances to those rigid motions that fix $\{0\} \times \R^{d_2}$ in $\R^{d}$, corresponding to $S \in\Skew(d)$ with $S_2=0$ and $S_3 = 0$. In particular, $\norm \fdot_\RR \le \norm \fdot_{\RR,0} \le \norm \fdot_{\RR,0,0}$. 

In Theorems~\ref{Theorem:equivalence} and~\ref{Theorem:NewEquivalenceAll} we observe that each of these seminorms does -- up to equivalence -- not depend on the particular choice of $\RR$ as long as $\RR$ is rich enough. Our main Theorem~\ref{Theorem:Korn} then states that indeed $\norm \fdot_\RR$ and $\norm \fdot_{\RR,0}$ are equivalent. In particular, for bulk structures with $d_1 = 0$ we thereby obtain a full Korn inequality for objective structures generated by a general space group. For $d_1\ge 1$ it can be interpreted as an `intrinsic rigidity' estimate within the extended dimensions of the structure. We summarize these findings in Theorem~\ref{Theorem:EquivalenceAll}. In Propositions~\ref{Proposition:ExOne} and~\ref{Proposition:ExTwo} we will also see that in general $\norm \fdot_{\RR,0}$ and $\norm \fdot_{\RR,0,0}$ are not equivalent. In view of possible buckling modes, this is in fact not to be expected and indeed long wave-length modulations of the extended dimensions within the surrounding space impede a strong Korn type inequality. 

In fact, in applications to the stability analysis of objective structures both seminorms $\norm \fdot_\RR$ (equivalently, $\norm \fdot_{\RR,0}$) and $\norm \fdot_{\RR,0,0}$ will be of relevance. There the question is addressed if an objective structure is a stable configuration when the particles at different sites are assumed to interact. Despite its importance, little appears to be known beyond bulk lattice systems. (See, \eg, \cite{Hudson2011,Braun2016} for lattice systems subject to very generic interaction potentials.) 
Indeed, stability estimates for homogeneous structures are not only of intrinsic value but may also serve as a fundamental step towards a quantitative description of the effect of a dislocation in such structures, cp.\ \cite{EhrlacherOrtnerShapeev:16,OlsonOrtner:17,OlsonOrtnerWangZhang:23,BraunHudsonOrtner:22}. 
In \cite{SchmidtSteinbach:21c} we provide a stability analysis in the general framework of objective structures and, in particular, establish characterizations of stability constants for objective structures in terms of the seminorms $\norm \fdot_{\RR}$ and $\norm \fdot_{\RR,0,0}$. Here $\norm \fdot_{\RR,0,0}$ applies to bulk systems and might also be used in lower dimensional tensile regimes in which pre-stresses have a stabilizing effect. The weaker seminorm $\norm \fdot_{\RR}$ appropriately describes lower dimensional systems in their ground state even at the onset of (buckling type) instabilities. Based on these results, we will be able to provide a numerical algorithm for determining the stability of a given structure. By way of example we will also show that indeed novel stability results for nanotubes can be obtained. 

\subsection*{Outline} 
In Section~\ref{section:structure} we discuss the kinematics of objective structures. We begin by collecting some fundamental results on discrete subgroups of the Euclidean group in Subsection~\ref{subsection:structure}
including a characterization up to conjugacy and basic notions of Fourier analysis on periodic mappings for such structures. In the following Subsection~\ref{subsection:orbit} we draw some conclusions on the geometry of objective structures which are orbits of a point under such discrete Euclidean groups. 

Section~\ref{section:seminorm} is the core section of our paper featuring our main Korn type Theorems~\ref{Theorem:Korn} and~\ref{Theorem:EquivalenceAll}. 
In Subsection~\ref{subsection:deformation-seminorms} we first define seminorms $\norm \fdot_\RR$ on deformations in terms of local finite differences with interaction range $\RR$. The following Subsection~\ref{subsection:local-equivalence} serves to prove that these seminorms are essentially independent of the particular choice of $\RR$. 
In Subsection~\ref{subsection:Korn} we then define the above-mentioned seminorms $\norm \fdot_{\RR,0}$. 
Having successfully established our main technical Lemma~\ref{Lemma:helpTheorem}, we prove our main Theorem~\ref{Theorem:Korn} stating that $\norm \fdot_\RR$ and $\norm \fdot_{\RR,0}$ are equivalent. In the last Subsection~\ref{subsection:seminorms-kernels} of the present section we explicitly describe the kernels of the previously defined seminorms. 

In Section~\ref{section:zero} we briefly discuss two naturally arising seminorms including the above-mentioned $\norm \fdot_{\RR,0,0}$ which turn out to be stronger than $\norm \fdot_\RR$ and $\norm \fdot_{\RR,0}$. 

The final Section~\ref{section:examples} discusses two basic examples which allow for amenable descriptions of the above studied seminorms, both in real and in Fourier space. They also serve as an explicit example showing that $\norm \fdot_{\RR,0,0}$ and $\norm \fdot_\RR$ are not equivalent. 

\subsection*{Notation} 
We denote by $e_i$ the $i^{\text{th}}$ standard coordinate vector in $\R^d$ and by $I_d\in\R^{d\times d}$ the identity matrix of size $d$ and by $\id$ the identity function $\R^d \to \R^d$, $x \mapsto x$. 
If $x\in\C^m$, $y\in\C^n$ we write $x\otimes y^T=xy^T=(x_iy_j)\in\C^{m\times n}$. 
$\C^{m\times n}$ is equipped with the usual Frobenius inner product $\angles{\fdot,\fdot}$ and induced norm $\norm\fdot$.
For a  group $\G$ and $\A, \A_1,\A_2\subset \G$, $g\in \G$  and $n\in\Z$ we denote by 
\[\A_1\A_2:=\set{a_1a_2}{a_1\in \A_1,a_2\in \A_2}\subset \G\quad\text{and}\quad 
g\A:=\set{ga}{a\in \A}\subset \G.
\]
the product of subsets, respectively, an element and a subset of a group, while we reserve 
\begin{align*}
\A^n:=\set{a^n}{a\in\A}\subset \G
\end{align*}
for the set of $n$-th powers of elements of $\A$. 
Finally, $\angles \A$ is the subgroup generated by $\A$. 

%-------------------------------------------------------------------
\subsection*{Acknowledgements}
This work was partially supported by project 285722765 of the Deutsche
Forschungsgemeinschaft (DFG, German Research Foundation).

%%-------------------------------------------------------------------
%\subsection*{Data availability statement}
%%
%Data sharing not applicable to this article as no datasets were generated or analyzed during the current study.

%-------------------------------------------------------------------
%-------------------------------------------------------------------
\section{Objective structures}\label{section:structure}
%-------------------------------------------------------------------

Objective structures are orbits of a point under the action of a discrete subgroup of the Euclidean group. For an efficient description, in Subsection~\ref{subsection:structure} we first describe the structure of these groups in some detail. We then present a number of basic results on the Fourier analysis of such groups. In Subsection~\ref{subsection:orbit} we introduce the atomic reference configurations and study their geometry in the ambient space. 

%-------------------------------------------------------------------
%-------------------------------------------------------------------
\subsection{Discrete subgroups of the Euclidean group}\label{subsection:structure}
%-------------------------------------------------------------------

We collect some basic material on discrete subgroups of the Euclidean group acting on $\R^d$ from \cite{SchmidtSteinbach:21a}.
For proofs of the results in this subsection we refer to \cite{SchmidtSteinbach:21a}.

The \emph{Euclidean group} $\E(d)$ in dimension $d\in\N$ is the set of all Euclidean distance preserving transformations of $\R^d$ into itself, their elements are called \emph{Euclidean isometries}. It may be described as $\E(d)=\O(d)\ltimes \R^d$, the (outer semidirect) product of $\R^d$ and the orthogonal group $\O(d)$ in dimension $d$ with group operation given by
\[\iso{A_1}{b_1}\iso{A_2}{b_2}=\iso{A_1A_2}{b_1+A_1b_2}\]
for $\iso{A_1}{b_1},\iso{A_1}{b_2}\in\E(d)$. We set 
\begin{align*}
\rot\colon\E(d)\to\O(d), \quad 
&\iso Ab\mapsto A \qquad \text{and} \\  
\trans\colon\E(d)\to\R^d, \quad 
&\iso Ab\mapsto b
\end{align*}
and for $\iso Ab\in\E(d)$ we call $\rot(\iso Ab)$ the \emph{linear component} and $\trans(\iso Ab)$ the \emph{translation component} of $\iso Ab$ so that  
\[g=\iso{I_d}{\trans(g)}\iso{\rot(g)}0\]
for each $g\in\E(d)$. 
An Euclidean isometry $\iso Ab$ is called a \emph{translation} if $A=I_d$. The set $\Trans(d):=\{I_d\}\ltimes\R^d$ of translations forms an abelian subgroup of $\E(d)$. 
$\E(d)$ acts on $\R^d$ via 
\[\iso Ab\gdot x := Ax+b\qquad\text{for all }\iso Ab\in \E(d)\text{ and }x\in\R^d.\]
%
%In this contribution we use a calligraphic font for subsets and particularly for subgroups of $\E(d)$.
For a group $\G<\E(d)$ the \emph{orbit} of a point $x\in\R^d$ under the action of the group is 
\[\G\gdot x:=\set{g\gdot x}{g\in\G}\]
and the \emph{stabilizer subgroup} of $\G$ with respect to $x\in\R^d$ is
\[\G_x:=\set{g\in\G}{g\gdot x=x}.\]
In the following we will consider \emph{discrete subgroups} of the Euclidean group, which are those $\G<\E(d)$ for which every orbit $\G\gdot x$, $x\in\R^d$, is discrete.

Particular examples of discrete subgroups of $\E(d)$ are the so-called \emph{space groups}. These are those discrete groups $\G < \E(d)$ that contain $d$ translations whose translation components form a basis of $\R^d$. Their subgroup of translations is generated by $d$ such linearly independent translations and forms a normal subgroup of $\G$ which is isomorphic to $\Z^d$. 

In general, discrete subgroups of $\E(d)$ can be characterized as follows. (Also cp.\ \cite[A.4 Theorem 2]{Brown1978}.) Recall that two subgroups $\G_1,\G_2<\E(d)$ are \emph{conjugate} in $\E(d)$ if there exists some $g\in\E(d)$ such that $g^{-1}\G_1 g=\G_2$. (This corresponds to a rigid coordinate transformation in $\R^d$.) 
\begin{Theorem}\label{Theorem:Browndecomposablediscretegroup}
Let $\G<\E(d)$ be discrete, $d\in\N$. There exist $d_1,d_2\in\N_0$ such that $d=d_1+d_2$, a $d_2$-dimensional space group $\SG$ and a discrete group $\G'<\gplus{\O(d_1)}{\SG}$ such that $\G$ is conjugate under $\E(d)$ to $\G'$ and $\pi(\G')=\SG$, where $\pi$ is the natural epimorphism $\gplus{\O(d_1)}{\E(d_2)}\to\E(d_2)$, $\gplus Ag\mapsto g$.
\end{Theorem}
Here $\gplus{}{}$ is the group homomorphism
\begin{align*}
\gplus{}{}\colon&\O(d_1)\times\E(d_2)\to\E(d_1+d_2)\\
&(A_1,\iso{A_2}{b_2})\mapsto \gplus{A_1}{\iso{A_2}{b_2}}:=\iso[\bigg]{\parens[\bigg]{\begin{matrix}A_1&0\\0&A_2\end{matrix}}}{\parens[\bigg]{\begin{matrix}0\\b_2\end{matrix}}} 
\end{align*}
and $\gplus{\O(d_1)}{\SG}$ is understood to be $\O(d)$ if $d_1=d$ and to be $\SG$ if $d_1=0$. The theorem allows us to assume that $\G$ from now on is of the form $\G'$ with no loss of generality.  

Such a discrete group $\G<\E(d)$ can be efficiently described in terms of the range $\SG$, the kernel $\F$ of $\pi|_\G$ and a section $\T \subset \G$ of the translation group $\T_\SG$ of $\SG$, \ie, a set $\T\subset\G$ such that the map $\T\to\T_\SG$, $g\mapsto \pi(g)$ is bijective. We remark that the quantities $d$, $d_1$, $d_2$, $\F$, $\SG$ and $\T_\SG$ are uniquely defined by $\G$. However, in general there is no canonical choice for $\T$, it might not a be group and the elements of $\T$ might not commute. Yet, a main result of \cite{SchmidtSteinbach:21a} states that there is an $m_0\in\N$ such that $\T^N = \set{t^N}{t\in\T}$ is a normal subgroup of $\G$ if and only if $N$ is a multiple of $m_0$: 
\[ \T^N \triangleleft \G \iff N \in \MM := m_0 \N. \] 
For each $N \in \MM$, $\T^N$ is isomorphic to $\Z^{d_2}$ and of finite index in $\G$. In this sense, $\mathcal{G}$ is a finite extension of the lattice $\mathcal{T}^{m_0} \cong \Z^{d_2}$. 

This observation allows us to introduce a notion of periodicity for functions defined on $\G$ as those functions which are invariant under $\T^N$ for some multiple $N$ of $m_0$. More precisely, for a set $S$ and $N\in \MM$ we say that a function $u\colon \G \to S$ is \emph{$\T^N$-periodic} if
\[u(g) = u(gt)\qquad\text{for all }g\in\G\text{ and }t\in\T^N.\]
It is called \emph{periodic} if there exists some $N\in \MM$ such that $u$ is $\T^N$-periodic. We also set
\[\Per(\G,\C^{m\times n}):=\set{u\colon\G\to\C^{m\times n}}{u\text{ is periodic}}.\]
(Recall that $\C^{m\times n}$ is equipped with the usual Frobenius inner product and induced norm.) We notice that the above definition of periodicity is independent of the choice of $\T$ and that $\Per(\G,\C^{m\times n})$ is a vector space. In fact, one has 
\[\Per(\G,\C^{m\times n}) = \set[\Big]{\G\to\C^{m\times n},g\mapsto u(g\T^N)}{N\in \MM, u\colon\G/\T^N\to\C^{m\times n}}.\]
For each $N\in \MM$ we now fix a representation set $\CC_N$ of $\G/\T^N$ and we equip $\Per(\G,\C^{m\times n})$ with the inner product $\angles{\fdot,\fdot}$ given by
\[\angles{u,v}:=\frac1{\abs{\CC_N}}\sum_{g\in\CC_N}\angles{u(g),v(g)}\qquad\text{if $u$ and $v$ are $\T^N$-periodic}\]
for all $u,v\in\Per(\G,\C^{m\times n})$.
The induced norm is denoted by $\norm\fdot_2$.

We denote by $\dual{\T^{m_0}}$ the \emph{dual space} of the abelian group $\T^{m_0}$, which consists of all homomorphisms from $\T^{m_0}$ to the complex unit circle. 
Observe that a homomorphism $\chi\in\dual{\T^{m_0}}$ is $\T^N$-periodic, $N\in \MM$, if and only if $\chi|_{\T^N}=1$. Let $\EE$ be the set $\set{\chi\in\dual{\T^{m_0}}}{\chi\text{ is periodic}}$. 
\begin{Remark}\label{rmk:dualspacelattice}
Suppose $\T^{m_0} \cong \Z^{d_2}$ is generated by $\{t_1,\ldots,t_{d_2}\}$. Then we have $\dual{\T^{m_0}}=\set{\chi_k}{k\in[0,1)^{d_2}}$, where $\chi_k\colon\T^{m_0}\to\C$ is given by 
\[\chi_k(t_1^{n_1}\cdots t_{d_2}^{n_{d_2}})=\euler^{2\pi\iu \angles{n,k}}\] 
for all $n\in\Z^{d_2}$. (Here $k_j$ is determined by the condition $\chi(t_j)=\euler^{2\pi\iu k_j}$, $j=1,\ldots,d_2$.) Such $\chi_k$ is periodic if and only if $k \in \Q^{d_2}$, whence $\EE = \set{\chi_k}{k\in[0,1)^{d_2} \cap \Q^{d_2}}$. 
\end{Remark}
Note that $\T^{m_0}\cap\CC_N$ is a representation set of $\T^{m_0}/\T^N$ for all $N\in\MM$. We define the Fourier transform as follows. 
\begin{Definition}\label{Definition:FourierPeriodic}
If $u\in\Per(\T^{m_0},\C^{m\times n})$ and $\chi\in\EE$, we set
\[\fourier u(\chi):=\frac1{\abs{\T^{m_0}\cap\CC_N}}\sum_{g\in\T^{m_0}\cap\CC_N}\chi(g)u(g)\in\C^{m\times n},\]
where $N\in\MM$ is such that $u$ and $\chi$ are $\T^N$-periodic.
\end{Definition}
\begin{Proposition}[The Plancherel formula]\label{Proposition:TFplancherelmatrix} 
The Fourier transformation
\[\fourier\fdot \colon \Per(\T^{m_0}, \C^{m\times n}) \to\bigoplus_{\chi\in\EE} \C^{m\times n}, \quad u \mapsto (\fourier u(\chi))_{\chi\in\EE}\]
is well-defined and bijective. Moreover, the Plancherel formula
\[\angles{u,v}=\sum_{\chi\in\EE}\angles{\fourier u(\chi),\fourier v(\chi)}\qquad\text{for all }u,v\in\Per(\T^{m_0},\C^{m\times n})\]
holds true. 
\end{Proposition}
We remark that for all $u\colon\T^{m_0}\to\C^{m\times n}$ and $N\in \MM$ such that $u$ is $\T^N$-periodic, one gets 
\[\set{\chi \in \EE}{\fourier u(\chi) \neq 0} \subset \set{\chi\in\EE}{\chi\text{ is $\T^N$-periodic}}.\]
The following lemma provides the Fourier transform of a translated function.
\begin{Lemma}\label{Lemma:PeriodicTranslation}
Let $f \in \Per(\T^{m_0},\C^{m\times n})$, $g\in\G$ and $\tau_g f$ denote the translated function $f(\fdot g)$.
Then we have $\tau_g f\in\Per(\T^{m_0},\C^{m\times n})$ and
\[\fourier{\tau_g f}(\chi)=\chi(g^{-1})\fourier f(\chi)\]
for all $\chi\in\EE$.
\end{Lemma}
%

%-------------------------------------------------------------------
%-------------------------------------------------------------------
\subsection{Orbits of discrete subgroups of the Euclidean group}\label{subsection:orbit}
%-------------------------------------------------------------------
%
%
As a far reaching generalization of a lattice, James \cite{James2006} defines an \emph{objective (atomic) structure} as a discrete point set $S$ in $\R^d$ such that for any $x_1, x_2 \in S$ there is an Euclidean isometry $g \in \E(d)$ with $g\cdot S = S$ and $g\cdot x_1 = x_2$. Equivalently, $S$ is an orbit of a point under the action of a discrete subgroup of $\E(d)$, see, \eg, \cite[Proposition 3.14]{Juestel2014}:
\begin{Definition}%
A subset $S$ of $\R^d$ is called an \emph{objective structure} if there exist a discrete group $\G<\E(d)$ and a point $x\in\R^d$ such that $S = \G\cdot x$.%
\end{Definition}
For a wealth of examples, we refer to the original contribution \cite{James2006}. Here we limit ourselves to two simple concrete examples that will serve to illustrate the results to be discussed below. 
\begin{Example}\label{ex:elementary-chains-i}
Elementary illustrative examples are given by atomic chains such as 
\begin{enumerate}
\item\label{ex:item-straight} $\G_1=\angles{t_1}<\E(2)$, where $t_1=\iso{I_2}{e_2}\in\E(2)$, and with $x_{1,0}=0\in\R^2$,  

\item $\G_2=\angles{t_2}<\E(2)$, where $t_2=\iso{\parens{\begin{smallmatrix}-1&0\\0&1\end{smallmatrix}}}{e_2}\in\E(2)$, and with $x_{2,0}=e_1\in\R^2$, 
\end{enumerate}
cf.\ Fig.~\ref{fig:atomic-chains}. Here we have $d_1=d_2=1$ in both cases.
\begin{figure}[h!]
\begin{center}
\begin{tikzpicture}[scale=0.5]
 \draw[gray,dotted] (-11.5,0) -- (-0.5,0);
 \draw[gray,fill=gray] (-11,0) circle (.1cm);
 \draw[gray,fill=gray] (-10,0) circle (.1cm);
 \draw[gray,fill=gray] (-9,0) circle (.1cm);
 \draw[gray,fill=gray] (-8,0) circle (.1cm);
 \draw[gray,fill=gray] (-7,0) circle (.1cm);
 \draw[gray,fill=gray] (-6,0) circle (.1cm);
 \draw[gray,fill=gray] (-5,0) circle (.1cm);
 \draw[gray,fill=gray] (-4,0) circle (.1cm);
 \draw[gray,fill=gray] (-3,0) circle (.1cm);
 \draw[gray,fill=gray] (-2,0) circle (.1cm);
 \draw[gray,fill=gray] (-1,0) circle (.1cm);
 \draw[gray,fill=gray] (3,1) circle (.1cm);
 \draw[gray,fill=gray] (4,-1) circle (.1cm);
 \draw[gray,fill=gray] (5,1) circle (.1cm);
 \draw[gray,fill=gray] (6,-1) circle (.1cm);
 \draw[gray,fill=gray] (7,1) circle (.1cm);
 \draw[gray,fill=gray] (8,-1) circle (.1cm);
 \draw[gray,fill=gray] (9,1) circle (.1cm);
 \draw[gray,fill=gray] (10,-1) circle (.1cm);
 \draw[gray,fill=gray] (11,1) circle (.1cm);
 \draw[gray,fill=gray] (12,-1) circle (.1cm);
 \draw[gray,fill=gray] (13,1) circle (.1cm);
 \draw[gray,dotted] (2.75,0.5) -- (3,1) -- (4,-1) -- (5,1) -- (6,-1) -- (7,1) -- (8,-1) -- (9,1) -- (10,-1) -- (11,1) -- (12,-1) -- (13,1) -- (13.25,0.25) ; 
 \draw[thick,->] (16,0) -- (16,-1.5); 
 \node[left] at (16,-1.5) {$x_1$};
 \draw[thick,->] (16,0) -- (17.5,0); 
 \node[below] at (17.5,0) {$x_2$};
\end{tikzpicture}
\end{center} 
\caption{$\G_1\cdot x_{1,0}$ (left) and $\G_2\cdot x_{2,0}$ (right)}\label{fig:atomic-chains}
\end{figure}
\end{Example}
We proceed with a couple of lemmas implying that without loss of generality objective structures lie in $\{0_{d-\daff}\}\times\R^\daff$ where $\daff$ is their affine dimension and, moreover, the associated discrete group of isometries acts trivially on $\R^{d-\daff}\times\{0_{\daff}\}$. 
\begin{Lemma}%\label{Lemma:config}
Let $S\subset\R^d$ be an objective structure.
Then for every $a\in\E(d)$ the set $\set{a\gdot x}{x\in S}$ is also an objective structure.
\end{Lemma}
\begin{proof}
This follows directly from the observation that, if for a subgroup $\G<\E(d)$ and $x_0\in\R^d$ the map $\G\to S$, $g\mapsto g\gdot x_0$ is surjective, then, for every $a\in\E(d)$ the map $a\G a^{-1}\to\set{a\gdot x}{x\in S}$, $g\mapsto g\gdot(a\gdot x_0)$ is surjective.
\end{proof}
We denote the \emph{affine hull} of a set $A\subset\R^d$ by $\aff(A)$ and write $\dim(A):=\dim(\aff (A))$ for its \emph{affine dimension}. Recall that this is the dimension of the vector space $\spano(\set{x-x_0}{x\in A})$ for any $x_0\in A$.   
\begin{Lemma}\label{Lemma:configsubsetNEW}
Let $\G<\E(d)$ be discrete and $x_0\in\R^d$.
Let $\daff=\dim(\G\cdot x_0)$.
Then there exists some $a\in\E(d)$ such that for the discrete group $\G'=a\G a^{-1}$ and $x_0'=a\cdot x_0$ it holds
\[\aff(\G'\cdot x_0')=\{0_{d-\daff}\}\times\R^\daff\]
and $\G'\cdot x_0'=a\cdot(\G\cdot x_0)$.
\end{Lemma}
\begin{proof}
There exists some $\daff$-dimensional vector space $V$ such that $\aff(\G\cdot x_0)=x_0+V$. 
Choosing $A\in\O(d)$ such that $\set{Ax}{x\in V}=\{0_{d-\daff}\}\times\R^\daff$ and setting $a=\iso A{-Ax_0}\in\E(d)$ implies the assertion.
\end{proof}
%
%Recall that $\rot\colon\E(d)\to\O(d)$, $\iso Ab\mapsto A$.
%

Note that in Example~\ref{ex:elementary-chains-i} $\G_1 \cdot x_{1,0}$ has $\daff=1$ and $\G_2 \cdot x_{2,0}$ has $\daff=2$. 
\begin{Lemma}\label{Lemma:Shelp2}
Let $\G<\E(d)$ be discrete and $x_0\in\R^d$ such that $\aff(\G\cdot x_0)=\{0_{d-\daff}\}\times\R^\daff$, where $\daff=\dim(\G\cdot x_0)$.
Then we have $\G<{\gplus{\O(d-\daff)}{\E(\daff)}}$.
\end{Lemma}
\begin{proof}
Set $V=\{0_{d-\daff}\}\times\R^\daff=\aff(\G\cdot x_0)$. For given $g\in\G$ we define the map $\phi\colon\R^d\to\R^d$, $x\mapsto\rot(g)x$. First we show that $V$ is invariant under $\phi$.
Let $x\in V$.
Since $V=\aff(\G\cdot x_0)-x_0$, there exist some $n\in\N$, $x_1,\dots,x_n\in\G\cdot x_0$ and $\alpha_1,\dots,\alpha_n\in\R$ such that $x=\sum_{i=1}^n\alpha_ix_i$ and $\sum_{i=1}^n\alpha_i=0$.
It holds
\[\rot(g)x=\sum_{i=1}^n\alpha_i\rot(g)x_i=\sum_{i=1}^n\alpha_i (g\cdot x_i)\in V.\]
Thus we have $\set{\rot(g)\tilde x}{\tilde x\in V}\subset V$.
Since $\rot(g)$ is invertible, it holds $\set{\rot(g)\tilde x}{\tilde x\in V}=V$.

Since $\rot(g)$ is orthogonal, also the complement $V^\bot=\R^{d-\daff}\times\{0_\daff\}$ is invariant under $\phi$.
This implies $\rot(g)\in\O(d-\daff)\oplus\O(\daff)$.
It holds $\trans(g)=g\cdot x_0-\rot(g)x_0\in V$ and thus, $g\in\gplus{\O(d-\daff)}{\E(\daff)}$.
\end{proof}
\begin{Lemma}%\label{Lemma:WLOGdOStwo}
Let $\G<\E(d)$ be discrete and $x_0\in\R^d$ such that $\aff(\G\cdot x_0)=\{0_{d-\daff}\}\times\R^\daff$, where $\daff=\dim(\G\cdot x_0)$.
Let $\G'=\set{\gplus{I_{d-\daff}}g}{g\in\E(\daff),\exists A\in\O(d-\daff): \gplus Ag\in\G}$ and
\begin{align*}
\phi\colon&\G\to\G'\\
&\gplus Ag\mapsto\gplus{I_{d-\daff}}g\qquad\text{if }A\in\O(d-\daff),g\in\E(\daff)\text{ and }\gplus Ag\in\G.
\end{align*}
Then $\G'$ is a discrete subgroup of $\E(d)$, $\phi$ is an epimorphism and $\G\cdot x_0=\G'\cdot x_0$.
\end{Lemma}
\begin{proof} 
By Lemma~\ref{Lemma:Shelp2} we have $\G<\gplus{\O(d-\daff)}{\E(\daff)}$. It is clear that $\phi$ is an epimorphism. 
If $x=x_1+x_2$ with $x_1\in\R^{d-\daff}\times\{0_{\daff}\}$ and $x_2\in\{0_{d-\daff}\}\times\R^\daff$, then it holds $\phi(g)\cdot x=x_1+g\cdot x_2$ for all $g\in\G$ and thus $\G'\cdot x=x_1+\G\cdot x_2$. This shows that $\G'$ is discrete and, for $x=x_0$ in particular, that $\G'\cdot x_0=\G\cdot x_0$.
\end{proof}
\begin{Remark}
\begin{enumerate}
\item Let $\G<\E(d)$ be discrete, $x_0\in\R^d$ and $A=\aff(\G\cdot x_0)$.
For all $g\in\G$ it holds $\set{g\cdot x}{x\in A}=A$.
%
%\begin{proof}
%%
%Let $\G<\E(d)$ be discrete and $x_0\in\R^d$ and $A=\aff(\G\cdot x_0)$.
%For all $g\in\G$ it holds
%%
%\begin{align*}
%\set{g\cdot x}{x\in A}&=\set[\bigg]{g\cdot\parens[\bigg]{\sum_{i=1}^n\alpha_ix_i}}{n\in\N,x_i\in \G\cdot x_0,\alpha_i\in\R,\sum_{i=1}^n\alpha_i=1}\\
%&=\set[\bigg]{\sum_{i=1}^n\alpha_i(g\cdot x_i)}{n\in\N,x_i\in \G\cdot x_0,\alpha_i\in\R,\sum_{i=1}^n\alpha_i=1}\\
%&=\set[\bigg]{\sum_{i=1}^n\alpha_i x_i}{n\in\N,x_i\in \G\cdot x_0,\alpha_i\in\R,\sum_{i=1}^n\alpha_i=1}\\
%&=A.\qedhere
%\end{align*}
%%
%\end{proof}
%
\item Let $\G<\E(d)$ be discrete and $x_0\in\R^d$.
Let $V$ be the vector space such that $\aff(\G\cdot x_0)=x_0+V$.
Then for all $g\in\G$ it holds $\set{\rot(g)x}{x\in V}=V$.
%
%\begin{proof}
%%
%Let $\G<\E(d)$ be discrete and $x_0\in\R^d$.
%Let $V$ be the vector space such that $\aff(\G\cdot x_0)=x_0+V$.
%Let $g\in\G$ and $x\in V$.
%There exist some $n\in\N$, $x_i\in\G\cdot x_0$ and $\alpha_i\in\R$ such that $x=\sum_{i=1}^n\alpha_ix_i$ and $\sum_{i=1}^n\alpha_i=0$.
%It holds
%%
%\[\rot(g)x=\sum_{i=1}^n\alpha_i\rot(g)x_i=\sum_{i=1}^n\alpha_i(g\cdot x_i)\in V.\]
%%
%Thus we have $\set{\rot(g)\tilde x}{\tilde x\in V}\subset V$.
%Since $\rot(g)$ is invertible, it holds $\set{\rot(g)\tilde x}{\tilde x\in V}=V$.
%%
%\end{proof}
%
\end{enumerate}
\end{Remark}
We close this section with some general remarks on the representation of objective structures.
\begin{Remark}
\begin{enumerate}
\item The representation of an objective structure by a discrete subgroup of $\E(d)$ and a point in $\R^d$ is not unique. Indeed, let $S=\{\pm e_1,\pm e_2\}\subset\R^2$. Denote by $R$ be the rotation matrix by the angle $\pi/2$ and by $P$ the reflection with $Pe_1=e_2$ and $Pe_2=e_1$. The cyclic group $\G_1=\angles[\big]{\iso{R}0}<\E(2)$ and the Klein four-group $\G_2=\angles[\big]{\iso{P}0,\iso{-P}0}<\E(2)$ are not isomorphic. Yet, $S = \G_1\cdot e_1 = \G_2\cdot e_1$. And both maps $\G_1\to S$, $g\mapsto g\gdot x$ and $\G_2\to S$, $g\mapsto g\gdot x$ are even bijective.
\item Generically, an objective structure $S$ can be faithfully represented as the orbit of a point $x \in \R^d$ under the action of a discrete subgroup of $\G$ of $\E(d)$, \ie, such that $\G\to S$, $g\mapsto g\gdot x$ is bijective, see the following point. However, there are counterexamples: 

Let be given a regular icosahedron centered at the origin.
Let $S$ be the set of the 30 centers of the edges of the icosahedron (\ie $S$ is the set of the vertices of the rectified icosahedron and moreover, $S$ is the set of the vertices of an icosidodecahedron).
The rotation group $\mathcal I<\SO(3)$ of the icosahedron has order 60, see, \eg, \cite[Section 2.4]{Grove1985} and we have $S=(\mathcal I\times\{0_3\})\gdot x_0$ for every point $x_0\in S$.
Now we suppose that there exist a discrete group $\G<\E(3)$ and a point $x\in\R^3$ such that the map $\G\to S$, $g\mapsto g\gdot x$ is injective.
Then we have $\abs\G=\abs{S}=30$.
Moreover, the group $\G$ is isomorphic to a finite subgroup of $\O(3)$, see, \eg, \cite[Section 4.12]{Opechowski1986}.
The finite subgroups of $\O(3)$ are classified, see, \eg, \cite[Theorem 2.5.2]{Grove1985}, and since every discrete subgroup of $\O(3)$ of order 30 contains an element of order 15, the group $\G$ contains an element $g$ of order 15.
Since the order of $g$ is odd, we have $\rot(g)\in\SO(3)$, \ie $g$ is a rotation.
Thus, the set $S$ contains 15 points which lie in the same plane.
This implies that $S$ cannot be the orbit of $\G$, and we have a contradiction.
%
%https://en.wikipedia.org/wiki/Point_groups_in_three_dimensions
%
\item For each discrete group $\G<\E(d)$, a.e.\ $x\in\R^d$ is such that the map $\G\to\R^d$, $g\mapsto g\gdot x$ is injective. Indeed, if $g, h \in \G$, $g\ne h$, then the affine space $\set{x\in\R^d}{g \cdot x = h \cdot x}$ has codimension at least $1$. Since $\G$ is at most countable, the claim follows. 
\item For each discrete group $\G<\E(d)$ and all $x\in\R^d$ the stabilizer group $\G_x = \set{g\in\G}{g\cdot x=x}$ is finite. 
To see this, one may use the previous point to choose $x'\in\R^d$ with $\norm{x-x'}<1$ such that $\G\to\R^d$, $g\mapsto g\gdot x'$ is injective. Then the discrete set $\G_x\cdot x'$ lies in the ball of radius $1$ centered at $x$ so that $\G_x\cdot x'$ and hence $\G_x$ is finite. 
\end{enumerate}
\end{Remark} 
%

%-------------------------------------------------------------------
%-------------------------------------------------------------------
\section{A discrete Korn type inequality}\label{section:seminorm}
%-------------------------------------------------------------------
%

This is the core section of our contribution. In particular, for a given interaction range $\RR$ we introduce the two seminorms $\norm \fdot_\RR$ and $\norm \fdot_{\RR,0}$. They measure the distance of a deformation of an objective structure to the set of (infinitesimally) rigid motions locally and, respectively, (intrinsically) globally. Our main result will be that -- under suitable conditions -- $\norm \fdot_\RR$ and $\norm \fdot_{\RR,0}$ are equivalent. 

We begin by introducing the seminorms $\norm \fdot_\RR$ in Subsection~\ref{subsection:deformation-seminorms} and show in Subsection~\ref{subsection:local-equivalence} that they are essentially independent of the particular choice of $\RR$. In Subsection~\ref{subsection:Korn} we first define the seminorms $\norm \fdot_{\RR,0}$. We then provide the main preparatory technical step by proving the analytical Lemma~\ref{Lemma:helpTheorem} and finally establish our main Theorem~\ref{Theorem:Korn}. We close this section by explicitly computing the kernel of the relevant seminorms in Subsection~\ref{subsection:seminorms-kernels}.

More precisely, given a finite interaction range $\RR$, one considers finite patches of a configuration by restricting to suitable neighborhoods of particles and averages the deviations from the set of rigid body motions (or a subclass thereof) over all such patches. The first seminorm $\norm \fdot_\RR$ is local in the sense that the full set of rigid motions is considered and so different finite patches can be close to completely different rigid motions, see Definition~\ref{Definition:UtransUrot}. The second seminorm $\norm \fdot_{\RR,0}$ is `intrinsically global' as the set of rigid motions is restricted to those that vanish when both preimage and target space are projected to the subspace that is invaded by the objective structure, see Definition~\ref{Definition:Seminorm-Null} for a precise statement. (For bulk structures defined in terms of a space group this is the whole space and the kernel of the resulting seminorm consists of translations only.) 

Our main result is Theorem~\ref{Theorem:Korn} (see also Theorem~\ref{Theorem:EquivalenceAll}) which states that these two seminorms are equivalent as long as the interaction range is sufficiently rich. We thus establish a Korn-type estimate for objective structures. For bulk structures we indeed obtain a full discrete Korn inequality. For lower dimensional structures this is in fact not to be expected as the structure might show buckling exploring the ambient space. Still, Theorem~\ref{Theorem:Korn} shows that intrinsically also such structures are rigid.

%-------------------------------------------------------------------
%-------------------------------------------------------------------
\subsection{Deformations and local rigidity seminorms}\label{subsection:deformation-seminorms}
%-------------------------------------------------------------------
%
Let $\G<\E(d)$ be a discrete group of Euclidean isometries and $x_0\in\R^d$. The set $\G\cdot x_0$ is an objective structure and $\G_{x_0}$ is the stabilizer subgroup. Recalling the discussion directly after Theorem~\ref{Theorem:Browndecomposablediscretegroup}, without loss of generality we assume in the following that $\G<\gplus{\O(d_1)}{\SG}$, $d=d_1+d_2$, that $\T \subset \G$ and $\CC_N$ (for $N \in M_0$) have been chosen and that $\G\cdot x_0\subset\{0_{d-\daff}\}\times\R^\daff$ and $\G$ acts trivially on $\R^{d-\daff}\times\{0_{\daff}\}$, $\daff=\dim(\G\cdot x_0)$.

We consider deformation mappings $y\colon\G\cdot x_0\to\R^d$. One can describe such a mapping by the induced `deformation' $v\colon\G/\G_{x_0}\to\R^d$ on left cosets which is given by $v(g)=y(g\cdot x_0)$.
In order to describe the action of a deformation at $g\cdot x_0$ in relation to its position within the whole structure $\G\cdot x_0$ in its environment (cf.\ \eqref{eq:DD} below), it turns out useful, see, \eg, Remark~\ref{Remark:UsefulLeftTranslationInvariantFourierPossible}\ref{item:UsefulLeftTranslationInvariantFourierPossible}, to define an associated `group displacement mapping' $u\colon\G\to\R^d$ such that
\[v(g)=\frac1{\abs{\G_{x_0}}}\sum_{g'\in g}g'\cdot(x_0+u(g'))\qquad\text{for all }g\in\G/\G_{x_0},\]
\eg by choosing $u(g')=\rot(g')^T(v(g'\G_{x_0})-g'\cdot x_0)$. More generally, for any mapping $u\colon\RR\to\R^d$ on $\RR=\RR\G_{x_0}\subset\G$ we define the averaged mapping $\proj{u}\colon\RR\to\R^d$ by 
\begin{align}\label{eq:pi-proj}
\proj{u}(g')=\frac1{\abs{\G_{x_0}}}\sum_{h\in \G_{x_0}}\rot(h)u(g'h)\qquad\text{for all }g'\in\RR.
\end{align}
So $\rot(g')\proj{u}(g')=v(g'\G_{x_0})-g'\cdot x_0$ only depends on $g=g'\G_{x_0}\in\RR/\G_{x_0}$ and we may write this expression as $\rot(g)\proj{u}(g)$ with no ambiguity. 
In particular, $v$ is the translation $v(g)=g\cdot x_0+a$ for all $g\in\G/\G_{x_0}$ and an $a\in\R^d$ if and only if $\rot(g)\proj{u}(g)=a$ for all $g\in\G$ and $v$ is the rotation $v(g)=R(g\cdot x_0)$ for all $g\in\G/\G_{x_0}$ and an $R\in\SO(d)$ if and only if $\rot(g)\proj{u}(g)=(R-I_d)(g\gdot x_0)$ for all $g\in\G$. In case $\G_{x_0}=\{\id\}$ we simply have $\proj{u}=u$.

As $\G\cdot x_0$ is typically infinite and we want to allow for deformations of long wave-length, we consider deformations $v$ corresponding to a periodic displacement $u$. A crucial point in the following is then to provide estimates that do not depend on the characteristics of the periodicity. 

Let $\RR$ be a finite subset of $\G$ such that $\RR\G_{x_0}=\RR$. 
Suppose $u\colon\G\to\R^d$ is $\T^N$-periodic for some $N\in\MM$. A natural quantity to measure the size of the associated deformation $v$ locally `modulo isometries' is  
\begin{equation}\label{eq:CC}
\parens[\bigg]{\frac1{\abs{\CC_N}}\sum_{g\in\CC_N}\dist^2\parens[\Big]{(v(h))_{h\in g\RR/\G_{x_0}},\set[\Big]{(a\gdot(h\gdot x_0))_{h\in g\RR/\G_{x_0}}}{a\in\E(d)}}}^{\frac12},
\end{equation}
where $\dist$ is the induced metric of the Euclidean norm on $(\R^d)^{g\RR/\G_{x_0}}$.
With the aim to consider small displacements $u\approx0$, for every $g\in\CC_N$ we linearize by observing that, for $U\subset\E(d)$ a sufficiently small open neighborhood of $\id$, the set
\[\set[\Big]{(a\gdot(h\gdot x_0))_{h\in g\RR/\G_{x_0}}}{a\in U}\]
is a manifold whose tangent space at the point $(h\cdot x_0)_{h\in g\RR/\G_{x_0}}$ is
\[\Viso{g\RR}=\set[\Big]{\parens[\big]{b+S(h\gdot x_0)}_{h\in g\RR/\G_{x_0}}}{b\in\R^d,S\in\Skew(d)}.\]
(This follows from the fact that the tangent space of $\E(d)$ at $\id$ is given by $\Skew(d) \times \R^d$.) 
A Taylor expansion shows that, in terms of $\proj{u}$ as defined in \eqref{eq:pi-proj}, 
\begin{align}
&\dist\parens[\Big]{(v(h))_{h\in g\RR/\G_{x_0}},\set[\Big]{(a\gdot(h\gdot x_0))_{h\in g\RR/\G_{x_0}}}{a\in\E(d)}}\nonumber\\
&\quad\approx\dist\parens[\Big]{(\rot(h)\proj{u}(h))_{h\in g\RR/\G_{x_0}},\Viso{g\RR}}\nonumber\\
&\quad=\dist\parens[\Big]{(\rot(h)\proj{u}(gh))_{h\in \RR/\G_{x_0}},\Viso{\RR}},\label{eq:DDDDD}
\end{align}
where in the second step we have used that $b+S(h\cdot x_0)=\rot(g)(\tilde b+\tilde S(\tilde h\cdot x_0))$ for $\tilde{b}=\rot(g)^T(b+S\tau(g))$, $\tilde S=\rot(g)^TS\rot(g)$ and $\tilde h=g^{-1}h$.
Similar to $\Viso\RR$ we define
\[\Uiso\RR=\set[\Big]{u\colon\RR\to\R^d}{\exists\+b\in\R^d\;\exists\+ S\in\Skew(d)\;\forall g\in\RR: \rot(g)\proj{u}(g)=b+S(g\gdot x_0)}\]
and with \eqref{eq:DDDDD} it follows that
\begin{align}
&\dist\parens[\Big]{(v(h))_{h\in g\RR/\G_{x_0}},\set[\Big]{(a\gdot(h\gdot x_0))_{h\in g\RR/\G_{x_0}}}{a\in\E(d)}}\nonumber\\
&\quad\approx\min\set[\bigg]{\sum_{h\in\RR/\G_{x_0}}\norm{\rot(h)\proj{u}(gh)-(b+S(h\cdot x_0))}^2}{b\in\R^d,S\in\Skew(d)}^{\frac12}\nonumber\\
&\quad=\frac1{\sqrt{\abs{\G_{x_0}}}}\min\set[\bigg]{\sum_{h'\in\RR}\norm{\rot(h')u(gh')-\rot(h')u_\mathrm{iso}(h')}^2}{u_\mathrm{iso}\in\Uiso\RR}^{\frac12}\nonumber\\
%&\quad=\frac1{\sqrt{\abs{\G_{x_0}}}}\dist\parens[\Big]{\parens[\big]{\rot(gh)^T\bar u(gh\G_{x_0})}_{h\in\RR/\G_{x_0}},\Uiso\RR}\nonumber\\
&\quad=\frac1{\sqrt{\abs{\G_{x_0}}}}\dist\parens{u(g\fdot)|_\RR,\Uiso\RR}.\label{eq:DD}
\end{align}
Here we have used that the optimal $u_\mathrm{iso}$ satisfies 
\begin{align*}
\rot(h')u(gh')-\rot(h')u_\mathrm{iso}(h')=\rot(h'')u(gh'')-\rot(h'')u_\mathrm{iso}(h'')
\end{align*}
for each $h\in\RR/\G_{x_0}$ and $h',h''\in h$.
By \eqref{eq:DD} and dividing \eqref{eq:CC} by $\abs{\G_{x_0}}$, we are led to introduce the seminorm $\norm\fdot_\RR$ by
\[\norm u_\RR=\parens[\bigg]{\frac1{\abs{\CC_N}}\sum_{g\in\CC_N}\dist^2\parens[\big]{u(g\fdot)|_\RR,\Uiso\RR}}^{\frac12}.\]
More precisely and in agreement with these definitions we have the following general definition. Recall the definition of $\Per(\G,\C^{m\times n})$ from Section~\ref{subsection:structure}.
\begin{Definition}\label{Definition:UtransUrot}
We define the vector spaces
\begin{align*}
\UPerC&:=\Per(\G,\C^{d\times 1})=\set{u\colon\G\to\C^d}{u\text{ is periodic}}
\shortintertext{and}
\UPer&:=\set{u\colon\G\to\R^d}{u\text{ is periodic}}\subset\UPerC.
\end{align*}
For all $\RR\subset\G$ such that $\RR\G_{x_0}=\RR$ we define the vector spaces
\begin{align*}
\Utrans{\RR}&:=\set[\Big]{u\colon\RR\to\R^d}{\exists\+ a\in\R^d\;\forall g\in\RR:\rot(g)\proj{u}(g)=a},\\
\Urot{\RR}&:=\set[\Big]{u\colon\RR\to\R^d}{\exists\+ S\in\Skew(d)\;\forall g\in\RR:\rot(g)\proj{u}(g)=S(g\gdot x_0-x_0)}
\intertext{with $\proj{u}$ as defined in \eqref{eq:pi-proj} and}
\Uiso{\RR}&:=\Utrans{\RR}+\Urot{\RR}.
\end{align*}
For all finite sets $\RR\subset\G$ such that $\RR\G_{x_0}=\RR$ we define the norm
\begin{align*}
\norm\fdot\colon&\{u\colon\RR\to\R^d\}\to[0,\infty),\qquad
u\mapsto\parens[\bigg]{\sum_{g\in\RR}\norm{u(g)}^2}^{\frac12}
\end{align*}
and the seminorm
\begin{align*}
\norm\fdot_{\RR}\colon&\UPer\to[0,\infty),\\
&u\mapsto\parens[\Big]{\frac1{\abs{\CC_N}}\sum_{g\in\CC_N}\norm{\pi_{\Uiso\RR}(u(g\fdot)|_\RR)}^2}^{\frac12}\quad\text{if $u$ is $\T^N$-periodic,}
\end{align*}
where $\pi_{\Uiso\RR}$ is the orthogonal projection on $\{u\colon\RR\to\R^d\}$ with respect to the scalar product induced by the norm $\norm\fdot$ with kernel $\Uiso\RR$.
\end{Definition}
\begin{Remark}\label{Remark:UsefulLeftTranslationInvariantFourierPossible}
\begin{enumerate}
\item The definition of $\norm\fdot_\RR$ is independent of the choice of $\CC_N$.
\item Instead of $\Urot\RR$ one could alternatively consider the vector space
\[\set[\Big]{u\colon\RR\to\R^d}{\exists\+ S\in\Skew(d)\;\forall g\in\RR:\rot(g)\proj{u}(g)}=S(g\gdot x_0),\]
whose sum with $\Utrans\RR$ is also $\Uiso\RR$.
We prefer $\Urot\RR$ in view of Definition~\ref{Definition:nablanorm}.
\item\label{item:UsefulLeftTranslationInvariantFourierPossible} The seminorm $\norm\fdot_\RR$ is left-translation invariant. Thus it can also be represented by means of a convolution operator, see, \eg, \cite[Lemma 5.4]{SchmidtSteinbach:21c}.
\end{enumerate}
\end{Remark}

It is worth noticing that, in view of the discrete nature of the underlying objective structure, the seminorm $\norm\fdot_\RR$ is equivalent to a seminorm acting on a `discrete derivative' in form of a suitable finite difference stencil of $u$. 
\begin{Definition}
For all $u\in\UPer$ and finite sets $\RR\subset\G$ such that $\RR\G_{x_0}=\RR$ we define the \emph{discrete derivative}
\begin{align*}
\nabla_\RR u\colon&\G\to\{v\colon\RR\to\R^d\}\\
&g\mapsto(\nabla_\RR u(g)\colon\RR\to\R^d,h\mapsto \proj{u}(gh)-\rot(h)^T\proj{u}(g)).
\end{align*}
\end{Definition}
\begin{Remark}
Let $\RR\subset\G$ be finite such that $\RR\G_{x_0}=\RR$ and assume that $u\in\UPer$ is induced by an associated deformation mapping such that $v\colon\G/\G_{x_0}\to\R^d,g\mapsto g\cdot x_0+\rot(g)\proj{u}(g)$.
Then $\nabla_\RR u$ encodes finite differences of $v$ via the relation 
\[v(gh\G_{x_0})-v(g\G_{x_0})=(gh)\gdot x_0-g\gdot x_0+\rot(gh)(\parens{\nabla_\RR u(g))(h)}\]
for all $g\in\G$ and $h\in\RR$.
%
%\[\parens[\big]{v_u(gh)-v_u(g)}_{h\in\RR}=\parens[\big]{(gh)\gdot x_0-g\gdot x_0}_{h\in\RR}+(\oplus_{h\in\RR}\rot(gh))\nabla_\RR u(g)\in\R^{d\abs\RR}\]
%, where $\diag((A_h)_{h\in\RR})=\sum_{h\in\RR}(\delta_{g,h}\delta_{g',h})_{g,g'\in\RR}\otimes A_h$ for all $(A_h)_{h\in\RR}\in(\R^{d\times d})^\RR$.
%
\end{Remark}
If $u\in\UPer$ is $\T^N$-periodic for some $N\in\MM$ and $\RR\subset \G$ is finite, then also the discrete derivative $\nabla_{\RR}u$ is $\T^N$-periodic.

\begin{Definition}\label{Definition:nablanorm}
For each finite set $\RR\subset\G$ we define the seminorm
\begin{align*}
\nablanorm\fdot\RR\colon&\UPer\to[0,\infty)\\
&u\mapsto\parens[\Big]{\frac1{\abs{\CC_N}}\sum_{g\in\CC_N}\norm{\pi_{\Urot\RR}(\nabla_\RR u(g))}^2}^{\frac12}\quad\text{if $u$ is $\T^N$-periodic,}
\end{align*}
where $\pi_{\Urot\RR}$ is the orthogonal projection on $\{u\colon\RR\to\R^d\}$ with respect to the norm $\norm\fdot$ with kernel $\Urot\RR$.
\end{Definition}
\begin{Remark}
\begin{enumerate}
\item We have $\nablanorm\fdot\RR=\nablanorm\fdot{\RR\setminus\G_{x_0}}$ for all finite sets $\RR\subset\G$ such that $\RR\G_{x_0}=\RR$.
\item Let $t_i=\iso{I_d}{e_i}$ for $i=1,\dots,d$.
If $\G=\angles{t_1,\dots,t_d}$ and $\RR=\{t_1,\dots,t_d\}$, then $\norm{\pi_{\Urot\RR}(\nabla_\RR u(g))}=\norm{(\nabla_\RR u(g)+(\nabla_\RR u(g))^T)/2}$ for all $u\in\UPer$ and $g\in\G$.
\end{enumerate}
\end{Remark}

\begin{Proposition}\label{Proposition:nablanormequivalent-ohne-0}
Let $\RR\subset\G$ be finite such that $\RR\G_{x_0}=\RR$ and $\G_{x_0}\subset\RR$.
Then the seminorms $\norm\fdot_\RR$ and $\nablanorm\fdot\RR$ are equivalent.
\end{Proposition}

\begin{proof}
Let $\RR\subset\G$ be finite such that $\RR\G_{x_0}=\RR$ and $\G_{x_0}\subset\RR$.
Let $u\in\UPer$ and $N\in\MM$ such that $u$ is $\T^N$-periodic.

We have
\begin{align*}
\nablanorm u\RR^2&=\frac1{\abs{\CC_N}}\sum_{g\in\CC_N}\norm[\big]{\pi_{\Urot\RR}(\nabla_\RR u(g))}^2\\
&=\frac1{\abs{\CC_N}}\sum_{g\in\CC_N}\norm[\big]{\pi_{\Urot\RR}\circ\pi(u(g\fdot)|_\RR)}^2,
\end{align*}
where the mapping $\pi\colon\{v\colon\RR\to\R^d\}\to\{v\colon\RR\to\R^d\}$, $v\mapsto\nabla_\RR v(\id)$ is a projection with kernel $\Utrans\RR$.
Thus we have
\begin{equation}\label{eq:vOne}
\nablanorm u\RR^2=\frac1{\abs{\CC_N}}\sum_{g\in\CC_N}\norm[\big]{u(g\fdot)|_\RR}_1^2,
\end{equation}
where 
\begin{align*}
\norm\fdot_1\colon\{v\colon\RR\to\R^d\}\to\R, \quad v\mapsto\norm[\big]{\pi_{\Urot\RR}\circ\pi(v)}
\end{align*}
is a seminorm with the kernel $\Urot\RR+\Utrans\RR=\Uiso\RR$.
Moreover, we have
\begin{equation}\label{eq:vTwo}
\norm u_\RR^2=\frac1{\abs{\CC_N}}\sum_{g\in\CC_N}\norm[\big]{\pi_{\Uiso\RR}(u(g\fdot)|_\RR)}^2.
\end{equation}
By \eqref{eq:vOne}, \eqref{eq:vTwo} and since the two seminorms $\norm\fdot_1$ and $\norm{\pi_{\Uiso\RR}(\fdot)}$ have the same kernel $\Uiso\RR$ and are thus equivalent, the seminorms $\nablanorm\fdot\RR$ and $\norm\fdot_\RR$ are equivalent.
\end{proof}

%-------------------------------------------------------------------
%-------------------------------------------------------------------
\subsection{Equivalence of local rigidity seminorms}
\label{subsection:local-equivalence}
%-------------------------------------------------------------------

Our aim is to show that, up to equivalence, $\norm\fdot_{\RR}$ does not depend on the particular choice of $\RR$ as long as $\RR$ is rich enough. We begin with some elementary preliminaries. 
\begin{Definition}\label{Definition:Property}
$\RR\subset\G$ is an \emph{admissible} neighborhood range of $\id$ if $\RR$ is finite, $\RR\G_{x_0}=\RR$ and there exist two sets $\RR',\RR''\subset \G$ with $\RR'\RR''\subset\RR$ such that $\id\in\RR'\cap\RR''$, $\RR'$ generates $\G$ and 
\[\aff(\RR''\cdot x_0)=\aff(\G\cdot x_0).\]
\end{Definition}
Admissibility of a neighborood range $\RR$ of $\id$ can be interpreted as a second order property of the stencil $\RR$: it contains a product of two subsets which themselves are rich enough so that the orbit of the first one spans the same affine space as $\G\cdot x_0$ and the second one generates $\G$.
This will be crucial in Lemma~\ref{Lemma:hRSIequivalence} below.
\begin{Example}\label{ex:elementary-chains-ii}
For the atomic chains introduced in Example~\ref{ex:elementary-chains-i} in terms of the groups $\G_1=\angles{t_1}$ and $\G_2=\angles{t_2}$ admissible neighborhood ranges of $\id$ are given by, e.g., $\{\id,t_1,t_1^2\}\subset\G_1$ and $\{\id,t_2,t_2^2,t_2^3\}\subset\G_2$, respectively. 
\end{Example}
\begin{Lemma}\label{Lemma:MatrixRank}
Suppose that $\RR\subset\G$ is finite and such that $\id\in\RR$ and $\aff(\RR\cdot x_0)=\aff(\G\cdot x_0)$. 
Then there exists some $A\in\R^{\daff\times\abs\RR}$ of rank $\daff$ such that in $(\R^d)^\RR \cong \R^{d\times\abs\RR}$
\[(g\cdot x_0-x_0)_{g\in\RR}=\parens[\bigg]{\begin{matrix}0_{d-\daff,\abs\RR}\\ A\end{matrix}}.\]
\end{Lemma}
\begin{proof}
Since $\G\cdot x_0\subset\{0_{d-\daff}\}\times\R^\daff$, there exists some $A\in\R^{\daff\times\abs\RR}$ such that
\[(g\cdot x_0-x_0)_{g\in\RR}=\parens[\bigg]{\begin{matrix}0\\ A\end{matrix}}.\]
It holds
\[\dim(\spano(\set{g\cdot x_0-x_0}{g\in\RR}))=\dim(\aff(\RR\cdot x_0))=\dim(\aff(\G\cdot x_0))=\daff\]
and thus, $\rank(A)=\daff$.
\end{proof}
Below we will estimate $\norm\fdot_{\RR}$ by summing over local contributions. 
To this end, we introduce two auxiliary seminorms that will be needed only in Lemma~\ref{Lemma:hRSIequivalence} and the proof of Theorem~\ref{Theorem:equivalence}.
\begin{Definition}\label{Definition:semiq}
For all finite sets $\RR\subset\G$ such that $\RR\G_{x_0}=\RR$ we define the seminorm
\begin{align*}
\semip{\RR}\colon&\{u\colon\RR\to\R^d\}\to[0,\infty),\qquad
u\mapsto\norm{\pi_{\Uiso\RR}(u)}
\end{align*}
on $(\R^{d})^{\RR}$ whose kernel is $\Uiso\RR$ (see Definition~\ref{Definition:UtransUrot}).
Moreover, for all finite sets $\RR_1,\RR_2\subset\G$ such that $\RR_2\G_{x_0}=\RR_2$ we define the seminorm
\begin{align*}
\semiq{\RR_1}{\RR_2}\colon\{u\colon\RR_1\RR_2\to\R^d\}\to[0,\infty),\qquad
u\mapsto \parens[\bigg]{\sum_{g\in\RR_1}\semip[2]{\RR_2}\parens[\big]{u(g\fdot)|_{\RR_2}}}^{\frac12}
\end{align*}
on $(\R^{d})^{\RR_1\RR_2}$. 
\end{Definition}
We remark that $\semiq{\RR_1}{\RR_2}$ itself is defined by summing the local contributions $\semip[2]{\RR_2}\parens[\big]{u(g\fdot)|_{\RR_2}}$ over $g\in\RR_1$. 
\begin{Lemma}\label{Lemma:hRSIequivalence}
Suppose that $\RR_1\subset\G$ is finite and $\RR_2\subset\G$ is an admissible neighborhood range of $\id$.
Then there exists a finite set $\RR_3\subset\G$ such that $\RR_1\subset \RR_3\RR_2$ and the seminorms $\semip{\RR_3\RR_2}$ and $\semiq{\RR_3}{\RR_2}$ are equivalent.
\end{Lemma}
This lemma is crucial: Any admissible neighborhood range $\RR_2$ can be modified to a set $\RR_3\RR_2$ which is rich enough to cover $\RR_1$ and such that $\semip{\RR_3\RR_2}$ is still controlled by $\semiq{\RR_3}{\RR_2}$ and hence ultimately in terms of local contributions with respect to the original $\semip{\RR_2}$. 
\begin{proof}
Since $(\R^{d})^{\RR_3\RR_2}$ is finite dimensional, it suffices to show that there exists a finite set $\RR_3\subset\G$ with $\RR_1\subset\RR_3\RR_2$ and
\[\ker(\semiq{\RR_3}{\RR_2})=\Uiso{\RR_3\RR_2}.\]

First we show that $\Uiso{\RR_3\RR_2}\subset\ker(\semiq{\RR_3}{\RR_2})$ for all finite sets $\RR_3\subset \G$ with $\RR_1\subset\RR_3\RR_2$:
Let $u\in\Uiso{\RR_3\RR_2}$.
As there are $a\in\R^d$ and $S\in\Skew(d)$ such that for all $h\in\RR_2$ and $g\in\RR_3$ 
\begin{align*}
\rot(h)\proj{u}(gh)&=\rot(g)^Ta+\rot(g)^TS((gh)\cdot x_0-x_0)\\
&=\rot(g)^Ta+\rot(g)^TS(g\cdot x_0-x_0)+\rot(g)^TS\rot(g)(h\cdot x_0-x_0), 
\end{align*}
we see that $u(g\fdot)|_{\RR_2}\in\Uiso{\RR_2}$ for every $g\in\RR_3$. 
Since $\semip{\RR_2}$ vanishes on $\Uiso{\RR_2}$, it follows
\begin{align*}
\semiq[2]{\RR_3}{\RR_2}(u)
=\sum_{g\in\RR_3}\semip[2]{\RR_2}\parens{u(g\fdot)|_{\RR_2}}
=0.
\end{align*}
Hence, we have $\Uiso{\RR_3\RR_2}\subset\ker(\semiq{\RR_3}{\RR_2})$.

Now we show that there exists some finite set $\RR_3\subset\G$ such that $\ker(\semiq{\RR_3}{\RR_2})\subset\Uiso{\RR_3\RR_2}$.
By admissibility of $\RR_2$ there exist finite sets $\RR_2',\RR_2''\subset \G$ such that $\id\in\RR_2'\cap\RR_2''$, $\RR_2'$ generates $\G$, $\RR_2''$ is such that $\aff(\RR_2''\cdot x_0)=\aff(\G\cdot x_0)$ and 
\[\RR_2'\RR_2''\subset\RR_2.\]
Without loss of generality we may assume that $\RR_2''\G_{x_0}=\RR_2''$.
Since $\RR_2'$ generates $\G$, there exists some $n_0\in\N$ such that
\[\RR_1\subset\{\id\}\cup\bigcup_{k=1}^{n_0}\set[\Big]{g_1\dots g_k}{g_1,\dots,g_k\in\RR_2'\cup(\RR_2')^{-1}}.\]
Let
\[\RR_3=\{\id\}\cup\bigcup_{k=1}^{n_0}\set[\Big]{g_1\dots g_k}{g_1,\dots,g_k\in\RR_2'\cup (\RR_2')^{-1}}.\]
Let $u\in\ker(\semiq{\RR_3}{\RR_2})$.
By Definition~\ref{Definition:semiq} and Definition~\ref{Definition:UtransUrot} for all $g\in\RR_3$ there exist some $a(g)\in\R^d$ and $S(g)\in\Skew(d)$ such that
\begin{equation}\label{eq:helpRSI1}
\rot(h)\proj{u}(gh)=a(g)+S(g)(h\gdot x_0-x_0)\qquad\text{for all }h\in\RR_2.
\end{equation}
Since $\G\cdot x_0\subset\{0_{d-\daff}\}\times\R^\daff$, we have $h\gdot x_0-x_0\in\{0_{d-\daff}\}\times\R^{\daff}$ for all $h\in\RR_2$.
Hence, for all $g\in\RR_3$ we may assume
\[S(g)=\parens[\bigg]{\begin{matrix}0 & S_1(g)\\-S_1(g)^T & S_2(g)\end{matrix}}\]
for some $S_1(g)\in\R^{(d-\daff)\times\daff}$ and $S_2(g)\in\Skew(\daff)$.
We prove inductively that for $n=0,1,\dots,n_0$ for all $g\in\{\id\}\cup\bigcup_{k=1}^{n}\set[\Big]{g_1\dots g_k}{g_1,\dots,g_k\in\RR_2'\cup(\RR_2')^{-1}}$ it holds
\begin{equation}\label{eq:helpRSI2}
\rot(g)a(g)=a(\id)+S(\id)(g\cdot x_0-x_0)\quad\text{and}\quad S(g)=\rot(g)^TS(\id)\rot(g).
\end{equation}
For $n=0$ the induction hypothesis is true.

We assume the induction hypothesis holds for arbitrary but fixed $0\le n<n_0$.
Let $g\in\{\id\}\cup\bigcup_{k=1}^n\set[\Big]{g_1\dots g_k}{g_1,\dots,g_k\in\RR_2'\cup(\RR_2')^{-1}}$ and $r\in \RR_2'\cup(\RR_2')^{-1}$.
\begin{case}[Case 1: $r\in \RR_2'$] \ \\
Since $g\in\RR_3$ and $r\RR_2''\subset\RR_2$, by \eqref{eq:helpRSI1} we have
\begin{equation}\label{eq:AAhelpRSIi}
\rot(rh)\proj{u}(grh)=a(g)+S(g)((rh)\cdot x_0-x_0)\qquad\text{for all }h\in\RR_2''.
\end{equation}
Since $gr\in\RR_3$ and $\RR_2''\subset\RR_2$, by \eqref{eq:helpRSI1} we have
\begin{equation}\label{eq:AAhelpRSIii}
\rot(h)\proj{u}(grh)=a(gr)+S(gr)(h\cdot x_0-x_0)\qquad\text{for all }h\in\RR_2''.
\end{equation}
By \eqref{eq:AAhelpRSIi} and \eqref{eq:AAhelpRSIii} we have
\begin{equation}\label{eq:AA1}
\rot(r)a(gr)+\rot(r)S(gr)(h\cdot x_0-x_0)=a(g)+S(g)((rh)\cdot x_0-x_0)
\end{equation}
for all $h\in\RR_2''$.
Since $\id\in\RR_2''$, by \eqref{eq:AA1} we have
\begin{equation}\label{eq:AA2}
\rot(r)a(gr)=a(g)+S(g)(r\cdot x_0-x_0)
\end{equation}
and with the induction hypothesis follows
\begin{align*}
\rot(gr)a(gr)&=a(\id)+S(\id)(g\cdot x_0-x_0)+S(\id)\rot(g)(r\cdot x_0-x_0)\\
&=a(\id)+S(\id)((gr)\cdot x_0-x_0).
\end{align*}
By \eqref{eq:AA1} and \eqref{eq:AA2} we have
\begin{align}\label{eq:AA4}
\rot(r)S(gr)(h\cdot x_0-x_0)&=S(g)((rh)\cdot x_0-r\cdot x_0)\nonumber\\
&=S(g)\rot(r)(h\cdot x_0-x_0)
\end{align}
for all $h\in\RR_2''$.
By Lemma~\ref{Lemma:MatrixRank} there exists some $A\in\R^{\daff\times\abs{\RR_2''}}$ of rank $\daff$ such that
\[(h\cdot x_0-x_0)_{h\in\RR_2''}=\parens[\bigg]{\begin{matrix}0_{\daff,\abs{\RR_2''}}\\ A\end{matrix}}.\]
By \eqref{eq:AA4} and the induction hypothesis we have
\begin{equation}\label{eq:AA5}
(S(gr)-\rot(gr)^TS(\id)\rot(gr))\parens[\bigg]{\begin{matrix}0\\ A\end{matrix}}=0.
\end{equation}
By Lemma~\ref{Lemma:Shelp2} there exist some $B_{gr}\in\O(d-\daff)$ and $C_{gr}\in\O(\daff)$ such that $\rot(gr)=B_{gr}\oplus C_{gr}$.
Equation~\eqref{eq:AA5} is equivalent to
\[\parens[\bigg]{\begin{matrix}(S_1(gr)-B_{gr}^TS_1(\id)C_{gr})A\\(S_2(gr)-C_{gr}^TS_2(\id)C_{gr})A\end{matrix}}=0.\]
Since the rank of $A$ is equal to the number of its rows, we have $S_1(gr)=B_{gr}^TS_1(\id)C_{gr}$ and $S_2(gr)=C_{gr}^TS_2(\id)C_{gr}$ which is equivalent to $S(gr)=\rot(gr)^TS(\id)\rot(gr)$.
\end{case}
\begin{case}[Case 2: $r^{-1}\in \RR_2'$] \ \\
Since $g\in\RR_3$ and $\RR_2''\subset\RR_2$, by \eqref{eq:helpRSI1} we have
\begin{equation}\label{eq:BBhelpRSIi}
\rot(h)\proj{u}(gh)=a(g)+S(g)(h\cdot x_0-x_0)\qquad\text{for all }h\in\RR_2''.
\end{equation}
Since $gr\in\RR_3$ and $r^{-1}\RR_2''\subset\RR_2$, by \eqref{eq:helpRSI1} we have
\begin{equation}\label{eq:BBhelpRSIii}
\rot(r^{-1}h)\proj{u}(gh)=a(gr)+S(gr)((r^{-1}h)\cdot x_0-x_0)\qquad\text{for all }h\in\RR_2''.
\end{equation}
By \eqref{eq:BBhelpRSIi} and \eqref{eq:BBhelpRSIii} we have
\begin{equation}\label{eq:BB1}
a(gr)+S(gr)((r^{-1}h)\cdot x_0-x_0)=\rot(r)^Ta(g)+\rot(r)^TS(g)(h\cdot x_0-x_0)
\end{equation}
for all $h\in\RR_2''$.
Since $\id\in\RR_2''$, by \eqref{eq:BB1} we have
\begin{equation}\label{eq:BB2}
a(gr)+S(gr)(r^{-1}\cdot x_0-x_0)=\rot(r)^Ta(g).
\end{equation}
By \eqref{eq:BB1} and \eqref{eq:BB2} we have
\[S(gr)((r^{-1}h)\cdot x_0-x_0)=S(gr)(r^{-1}\cdot x_0-x_0)+\rot(r)^TS(g)(h\cdot x_0-x_0)\]
for all $h\in\RR_2''$.
This is equivalent to
\begin{equation}\label{eq:BB4}
S(gr)\rot(r)^T(h\cdot x_0-x_0)=\rot(r)^TS(g)(h\cdot x_0-x_0)
\end{equation}
for all $h\in\RR_2''$.
By Lemma~\ref{Lemma:MatrixRank} there exists some $A\in\R^{\daff\times\abs{\RR_2''}}$ of rank $\daff$ such that
\[(h\cdot x_0-x_0)_{h\in\RR_2''}=\parens[\bigg]{\begin{matrix}0_{\daff,\abs{\RR_2''}}\\ A\end{matrix}}.\]
By \eqref{eq:BB4} and the induction hypothesis we have
\begin{equation}\label{eq:BB5}
(S(gr)-\rot(gr)^TS(\id)\rot(gr))\rot(r)^T\parens[\bigg]{\begin{matrix}0\\ A\end{matrix}}=0.
\end{equation}
By Lemma~\ref{Lemma:Shelp2} there exist $B_r,B_{gr}\in\O(d-\daff)$ and $C_r,C_{gr}\in\O(\daff)$ such that $\rot(r)=B_r\oplus C_r$ and $\rot(gr)=B_{gr}\oplus C_{gr}$.
Equation~\eqref{eq:BB5} is equivalent to
\[\parens[\bigg]{\begin{matrix}(S_1(gr)-B_{gr}^TS_1(\id)C_{gr})C_r^TA\\(S_2(gr)-C_{gr}^TS_2(\id)C_{gr})C_r^TA\end{matrix}}=0.\]
Since $C_r$ is invertible and the rank of $A$ is equal to the number of its rows, we have $S_1(gr)=B_{gr}^TS_1(\id)C_{gr}$ and $S_2(gr)=C_{gr}^TS_2(\id)C_{gr}$ which is equivalent to $S(gr)=\rot(gr)^TS(\id)\rot(gr)$.
Since $S(gr)=\rot(gr)^TS(\id)\rot(gr)$, we have by \eqref{eq:BB2} and the induction hypothesis that
\begin{align*}
\rot(gr)a(gr)&=\rot(g)a(g)-\rot(gr)S(gr)(r^{-1}\cdot x_0-x_0)\\
&=a(\id)+S(\id)(g\cdot x_0-x_0)-S(\id)\rot(gr)(r^{-1}\cdot x_0-x_0)\\
&=a(\id)+S(\id)((gr)\cdot x_0-x_0).
\end{align*}
\end{case}
By \eqref{eq:helpRSI1} and \eqref{eq:helpRSI2} we have that
\[\rot(g)u(g)=\rot(g)a(g)=a(\id)+S(\id)(g\gdot x_0-x_0)\quad\text{for all }g\in\RR_3\RR_2\]
and thus, $u\in\Uiso{\RR_3\RR_2}$.
\end{proof}
\begin{Theorem}\label{Theorem:equivalence}
Suppose that $\RR_1,\RR_2\subset\G$ are admissible neighborhood ranges of $\id$.
Then the two seminorms $\norm\fdot_{\RR_1}$ and $\norm\fdot_{\RR_2}$ are equivalent.
\end{Theorem}
\begin{proof}
It is sufficient to show that there exists a constant $C>0$ such that $\norm\fdot_{\RR_1}\le C\norm\fdot_{\RR_2}$.
Since $\RR_1$ is finite, by Lemma~\ref{Lemma:hRSIequivalence} there exists a finite set $\RR_3\subset\G$ such that $\RR_1\subset \RR_3\RR_2$ and some $C>0$ with $\semip{\RR_3\RR_2}\le C\semiq{\RR_3}{\RR_2}$.
Let $u\in\UPer$.
There exists some $N\in \MM$ such that $u$ is $\T^N$-periodic.
We have
\begin{align*}
\norm u_{\RR_1}^2&\le\norm u_{\RR_3\RR_2}^2\\
&=\frac1{\abs{\CC_N}}\sum_{g\in\CC_N}\semip[2]{\RR_3\RR_2}(u(g\fdot)|_{\RR_3\RR_2})\\
&\le\frac {C^2}{\abs{\CC_N}}\sum_{g\in\CC_N}\semiq[2]{\RR_3}{\RR_2}(u(g\fdot)|_{\RR_3\RR_2})\\
&= \frac {C^2}{\abs{\CC_N}}\sum_{g\in\CC_N}\sum_{\tilde g\in\RR_3}\semip[2]{\RR_2}\parens[\big]{u(g\tilde g\fdot)|_{\RR_2}}\\
&=\frac {C^2}{\abs{\CC_N}}\sum_{\tilde g\in\RR_3}\sum_{g\in\CC_N\tilde g}\semip[2]{\RR_2}\parens[\big]{u(g\fdot)|_{\RR_2}}\\
&=C^2\abs{\RR_3}\norm u_{\RR_2}^2,
\end{align*}
where we used that $\CC_N\tilde g$ is a representation set of $\G/\T^N$ for all $\tilde g\in\RR_3$ in the last step.
Hence, we have $\norm\fdot_{\RR_1}\le C\abs{\RR_3}^{\frac 12}\norm\fdot_{\RR_2}$.
\end{proof}
\begin{Remark}
In Theorem~\ref{Theorem:equivalence} the premise that $\RR_1$ and $\RR_2$ are admissible neighborhood ranges of $\id$ cannot be weakened to the premise that for both $i = 1$ and $i=2$ one has $\RR_i\G_{x_0}=\RR_i$, $\aff(\RR_i\cdot x_0)=\aff(\G\cdot x_0)$ and $\RR_i$ is a generating set of $\G$, see Example~\ref{Example:propertystarstar}.
\end{Remark}
%
%-------------------------------------------------------------------
%-------------------------------------------------------------------
\subsection{Intrinsic seminorms and their equivalence to local seminorms}
\label{subsection:Korn}
%-------------------------------------------------------------------
%
We now define the seminorm $\zeronorm{\fdot}{\RR}$ which measures the local distance of a deformation to the subset of those isometries that vanish if both the preimage and target space are projected to $\R^{d_2}$. Thus $\zeronorm{u}{\RR}$ controls the size of the corresponding part of the discrete gradient of the displacement $u$ globally. 
\begin{Definition}
For all $\RR\subset\G$ such that $\RR\G_{x_0}=\RR$ we define the vector spaces
\begin{align*}
&\begin{aligned}\zeroUrot{\RR}&:=\set[\bigg]{u\colon\RR\to\R^d}{\exists\+ S\in\Skew_{0,d_2}(d)\;\forall g\in\RR:\rot(g)\proj{u}(g)=S(g\gdot x_0-x_0)}\\
&\subset\Urot\RR\end{aligned}
\shortintertext{and}
&\zeroUiso{\RR}:=\Utrans{\RR}+\zeroUrot{\RR}\subset\Uiso\RR,
\end{align*}
where
\[\Skew_{0,d_2}(d):=\set[\bigg]{\parens[\bigg]{\begin{matrix}S_1&S_2\\-S_2^T&0\end{matrix}}}{S_1\in\Skew(d_1),S_2\in\R^{d_1\times d_2}}\subset \Skew(d).\]
\end{Definition}
\begin{Definition}\label{Definition:Seminorm-Null}
For all finite sets $\RR\subset\G$ such that $\RR\G_{x_0}=\RR$ we define the seminorms
\begin{align*}
\zeronorm{\fdot}{\RR}\colon&\UPer\to[0,\infty)\\
&u\mapsto\parens[\Big]{\frac1{\abs{\CC_N}}\sum_{g\in\CC_N}\norm{\pi_{\zeroUiso\RR}(u(g\fdot)|_\RR)}^2}^{\frac12}\quad\text{if $u$ is $\T^N$-periodic,}\\
\shortintertext{and}
\nablazeronorm\fdot\RR\colon&\UPer\to[0,\infty)\\
&u\mapsto\parens[\Big]{\frac1{\abs{\CC_N}}\sum_{g\in\CC_N}\norm{\pi_{\zeroUrot\RR}(\nabla_\RR u(g))}^2}^{\frac12}\quad\text{if $u$ is $\T^N$-periodic,}
\end{align*}
where $\pi_{\zeroUiso\RR}$ and $\pi_{\zeroUrot\RR}$ are the orthogonal projections on $\{u\colon\RR\to\R^d\}$ with respect to the norm $\norm\fdot$ with kernels $\zeroUiso\RR$ and $\zeroUrot\RR$, respectively.
\end{Definition}
\begin{Remark}
We have $\nablazeronorm\fdot\RR=\nablazeronorm\fdot{\RR\setminus\G_{x_0}}$ for all finite sets $\RR\subset\G$.
\end{Remark}
\begin{Proposition}\label{Proposition:nablanormequivalent-0}
Let $\RR\subset\G$ be finite and $\id\in\RR$.
Then the seminorms $\zeronorm\fdot\RR$ and $\nablazeronorm\fdot\RR$ are equivalent.
\end{Proposition}
\begin{proof}
The proof is analogous to the proof of Proposition~\ref{Proposition:nablanormequivalent-ohne-0}.
\end{proof}
As a final preparation we state the following elementary lemma, which is well-known, and include its short proof.
\begin{Lemma}\label{Lemma:constantSkew}
There exists a constant $c>0$ such that for every $n\in\N$ it holds
\[\norm[\big]{x\otimes y^T+A}\ge c(\norm[\big]{x\otimes y^T}+\norm[\big]{A})
\qquad\text{for all }x,y\in\C^n,\,A\in\Skew(n,\C).\]
\end{Lemma}
\begin{proof}
Let $x,y\in\C^n$ and $A\in\Skew(n,\C)$. Since $\C^{n\times n}=\Sym(n,\C)\oplus\Skew(n,\C)$ we have
\begin{align*}
\norm[\big]{x\otimes y^T+A}^2&\ge\norm[\Big]{\frac12\parens[\big]{x\otimes y^T+y\otimes x^T}}^2\\
&=\frac12\norm[\big]{x\otimes y^T}^2+\frac12\abs[\bigg]{\sum_{i=1}^n x_i\conj{y_i}}^2\\
&\ge\frac12\norm[\big]{x\otimes y^T}^2.
\end{align*}
If $\norm A\le 2\norm{x\otimes y^T}$, then
\[\norm[\big]{x\otimes y^T+A}\ge\frac1{\sqrt2}\norm[\big]{x\otimes y^T}\ge\frac1{3\sqrt 2}\parens[\big]{\norm[\big]{x\otimes y^T}+\norm A}.\]
If $\norm A\ge 2\norm{x\otimes y^T}$, then
\[\norm[\big]{x\otimes y^T+A}\ge\norm A-\norm[\big]{x\otimes y^T}\ge \frac13\parens[\big]{\norm[\big]{x\otimes y^T}+\norm A}.\qedhere\]
\end{proof}
%
% That every skew-symmetric matrix has even rank can be found in The Theory of Matrices of Gantmacher, 2000, Theorem 6 on p.12 and also in Elementary Matrix Theory of Howard Eves (Corollary 5.4.2)
%
The following lemma provides a technical core estimate on which the proof of our main Theorem~\ref{Theorem:Korn} hinges. 
\begin{Lemma}\label{Lemma:helpTheorem}
Let $n\in\N$, $q\in\N_0$ and $\beta_1,\dots,\beta_q\in\R$. Then there exists an integer $N\in\N$ such that
\[\max_{m\in\{1,\dots,N\}}\norm[\bigg]{a\otimes(\sin(m\alpha_1),\dots,\sin(m\alpha_n))+\sum_{k=1}^q\sin(m\beta_k)B_k+mS}\ge \norm S\]
for all $a\in\C^n$, $\alpha_1,\dots,\alpha_n\in\R$, $B_1,\dots,B_q\in\C^{n\times n}$ and $S\in\Skew(n,\C)$.
\end{Lemma}
\begin{Remark}
If $q=0$, then the term $\sum_{k=1}^q\sin(m\beta_k)B_k$ is the empty sum.
\end{Remark}
\begin{proof}
It suffices to prove that there exists a constant $c>0$ such that for all $n\in\N$, $q\in\N_0$ and $\beta_1,\dots,\beta_q\in\R$ there exists an integer $N\in\N$ such that
\[\max_{m\in\{1,\dots,N\}}\norm[\bigg]{a\otimes(\sin(m\alpha_1),\dots,\sin(m\alpha_n))+\sum_{k=1}^q\sin(m\beta_k)B_k+mS}\ge c\norm S\]
for all $a\in\C^n$, $\alpha_1,\dots,\alpha_n\in\R$, $B_1,\dots,B_q\in\C^{n\times n}$ and $S\in\Skew(n,\C)$. Indeed, applying this inequality with $\tilde N$, $(\ceil{\tfrac1c}\alpha_i)_{1\le i\le n}$, $(\ceil{\tfrac1c}\beta_k)_{1\le k\le q}$ and $\ceil{\tfrac1c}S$ and setting $N=\ceil{\tfrac1c}\tilde N$ we obtain the original claim from
\begin{align*}
\max_{m\in\{1,\dots,N\}}&\norm[\bigg]{a\otimes(\sin(m\alpha_1),\dots,\sin(m\alpha_n))+\sum_{k=1}^q\sin(m\beta_k)B_k+mS}\\
\begin{split}
&\ge\max_{m\in\braces[\big]{1,\dots,\tilde N}}\norm[\Big]{a\otimes(\sin(m(\ceil{\tfrac1c}\alpha_1)),\dots,\sin(m(\ceil{\tfrac1c}\alpha_n)))\\
&\qquad+\sum_{k=1}^q\sin(m(\ceil{\tfrac1c}\beta_k))B_k+m(\ceil{\tfrac1c}S)}. 
\end{split}
\end{align*}
Since
\[\norm M\ge\frac1{n^2}\sum_{\substack{i,j\in\{1,\dots,n\}\\i<j}}\norm[\Big]{\big(\begin{smallmatrix}m_{ii}&m_{ij}\\m_{ji}&m_{jj}\end{smallmatrix}\big)}\]
for all $M=(m_{ij})\in\C^{n\times n}$, it suffices to prove the assertion for $n=2$.\\
Let $q\in\N_0$ and $\beta_1,\dots,\beta_q\in\R$. Without loss of generality we assume $\beta_1,\dots,\beta_q\in\R\setminus(\pi\Q)$: Let $n_0\in\N$ be such that $n_0\beta_k\in\pi\Z$ for all $k\in\{1,\dots,q\}$ with $\beta_k\in\pi\Q$.
Then we have
\begin{align*}
&\max_{m\in\{1,\dots,n_0N\}}\norm[\bigg]{a\otimes(\sin(m\alpha_1),\sin(m\alpha_2))+\sum_{k=1}^q\sin(m\beta_k)B_k+mS}\ge\\
&\quad\max_{m\in\{1,\dots,N\}}\norm[\bigg]{a\otimes(\sin(m(n_0\alpha_1)),\sin(m(n_0\alpha_2)))+\sum_{\substack{k=1\\ \beta_k\notin\pi\Q}}^q\sin(m(n_0\beta_k))B_k+m(n_0S)}
\end{align*}
for all $N\in\N$, $a\in\C^2$, $\alpha_1,\alpha_2\in\R$, $B_1,\dots,B_q\in\C^{2\times 2}$ and $S\in\Skew(2,\C)$.\\
For all $a>0$ we define the function
\begin{align*}
\abs\fdot_a\colon&\R\to[0,\infty)\\
&x\mapsto\dist(x,a\Z).
\end{align*}
%\min\set{\abs{x+na}}{n\in\Z}
%
%(The function $\abs\fdot_1$ is the \emph{distance to nearest integer function}.) \\
%
Moreover, without loss of generality we may assume $\abs{\beta_k-\beta_l}_{2\pi}>0$ for all $k\neq l$ and since
\[\sin(m\beta)=-\sin(m(2\pi-\beta))\]
also $\abs{\beta_k+\beta_l}_{2\pi}\neq0$ for all $k\neq l$.
For the definition of a suitable integer $N\in\N$ and the following proof we define some positive constants.
By Lemma~\ref{Lemma:constantSkew} there exists a constant $c_L>0$ such that
\[\norm{x\otimes y^T+S}\ge c_L\norm x(\abs{y_1}+\abs{y_2})+c_L\norm{S}\]
for all $x,y\in\C^2$ and $S\in\Skew(2,\C)$.
In particular, this inequality implies the assertion for $q=0$.
Hence we may assume $q\neq0$, \ie $q\in\N$.
Let
\begin{align*}
\delta_1&=\min_{\substack{\gamma_1,\gamma_2\in\{\pm\beta_1,\dots,\pm\beta_q\}\\ \gamma_1\neq\gamma_2}}\abs{\gamma_1-\gamma_2}_{2\pi},\quad\mu_1=\frac1{2q}\parens[\Big]{\frac{\delta_1}{2\pi}}^{2q-1},\\
C_1&=\frac{4(2q+1)}{\mu_1},\quad C_2=\frac{6q}{\mu_1} \quad\text{and}\quad C_3=\max\braces[\Big]{\frac{4q+2}{\mu_1},\frac{32\pi C_2}{\delta_1}}.
\end{align*}
By Kronecker's approximation Theorem~\ref{Theorem:Kronecker}, for all $k\in\{1,\dots,q\}$ there exists an integer $q_k$ such that $2C_3+2<q_k$ and
\[\abs[\Big]{q_k\cdot\frac{\beta_k}\pi+\frac12}_1\le\frac1{3\pi C_3}.\]
Let
\[N_1=\max\braces[\bigg]{\ceil[\bigg]{\frac{2C_1}{c_L}},2q,1+\ceil[\bigg]{\frac{16\pi C_2}{\delta_1}},q_1,\dots,q_q}\in\N.\]
For all $\alpha\in\R$ we define $(\alpha)_{2\pi}\in\R$ by $\{(\alpha)_{2\pi}\}=[-\pi,\pi)\cap(\alpha+2\pi\Z)$. We have $\abs{(\alpha)_{2\pi}}=\abs\alpha_{2\pi}$.
By Taylor's Theorem we have for all $\alpha,\beta\in\R$ and $n\in\N$
\[\sin(n\alpha)=\sin(n(\beta+(\alpha-\beta)_{2\pi}))=\sin(n\beta)+n(\alpha-\beta)_{2\pi}\cos(n\beta)+R(n,\alpha,\beta)\]
where $R(n,\alpha,\beta)$ is the remainder term. Let $\delta_2>0$ be so small that
\begin{equation}\label{eq:LemmaTaylor}
\abs{R(n,\alpha,\beta)}\le\tfrac12n\abs{\alpha-\beta}_{2\pi}\abs{\cos(n\beta)}
\end{equation}
for all $n\in\{1,\dots,N_1\}$, $\alpha\in\R$ with $\abs{\alpha-\beta}_{2\pi}<\delta_2$ and $\beta\in\{0,\pi,\beta_1,\dots,\beta_q\}$.
Let
\[\delta_3=\min\{\delta_1,\delta_2\}, \quad\mu_2=\frac1{2q+2}\parens[\Big]{\frac{\delta_3}{2\pi}}^{2q+1}\quad\text{and}\quad C_4=\frac{2q+3}{\mu_2}.\]
Let
\[N=\max\braces[\big]{N_1,1+\ceil{C_4}}\in\N.\]
Now, let $a=(a_1,a_2)^T\in\C^2$, $\alpha_1,\alpha_2\in\R$, $B_k=\parens[\Big]{\begin{smallmatrix}b_{11}^{(k)}&b_{12}^{(k)}\\ b_{21}^{(k)}&b_{22}^{(k)}\end{smallmatrix}}\in\C^{2\times2}$ for all $k\in\{1,\dots,q\}$ and $S=\parens[\big]{\begin{smallmatrix}0&-s\\s&0\end{smallmatrix}}\in\Skew(2,\C)$. We denote
\[\text{LHS}=\max_{m\in\{1,\dots,N\}}\norm[\bigg]{a\otimes(\sin(m\alpha_1),\sin(m\alpha_2))+\sum_{k=1}^q\sin(m\beta_k)B_k+mS}.\]
\begin{case}[Case 1: $\forall\,i\in\{1,2\} : ((\abs{\alpha_i}_{2\pi}<\delta_2)\lor(\abs{\alpha_i-\pi}_{2\pi}<\delta_2))$] \
\begin{caseinside}[Case 1.1: $\sum_{k=1}^q\norm{B_k}\ge C_1(\norm a(\abs{\alpha_1}_\pi+\abs{\alpha_2}_\pi)+\norm S)$]\ \\
Let $i,j\in\{1,2\}$ with $\sum_{k=1}^q\abs{b_{ij}^{(k)}}\ge\tfrac14\sum_{k=1}^q\norm{B_k}$.
By the definition of $\delta_1$ we have
\[\min_{\substack{\gamma_1,\gamma_2\in\{\pm\beta_1,\dots,\pm\beta_q\}\\ \gamma_1\neq\gamma_2}}\abs{\euler^{\iu\gamma_1}-\euler^{\iu\gamma_2}}\ge\min_{\substack{\gamma_1,\gamma_2\in\{\pm\beta_1,\dots,\pm\beta_q\}\\ \gamma_1\neq\gamma_2}}\frac{\abs{\gamma_1-\gamma_2}_{2\pi}}\pi\ge\frac{\delta_1}\pi.\]
By Tur\'{a}n's Minimax Theorem~\ref{Theorem:TuranTheorem} there exist an integer $\nu\in\{1,\dots,2q\}$ such that
\begin{align*}
\norm[\bigg]{\sum_{k=1}^q \sin(\nu\beta_k)B_k}&\ge\abs[\bigg]{\sum_{k=1}^q b_{ij}^{(k)} \sin(\nu\beta_k)}\\
&=\abs[\bigg]{\sum_{k=1}^q \parens[\bigg]{\frac{\iu b_{ij}^{(k)}}2\euler^{-\iu\nu\beta_k}+\frac{-\iu b_{ij}^{(k)}}2\euler^{\iu \nu\beta_k}}}\\
&\ge\mu_1\sum_{k=1}^q\abs[\big]{b_{ij}^{(k)}}\\
&\ge\frac{\mu_1}4\sum_{k=1}^q\norm{B_k}.
\end{align*}
We have
\begin{align*}
\text{LHS}&\ge\norm[\bigg]{\sum_{k=1}^q \sin(\nu\beta_k)B_k}-\norm{a\otimes(\sin(\nu\alpha_1),\sin(\nu\alpha_2))}-\norm{\nu S}\\
&\ge\frac{\mu_1}4\sum_{k=1}^q\norm{B_k}-2q\norm a(\abs{\alpha_1}_\pi+\abs{\alpha_2}_\pi)-2q\norm S\\
&\ge\norm S.
\end{align*}
\end{caseinside}
\begin{caseinside}[Case 1.2: $\sum_{k=1}^q\norm{B_k}\le C_1(\norm a(\abs{\alpha_1}_\pi+\abs{\alpha_2}_\pi)+\norm S)$]\ \\
We have
\begin{align*}
\text{LHS}&\ge\norm{a\otimes(\sin(N_1\alpha_1),\sin(N_1\alpha_2))+N_1S}-\norm[\bigg]{\sum_{k=1}^q \sin(N_1\beta_k) B_k}\\
&\ge c_L\norm a(\abs{\sin(N_1\alpha_1)}+\abs{\sin(N_1\alpha_2)})+c_L\norm{N_1S}-\sum_{k=1}^q\norm{B_k}\\
&\overset{\eqref{eq:LemmaTaylor}}\ge\frac{c_LN_1}2\norm a(\abs{\alpha_1}_\pi+\abs{\alpha_2}_\pi)+c_LN_1\norm S-\sum_{k=1}^q\norm{B_k}\\
&\ge\frac{c_LN_1}2\parens[\big]{\norm a(\abs{\alpha_1}_\pi+\abs{\alpha_2}_\pi)+\norm S}+\frac{c_L}2\norm S-\sum_{k=1}^q\norm{B_k}\\
&\ge\frac{c_L}2\norm S.
\end{align*}
\end{caseinside}
\end{case}
\begin{case}[Case 2: $\exists\,i\in\{1,2\},\,\exists\, k\in\{1,\dots,q\} : ((\abs{\alpha_i-\beta_k}_{2\pi}<\delta_2)\lor(\abs{\alpha_i+\beta_k}_{2\pi}<\delta_2))$]\ \\
Without loss of generality let $i=1$ and $k=1$. 
Without loss of generality we may assume $\abs{\alpha_1-\beta_1}_{2\pi}<\delta_2$ since
\[a\otimes (\sin(m\alpha_1),\sin(m\alpha_2))=(-a)\otimes(\sin(m(-\alpha_1)),\sin(m(-\alpha_2)))\quad\text{for all }m\in\N.\]
Let $\delta_k$ be equal to 1 if $k=0$ and equal to 0 otherwise.
\begin{caseinside}[Case 2.1: $\sum_{k=1}^q\abs{a_2\delta_{k-1}+b_{21}^{(k)}}\ge C_2\abs{a_2}\abs{\alpha_1-\beta_1}_{2\pi}$ and \\
$\phantom{\qquad \qquad} \max\{\abs{a_2}\abs{\alpha_1-\beta_1}_{2\pi},\sum_{k=1}^q\abs{a_2\delta_{k-1}+b_{21}^{(k)}}\}\ge C_3\abs s$]\ \\
Since $C_2\ge1$ the condition is equivalent to
\[\sum_{k=1}^q\abs{a_2\delta_{k-1}+b_{21}^{(k)}}\ge C_2\abs{a_2}\abs{\alpha_1-\beta_1}_{2\pi} \text{  and  } \sum_{k=1}^q\abs{a_2\delta_{k-1}+b_{21}^{(k)}}\ge C_3\abs s.\]
By Tur\'{a}n's minimax theorem (analogously to Case 1.1) there exists an integer $\nu\in\{1,\dots,2q\}$ such that
\begin{align*}
\abs[\bigg]{\sum_{k=1}^q&\parens[\big]{a_2\delta_{k-1}+b_{21}^{(k)}}\sin(\nu\beta_k)}\\
&=\abs[\bigg]{\sum_{k=1}^q\parens[\bigg]{\frac{\iu(a_2\delta_{k-1}+b_{21}^{(k)})}2\euler^{-\iu\nu\beta_k}+\frac{-\iu(a_2\delta_{k-1}+b_{21}^{(k)})}2\euler^{\iu\nu\beta_k}}}\\
&\ge\mu_1\sum_{k=1}^q\abs[\big]{a_2\delta_{k-1}+b_{21}^{(k)}}.
\end{align*}
We have
\begin{align*}
\text{LHS}&\overset{\eqref{eq:LemmaTaylor}}\ge \abs[\bigg]{\sum_{k=1}^q\parens[\big]{a_2\delta_{k-1}+b_{21}^{(k)}}\sin(\nu\beta_k)}-\frac32\abs{a_2}\nu\abs{\alpha_1-\beta_1}_{2\pi}\abs{\cos(\nu\beta_1)}-\nu\abs s\\
&\ge\parens[\Big]{\frac{\mu_1}2+\frac{\mu_1}2}\sum_{k=1}^q\abs[\big]{a_2\delta_{k-1}+b_{21}^{(k)}}-3q\abs{a_2}\abs{\alpha_1-\beta_1}_{2\pi}-2q\abs s\\
&\ge\abs s\\
&=\frac1{\sqrt 2}\norm S.
\end{align*}
\end{caseinside}
\begin{caseinside}[Case 2.2: $\sum_{k=1}^q\abs{a_2\delta_{k-1}+b_{21}^{(k)}}\le C_2\abs{a_2}\abs{\alpha_1-\beta_1}_{2\pi}$ and \\
$\phantom{\qquad \qquad} \max\{\abs{a_2}\abs{\alpha_1-\beta_1}_{2\pi},\sum_{k=1}^q\abs{a_2\delta_{k-1}+b_{21}^{(k)}}\}\ge C_3\abs s$]\ \\
By Tur\'{a}n's minimax theorem there exists an integer $\nu\in\{N_1-1,N_1\}$ such that
\[\abs{\cos(\nu\beta_1)}=\abs[\Big]{\tfrac12\euler^{\iu\nu\beta_1}+\tfrac12\euler^{-\iu\nu\beta_1}}\ge\frac{\delta_1}{4\pi}.\]
We have
\begin{align*}
\text{LHS}&\overset{\eqref{eq:LemmaTaylor}}\ge\frac12\abs{a_2}\abs{\cos(\nu\beta_1)}\nu\abs{\alpha_1-\beta_1}_{2\pi}-\abs[\bigg]{\sum_{k=1}^q(a_2\delta_{k-1}+b_{21}^{(k)})\sin(\nu\beta_k)}-\nu\abs s\\
&\ge\parens[\bigg]{\frac{\delta_1(N_1-1)}{16\pi}+\frac{\delta_1\nu}{16\pi}}\abs{a_2}\abs{\alpha_1-\beta_1}_{2\pi}-\sum_{k=1}^q\abs[\big]{a_2\delta_{k-1}+b_{21}^{(k)}}-\nu\abs s\\
&\ge\nu\abs s\\
&\ge\frac1{\sqrt 2}\norm S.
\end{align*}
\end{caseinside}
\begin{caseinside}[Case 2.3: $\max\{\abs{a_2}\abs{\alpha_1-\beta_1}_{2\pi},\sum_{k=1}^q\abs{a_2\delta_{k-1}+b_{21}^{(k)}}\}\le C_3\abs s$]\ \\
By Definition of $q_1$ we have
\[\abs{\cos(q_1\beta_1)}=\abs{\sin(q_1\beta_1+\tfrac\pi2)}\le\abs{q_1\beta_1+\tfrac\pi2}_\pi=\pi\abs{\tfrac{q_1\beta_1}\pi+\tfrac12}_1\le\tfrac1{3C_3}.\]
So we have
\begin{align*}
\text{LHS}&\overset{\eqref{eq:LemmaTaylor}}\ge q_1\abs s-\tfrac32\abs{a_2}\abs{\cos(q_1\beta_1)}q_1\abs{\alpha_1-\beta_1}_{2\pi}-\abs[\bigg]{\sum_{k=1}^q\parens[\big]{a_2\delta_{k-1}+b_{21}^{(k)}}\sin(q_1\beta_k)}\\
&\ge\parens[\Big]{1+\frac{q_1}2+C_3}\abs s-\frac{q_1}{2C_3}\abs{a_2}\abs{\alpha_1-\beta_1}_{2\pi}-\sum_{k=1}^q\abs[\big]{a_2\delta_{k-1}+b_{21}^{(k)}}\\
&\ge\frac1{\sqrt 2}\norm S.
\end{align*}
\end{caseinside}
\end{case}
\begin{case}[Case 3: $\exists\,i\in\{1,2\} : (\abs{\alpha_i-\beta}_{2\pi}\ge\delta_2\quad\forall\,\beta\in\{0,\pi,\pm\beta_1,\dots,\pm\beta_q\})$]\ \\
Without loss of generality let $i=1$.
\begin{caseinside}[Case 3.1: $\abs{a_2}+\sum_{k=1}^q\abs{b_{21}^{(k)}}\ge C_4\abs s$]$\,$\\
By Definition of $\delta_3$ we have
\begin{align*}
\min_{\substack{\gamma_1,\gamma_2\in\{\pm\alpha_1,\pm\beta_1,\dots,\pm\beta_q\}\\ \gamma_1\neq\gamma_2}}\abs[\big]{\euler^{\iu\gamma_1}-\euler^{\iu\gamma_2}}&\ge\min_{\substack{\gamma_1,\gamma_2\in\{\pm\alpha_1,\pm\beta_1,\dots,\pm\beta_q\}\\\gamma_1\neq\gamma_2}}\frac{\abs{\gamma_1-\gamma_2}_{2\pi}}\pi\\
&\ge\frac{\min\{\delta_1,\delta_2\}}\pi\\
&=\frac{\delta_3}\pi.
\end{align*}
By Tur\'{a}n's minimax theorem there exists an integer $\nu\in\{1,\dots,2q+2\}$ such that
\begin{align*}
\abs[\bigg]{a_2&\sin(\nu\alpha_1)+\sum_{k=1}^qb_{21}^{(k)}\sin(\nu\beta_k)}\\
&=\abs[\bigg]{\frac{\iu a_2}2\euler^{-\iu\nu\alpha_1}+\frac{-\iu a_2}2\euler^{\iu\nu\alpha_1}+\sum_{k=1}^q\parens[\bigg]{\frac{\iu b_{21}^{(k)}}2\euler^{-\iu\nu\beta_k}+\frac{-\iu b_{21}^{(k)}}2\euler^{\iu\nu\beta_k}}}\\
&\ge\mu_2\parens[\bigg]{\abs{a_2}+\sum_{k=1}^q\abs[\big]{b_{21}^{(k)}}}.
\end{align*}
We have
\begin{align*}
\text{LHS}&\ge\abs[\bigg]{a_2\sin(\nu\alpha_1)+\sum_{k=1}^qb_{21}^{(k)}\sin(\nu\beta_k)}-\nu\abs s\\
&\ge\mu_2\parens[\bigg]{\abs{a_2}+\sum_{k=1}^q\abs[\big]{b_{21}^{(k)}}}-(2q+2)\abs s\\
&\ge\abs s\\
&=\frac1{\sqrt 2}\norm S.
\end{align*}
\end{caseinside}
\begin{caseinside}[Case 3.2: $\abs{a_2}+\sum_{k=1}^q\abs{b_{21}^{(k)}}\le C_4\abs s$]\ \\
We have
\begin{align*}
\text{LHS}&\ge N\abs s-\abs{a_2\sin(N\alpha_1)}-\abs[\bigg]{\sum_{k=1}^qb_{21}^{(k)}\sin(N\beta_k)}\\
&\ge N\abs s-\parens[\bigg]{\abs{a_2}+\sum_{k=1}^q\abs[\big]{b_{21}^{(k)}}}\\
&\ge\abs s\\
&=\frac1{\sqrt 2}\norm S.
\end{align*}
\end{caseinside}
\end{case}
Since Case 2 and Case 3 include the case that
\[\exists\, i\in\{1,2\} : ((\abs{\alpha_i}_{2\pi}\ge\delta_2)\land(\abs{\alpha_i-\pi}_{2\pi}\ge\delta_2)),\]
the assertion is proven.
\end{proof}
\begin{Theorem}[A discrete Korn inequality]\label{Theorem:Korn}
Suppose that $\RR\subset\G$ is an admissible neighborhood range of $\id$.
Then the two seminorms $\norm\fdot_{\RR}$ and $\zeronorm\fdot{\RR}$ are equivalent.
\end{Theorem}
\begin{proof}
First we show the trivial inequality $\norm\fdot_{\RR}\le\zeronorm\fdot{\RR}$:\\
Let $u\in\UPer$.
Let $N\in \MM$ be such that $u$ is $\T^N$-periodic.
Since $\zeroUiso\RR\subset\Uiso\RR$, we have
\begin{align*}
\norm u_{\RR}^2&=\frac1{\abs{\CC_N}}\sum_{g\in\CC_N}\norm{\pi_{\Uiso\RR}(u(g\fdot)|_\RR)}^2\\
&\le\frac1{\abs{\CC_N}}\sum_{g\in\CC_N}\norm{\pi_{\zeroUiso\RR}(u(g\fdot)|_\RR)}^2\\
&=\zeronorm u{\RR}^2.
\end{align*}
Now we show with the aid of the Plancherel formula that there exists a constant $c>0$ such that $\norm\fdot_{\RR}\ge c\zeronorm\fdot{\RR}$:\\
We choose $m=m_0$ such that $\MM=m\N$ and the group $\T^m$ is isomorphic to $\Z^{d_2}$, see Section~\ref{subsection:structure}. In particular, there exist $t_1,\dots,t_{d_2}\in\T^m$ such that $\{t_1,\dots,t_{d_2}\}$ generates $\T^m$.
Since $\rot(\T^m)$ is a subgroup of $\{I_{d-\daff}\}\oplus\O(\daff-d_2)\oplus\{I_{d_2}\}$ and the elements $t_1,\dots,t_{d_2}$ commute, by Theorem~\ref{Theorem:simultaneousdiagonalisation} we may without loss of generality (by a coordinate transformation) assume that for all $i\in\{1,\dots,d_2\}$ there exist an integer $q_i\in\{0,\dots,\floor{(\daff-d_2)/2}\}$, a vector $v_i\in\{\pm1\}^{\daff-d_2-2q_i}$ and angles $\theta_{i,1},\dots,\theta_{i,q_i}\in[0,2\pi)$ such that
\[\rot(t_i)=I_{d-\daff}\oplus\diag(v_i)\oplus R(\theta_{i,q_i})\oplus\dots\oplus R(\theta_{i,1})\oplus I_{d_2}.\]
By Lemma~\ref{Lemma:helpTheorem} there exists an integer $N_0\in\N$ such that
\begin{equation}\label{eq:helpTheorem}
\max_{n\in\{1,\dots,N_0\}}\norm[\bigg]{a\otimes(\sin(n\alpha_1),\dots,\sin(n\alpha_{d_2}))-\sum_{i=1}^{d_2}\sum_{j=1}^{q_i}\sin(n\theta_{i,j})B_{i,j}-nS}\ge \norm S
\end{equation}
for all $a\in\C^{d_2}$, $\alpha_1,\dots,\alpha_{d_2}\in[0,2\pi)$, $B_{1,1},\dots,B_{d_2,q_{d_2}}\in\C^{d_2\times d_2}$, and $S\in\Skew(d_2,\C)$.
Let $\RR_0=\set{t_i^n}{i\in\{1,\dots,d_2\},n\in \{\pm1,\dots,\pm N_0\}}\subset \T^m$.
Since $\zeronorm\fdot{\RR\cup\RR_0\G_{x_0}}\ge\zeronorm\fdot\RR$ and by Theorem~\ref{Theorem:equivalence}, we may without loss of generality assume that $\RR_0\G_{x_0}\subset\RR$.
For all finite sets $\RR'\subset\G$ we define the map
\begin{align*}
g_{\RR'}\colon&\Skew(d,\C)\to\C^{d\times\abs{\RR'}}\\
&S\mapsto(\rot(h)^TS(h\cdot x_0-x_0))_{h\in\RR'}.
\end{align*}

Recall the definition of the dual space $\dual{\T^m}$ from Section~\ref{subsection:structure}. Now we show that there exists a constant $c_0>0$ such that
\begin{equation}\label{eq:Theoremstar}
\norm[\big]{\parens[\big]{\chi(h)^{-1}v-\rot(h)^Tv}_{h\in\RR_0}-g_{\RR_0}(S)}\ge c_0\norm{S_3}
\end{equation}
for all $\chi\in\dual{\T^m}$, $v\in\C^d$ and $S=\parens[\Big]{\begin{smallmatrix}S_1&-S_2^T\\S_2 & S_3\end{smallmatrix}}\in\Skew(d_1+d_2,\C)$.

Writing $v=\parens[\Big]{\begin{matrix}v_1\\v_2\end{matrix}}\in\C^{d_1+d_2}$ we have
\begin{align}\label{eq:lhs}
\text{LHS}&:=\norm[\Big]{\parens[\Big]{\chi(h)^{-1}v-\rot(h)^Tv}_{h\in\RR_0}-g_{\RR_0}(S)}\nonumber\\
&\ge\norm[\Big]{\parens[\Big]{\chi(h^{-1})v_2-v_2-(S_2,S_3)(h\gdot x_0-x_0)}_{h\in\RR_0}}\nonumber\\
\begin{split}
&\ge\frac1{\sqrt2}\parens[\bigg]{\norm[\Big]{\parens[\Big]{\chi(t_i^{-n})v_2-v_2-(S_2,S_3)(t_i^n\gdot x_0-x_0)}_{i\in\{1,\dots,d_2\}}}\\
&\qquad+\norm[\Big]{\parens[\Big]{\chi(t_i^n)v_2-v_2-(S_2,S_3)(t_i^{-n}\gdot x_0-x_0)}_{i\in\{1,\dots,d_2\}}}}\\
\end{split}\nonumber\\
&\ge\frac1{\sqrt2}\norm[\Big]{\parens[\Big]{(\chi(t_i^n)-\chi(t_i^{-n}))v_2+(S_2,S_3)(t_i^n\gdot x_0-t_i^{-n}\gdot x_0)}_{i\in\{1,\dots,d_2\}}}
\end{align}
for all $n\in\{1,\dots,N_0\}$.
For all $j\in\{1,\dots,d_2\}$ we define $\alpha_j\in[0,2\pi)$ by $\euler^{\iu\alpha_j}=\chi(t_j)$.
Let $\xzeroone\in\R^{d_1}$ and $\xzerotwo\in\R^{d_2}$ be such that $x_0=\parens[\Big]{\begin{matrix}\xzeroone\\ \xzerotwo\end{matrix}}$.
For all $j\in\{1,\dots,\max\{q_1,\dots,q_{d_2}\}\}$ we define $n_j=d_1-2j$, $m_j=2j-2$ and
\[b_j=S_2(0_{n_j,n_j}\oplus(\begin{smallmatrix}0&-2\\2&0\end{smallmatrix})\oplus 0_{m_j,m_j})\xzeroone\in\C^{d_2}.\]
Let $\transtwo\colon\T^m\to\R^{d_2}$ be uniquely defined by the condition $\trans(t)=\parens[\Big]{\begin{matrix}0_{d_1}\\ \transtwo(t)\end{matrix}}$ for all $t\in\T^m$.
Then for all $i\in\{1,\dots,d_2\}$ and $n\in\{1,\dots,N_0\}$ we have
\begin{align*}
(S_2&,S_3)(t_i^{n}\gdot x_0-t_i^{-n}\gdot x_0)\\
&=S_2\parens[\big]{0_{d_1-2q_i,d_1-2q_i}\oplus(R(n\theta_{i,q_i})-R(-n\theta_{i,q_i}))\oplus\dots\oplus(R(n\theta_{i,1})-R(-n\theta_{i,1}))}\xzeroone\\
&\qquad +2nS_3\transtwo(t_i)\\
&=\sum_{j=1}^{q_i}\sin(n\theta_{i,j})S_2\parens[\big]{0_{n_j,n_j}\oplus(\begin{smallmatrix}0&-2\\2&0\end{smallmatrix})\oplus 0_{m_j,m_j}}\xzeroone+2nS_3\transtwo(t_i)\\
&=\sum_{j=1}^{q_i}\sin(n\theta_{i,j})b_j+2nS_3\transtwo(t_i).
\end{align*}
For all $i\in\{1,\dots,d_2\}$ and $j\in\{1,\dots,q_i\}$ we define $B_{i,j}=-b_j\otimes e_i^T\in\C^{d_2\times d_2}$.
Let $T=2\parens[\big]{\transtwo(t_1),\dots,\transtwo(t_{d_2})}\in\GL(d_2)$.
By equation~\eqref{eq:lhs} for all $n\in\{1,\dots,N_0\}$ we have
\begin{align*}
\text{LHS}&\ge\frac1{\sqrt 2}\norm[\Big]{2\iu v_2\otimes(\sin(n\alpha_1),\dots,\sin(n\alpha_{d_2}))-\sum_{i=1}^{d_2}\sum_{j=1}^{q_i}\sin(n\theta_{i,j})B_{i,j}+nS_3T}\\
&\ge c_1\norm[\Big]{(2\iu T^Tv_2)\otimes(\sin(n\alpha_1),\dots,\sin(n\alpha_{d_2}))-\sum_{i=1}^{d_2}\sum_{j=1}^{q_i}\sin(n\theta_{i,j})T^TB_{i,j}+nT^TS_3T},
\end{align*}
where $c_1=\sigma_{\min}(T^{-T})/\sqrt2>0$, $\sigma_{\min}(M)$ denotes the minimum singular value of a matrix $M$ and we used Theorem~\ref{Theorem:singularvalueinequality} in the last step.
With equation~\eqref{eq:helpTheorem} it follows
\[\text{LHS}\ge c_1\norm{-T^TS_3T}\ge c_0\norm{S_3},\]
where $c_0=\sigma_{\min}(T)^2c_1>0$.

By Propositions~\ref{Proposition:nablanormequivalent-ohne-0} and~\ref{Proposition:nablanormequivalent-0} it suffices to show that there exists a constant $c>0$ such that $\nablanorm\fdot\RR\ge c\nablazeronorm\fdot\RR$.
Let $u\in\UPer$.
Let $N\in \MM$ be such that $u$ is $\T^N$-periodic.
In particular, $m$ divides $N$.
Let $v\colon\G\to\Skew(d)$ be $\T^N$-periodic such that $\pi_{\Urot\RR}(\nabla_\RR u(g))=\nabla_\RR u(g)-g_\RR\circ v(g)$ for all $g\in\G$.
Let
\[v_1\colon\G\to\set[\bigg]{\parens[\bigg]{\begin{matrix}S_1&S_2\\-S_2^T&0\end{matrix}}}{S_1\in\Skew(d_1),S_2\in\R^{d_1\times d_2}}\]
and 
\[v_2\colon\G\to\set{0_{d_1,d_1}\oplus S}{S\in\Skew(d_2)}\]
such that $v=v_1+v_2$.
For all $g\in\CC_m$ we define the functions
\begin{align*}
u_g&\colon\T^m\to\C^d,\;t\mapsto\proj{u}(gt)\\
v_g&\colon\T^m\to\Skew(d,\C),\; t\mapsto v(gt)\\
v_{1,g}&\colon\T^m\to\Skew(d,\C),\;t\mapsto v_1(gt)
\qquad\text{and}\\
v_{2,g}&\colon\T^m\to\Skew(d,\C),\;t\mapsto v_2(gt).
\end{align*}
Let $\EE=\set{\chi\in\dual{\T^m}}{\chi\text{ is periodic}}$.
For all $g\in\CC_m$ and $\chi\in\EE$ it holds
\begin{align*}
&\fourier{v_g}(\chi)=\fourier{v_{1,g}}(\chi)+\fourier{v_{2,g}}(\chi),\\
&\fourier{v_{1,g}}(\chi)\in\set[\bigg]{\parens[\bigg]{\begin{matrix}S_1&S_2\\-S_2^T&0\end{matrix}}}{S_1\in\Skew(d_1,\C),S_2\in\C^{d_1\times d_2}}
\shortintertext{and}
&\fourier{v_{2,g}}(\chi)\in\set{0_{d_1,d_1}\oplus S}{S\in\Skew(d_2,\C)}.
\end{align*}
We have
\begin{align}\label{eq:as}
\nablanorm u\RR^2&=\frac1{\abs{\CC_N}}\sum_{(g,t)\in\CC_m\times(\T^m\cap\CC_N)}\norm{\pi_{\Urot\RR}(\nabla_\RR u(gt))}^2\nonumber\\
&=\frac1{\abs{\CC_N}}\sum_{g\in\CC_m}\sum_{t\in\T^m\cap\CC_N}\norm{\nabla_\RR u(gt)-g_\RR\circ v(gt)}^2\nonumber\\
&\ge\frac1{\abs{\CC_N}}\sum_{g\in\CC_m}\sum_{t\in\T^m\cap\CC_N}\norm{\nabla_{\RR_0\G_{x_0}}u(gt)-g_{\RR_0\G_{x_0}}\circ v(gt)}^2\nonumber\\
&\ge\frac1{\abs{\CC_N}}\sum_{g\in\CC_m}\sum_{t\in\T^m\cap\CC_N}\norm[\big]{\parens[\big]{u_g(th)-\rot(h)^Tu_g(t)}_{h\in\RR_0}-g_{\RR_0}\circ v_g(t)}^2\nonumber\\
&=\frac1{\abs{\CC_N}}\sum_{g\in\CC_m}\abs{\T^m\cap\CC_N}\sum_{\chi\in\EE}\norm[\big]{\parens[\big]{\chi(h)^{-1}\fourier{u_g}(\chi)-\rot(h)^T\fourier{u_g}(\chi)}_{h\in\RR_0}-g_{\RR_0}\circ\fourier{v_g}(\chi)}^2\nonumber\\
&\ge\frac{c_0^2}{\abs{\CC_N}}\sum_{g\in\CC_m}\abs{\T^m\cap\CC_N}\sum_{\chi\in\EE}\norm{\fourier{v_{2,g}}(\chi)}^2\nonumber\\
&=\frac{c_0^2}{\abs{\CC_N}}\sum_{g\in\CC_m}\sum_{t\in\T^m\cap\CC_N}\norm{v_{2,g}(t)}^2\nonumber\\
&=\frac{c_0^2}{\abs{\CC_N}}\sum_{(g,t)\in\CC_m\times(\T^m\cap\CC_N)}\norm{v_2(gt)}^2\nonumber\\
&=c_0^2\norm{v_2}_2^2.
\end{align}
In the first and last step we used that the set $\bigcup_{(g,t)\in\CC_m\times(\T^m\cap\CC_N)}\{gt\}$ is a representation set of $\G/\T^N$.
In the fifth and seventh step we used Proposition~\ref{Proposition:TFplancherelmatrix} for the group $\T^m$ and $\T^N$-periodic functions and Lemma~\ref{Lemma:PeriodicTranslation}.
Note that $\T^m\cap\CC_N$ is a representation set of $\T^m/\T^N$.
In the sixth step we used \eqref{eq:Theoremstar}.
Let $C=\abs\RR\max\set{\norm{h\cdot x_0-x_0}}{h\in\RR}$.
We have
\begin{align}\label{eq:ar}
\nablanorm u\RR^2&=\frac1{\abs{\CC_N}}\sum_{g\in\CC_N}\norm{\nabla_\RR u(g)-g_\RR\circ v(g)}^2\nonumber\\
&\ge\frac1{\abs{\CC_N}}\sum_{g\in\CC_N}\parens[\bigg]{\frac12\norm{\nabla_\RR u(g)-g_\RR\circ v_1(g)}^2-\norm{g_\RR\circ v_2(g)}^2}\nonumber\\
&\ge\frac1{\abs{\CC_N}}\sum_{g\in\CC_N}\parens[\bigg]{\frac12\norm{\pi_{\zeroUrot\RR}(\nabla_\RR u(g)))}^2-C\norm{v_2(g)}^2}\nonumber\\
&=\frac12\nablazeronorm u\RR^2-C\norm{v_2}_2^2,
\end{align}
where in the second step we used that $(a-b)^2\ge a^2/2-b^2$ for all $a,b\ge0$.
Let $c_2=\min\{1/2,c_0^2/(2C)\}$.
By \eqref{eq:as} and \eqref{eq:ar} we have
\begin{align*}
\nablanorm u\RR^2&\ge\frac12\nablanorm u\RR^2+c_2\nablanorm u\RR^2\\
&\ge\frac{c_0^2}2\norm{v_2}_2^2+c_2\parens[\Big]{\frac12\nablazeronorm u\RR^2-C\norm{v_2}_2^2}\\
&\ge\frac{c_2}2\nablazeronorm u\RR^2.
\end{align*}
Thus, we have $\nablanorm\fdot\RR\ge\sqrt{c_2/2}\nablazeronorm\fdot\RR$.
\end{proof}

%-------------------------------------------------------------------
%-------------------------------------------------------------------
\subsection{Seminorm kernels}
\label{subsection:seminorms-kernels}
%-------------------------------------------------------------------
%
It is interesting to explicitly describe the kernel of the seminorms that measure the rigidity of deformations as this entails a characterization of fully rigid deformations.  

Recall from Definition~\ref{Definition:UtransUrot} that $\UTrans$ is the vector space of displacements corresponding to translations. We now introduce the vector space $\URot$ which corresponds to infinitesimal rotations of $\G\cdot x_0$ about the affine subspace $x_0+\{0_{d_1}\}\times\R^{d_2}$.
\begin{Definition}\label{Definition:Unew}
For all $\RR\subset\G$ such that $\RR\G_{x_0}=\RR$ we define the vector spaces
\begin{align*}
&\Unewrot{\RR}\\
&:=\set[\Big]{u\colon\RR\to\R^d}{\exists\+ S\in\Skew(d_1)\;\forall g\in\RR:\rot(g)\proj{u}(g)=(S\oplus0_{d_2,d_2})(g\gdot x_0-x_0)}\\
&\subset\zeroUrot\RR\cap L^\infty(\G,\R^{d_1}\times\{0_{d_2}\})
\shortintertext{and}
&\Unewiso{\RR}:=\Utrans{\RR}+\Unewrot{\RR}\subset\zeroUiso\RR\cap L^\infty(\G,\R^d)
\end{align*}
with $\Utrans{\RR}$ as in Definition~\ref{Definition:UtransUrot}. In case $\RR=\G$ we suppress the argument $\RR$ for brevity and simply write $\UTrans$, $\URot$ and $\UIso$, respectively. 
\end{Definition}
\begin{Remark}
We have $\URot\subset\zeroUrot\G$.
If $d_1\ge1$ and $d_2\ge1$, then we have $\URot\subsetneq\zeroUrot\G$.
Moreover, in general we have $\UTrans\not\subset\UPer$ and $\URot\not\subset\UPer$.
For example let $\alpha\in\R\setminus(2\pi\Q)$, $R(\alpha)$ be the rotation matrix by the angle $\alpha$, $\G=\angles{\gplus{R(\alpha)}{\iso{I_1}{1}}}<\E(3)$ and $x_0=e_1$.
Then we have $\dim(\zeroUrot\G)=3$, $\dim(\URot)=1$ and $\dim(\URot\cap\UPer)=0$.
Moreover, we have $\dim(\UTrans)=3$ and $\dim(\UTrans\cap\UPer)=1$.
\end{Remark}
\begin{Example}\label{Example:SpaceGroupURot}
Suppose that $\G_{x_0}$ is trivial. 
If $d_1=1$ or $\daff=d_2$, then we have $\URot=\{0\}$.
In particular, if $\G$ is a space group, then we have $\URot=\{0\}$.
\end{Example}
The following proposition characterizes the vector spaces $\Utrans\RR$, $\Urot\RR$, $\zeroUrot\RR$, $\Unewrot\RR$, $\Uiso\RR$, $\zeroUiso\RR$ and $\Unewiso\RR$ for appropriate $\RR\subset\G$.
In particular, the proposition characterizes $\UTrans$, $\URot$ and $\UIso$.

In the following two results we write $\pi\colon\{u\colon\RR\to\R^d\}\to\{u\colon\RR\to\R^d\}$, $\pi(u) = \proj{u}$ for the projection defined by \eqref{eq:pi-proj}. Note that, by construction, all the sets $\Utrans\RR$, $\Urot\RR$, $\zeroUrot\RR$, $\Unewrot\RR$, $\Uiso\RR$, $\zeroUiso\RR$ and $\Unewiso\RR$ are invariant under $\pi$. 
\begin{Proposition}\label{Proposition:UtransUrotbig}
Suppose that $\RR\subset\G$ is such that $\RR\G_{x_0}=\RR$, $\id\in\RR$ and $\aff(\RR\cdot x_0)=\aff(\G\cdot x_0)$. Then the maps
\begin{align*}
&\begin{aligned}\phi_1\colon&\R^d\to\pi({\Utrans\RR})\\
&a\mapsto\parens[\big]{\RR\to\R^d,g\mapsto\rot(g)^Ta},\end{aligned}\\
&\begin{aligned}\phi_2\colon&\R^{d_3\times\daff}\times\Skew(\daff)\to\pi(\Urot\RR)\\
&(A_1,A_2)\mapsto \parens[\Big]{\RR\to\R^d,g\mapsto\rot(g)^T\parens[\Big]{\begin{smallmatrix}0&A_1\\-A_1^T&A_2\end{smallmatrix}}(g\gdot x_0-x_0)},\end{aligned}\\
&\begin{aligned}\phi_3\colon&\R^{d_3\times d_4}\times\R^{d_3\times d_2}\times\Skew(d_4)\times\R^{d_4\times d_2}\to\pi(\zeroUrot\RR)\\
&(A_1,A_2,A_3,A_4)\mapsto \parens[\bigg]{\RR\to\R^d,g\mapsto\rot(g)^T\parens[\bigg]{\begin{smallmatrix}0&A_1&A_2\\-A_1^T&A_3&A_4\\-A_2^T&-A_4^T&0\end{smallmatrix}}(g\gdot x_0-x_0)},\end{aligned}
\shortintertext{and}
&\begin{aligned}\phi_4\colon&\R^{d_3\times d_4}\times\Skew(d_4)\to\pi(\Unewrot\RR)\\
&(A_1,A_2)\mapsto \parens[\Big]{\RR\to\R^d,g\mapsto\rot(g)^T\parens[\Big]{\parens[\Big]{\begin{smallmatrix}0&A_1\\-A_1^T&A_2\end{smallmatrix}}\oplus0_{d_2,d_2}}(g\gdot x_0-x_0)}\end{aligned}
\end{align*}
are isomorphisms, where $d_3=d-\daff$, $d_4=\daff-d_2$. 
In particular, we have
\begin{align*}
\dim(\pi(\Utrans{\RR}))&=d\\
\dim(\pi(\Urot{\RR}))&=\daff(d-\tfrac12\daff-\tfrac12),\\
\dim(\pi(\zeroUrot{\RR}))&=d_3\daff+\tfrac12d_4(\daff+d_2-1)
\text{ and}\\
\dim(\pi(\Unewrot\RR))&=d_4(d_3+d_1-1)/2.
\end{align*}
Moreover we have 
\begin{align*}
\pi(\Uiso\RR)&=\pi(\Utrans\RR)\oplus\pi(\Urot\RR),\\
\pi(\zeroUiso\RR)&=\pi(\Utrans\RR)\oplus\pi(\zeroUrot\RR)
\text{ and}\\
\pi(\Unewiso\RR)&=\pi(\Utrans\RR)\oplus\pi(\Unewrot\RR).
\end{align*}
\end{Proposition}
We include the elementary proof for the sake of completeness. 
\begin{proof}
Since $\rot(\id)=I_d$, the map $\phi_1$ is injective and thus, an isomorphism.

Now we prove that $\phi_3$ is an isomorphism.
The map $\phi_3$ is well-defined and linear.
First we show that $\phi_3$ is surjective.
Let $u\in\pi(\zeroUrot\RR)$.
There exist some $A_1\in\Skew(d_1)$ and $A_2\in\R^{d_1\times d_2}$ such that
\[\rot(g)u(g)=\parens[\Big]{\begin{smallmatrix}A_1&A_2\\-A_2^T&0\end{smallmatrix}}(g\gdot x_0-x_0)\quad\text{for all }g\in\G.\]
Let $A_3\in\Skew(d_3)$, $A_4\in\R^{d_3\times d_4}$, $A_5\in\Skew(d_4)$, $A_6\in\R^{d_3\times d_2}$ and $A_7\in\R^{d_4\times d_2}$ be such that
\[A_1=\parens[\Big]{\begin{smallmatrix}A_3&A_4\\-A_4^T&A_5\end{smallmatrix}}\text{ and }A_2=\parens[\Big]{\begin{smallmatrix}A_6\\ A_7\end{smallmatrix}}.\]
Since $\G\gdot x_0\subset\{0_{d_3}\}\times\R^{\daff}$, we have $\phi_3(A_4,A_6,A_5,A_7)=u$.

Now we show that $\phi_3$ is injective.
Let $A_1,B_1\in\R^{d_3\times d_4}$, $A_2,B_2\in\R^{d_3\times d_2}$, $A_3,B_3\in\Skew(d_4)$ and $A_4,B_4\in\R^{d_4\times d_2}$ be such that $\phi_3(A_1,A_2,A_3,A_4)=\phi_3(B_1,B_2,B_3,B_4)$.
Let $\RR'\subset\RR$ be finite with $\id\in\RR'$ and $\aff(\RR'\cdot x_0)=\aff(\G\cdot x_0)$.
By Lemma~\ref{Lemma:MatrixRank} there exists some $C\in\R^{\daff\times\abs{\RR'}}$ of rank $\daff$ such that
\[(g\gdot x_0-x_0)_{g\in\RR'}=\parens[\Big]{\begin{smallmatrix}0\\C\end{smallmatrix}}.\]
The identity $\phi_3(A_1,A_2,A_3,A_4)=\phi_3(B_1,B_2,B_3,B_4)$ implies
\[\parens[\bigg]{\begin{smallmatrix}0&A_1&A_2\\-A_1^T&A_3&A_4\\-A_2^T&-A_4^T&0\end{smallmatrix}}(g\gdot x_0-x_0)=\parens[\bigg]{\begin{smallmatrix}0&B_1&B_2\\-B_1^T&B_3&B_4\\-B_2^T&-B_4^T&0\end{smallmatrix}}(g\gdot x_0-x_0)\]
for all $g\in\RR$ and in particular, we have
\[\parens[\Big]{\begin{smallmatrix}(\begin{smallmatrix}A_1&A_2\end{smallmatrix})C\\(\begin{smallmatrix}A_3&A_4\end{smallmatrix})C\end{smallmatrix}}=\parens[\Big]{\begin{smallmatrix}(\begin{smallmatrix}B_1&B_2\end{smallmatrix})C\\(\begin{smallmatrix}B_3&B_4\end{smallmatrix})C\end{smallmatrix}}.\]
Since the rank of $C$ is equal to the number of its rows, we have $A_i=B_i$ for all $i\in\{1,\dots,4\}$.

The proofs that $\phi_2$ and $\phi_4$ are isomorphisms are analogous.

For all $u\in\pi(\Urot\RR)$ we have $u(\id)=0$ and for all $u\in\pi(\Utrans\RR)$ and $g\in\RR$ we have $\rot(g)u(g)=u(\id)$.
This implies $\pi(\Utrans\RR)\cap\pi(\Urot\RR)=\{0\}$ and thus $\pi(\Uiso\RR)=\pi(\Utrans\RR)\oplus\pi(\Urot\RR)$.
Analogously, we have $\pi(\zeroUiso\RR)=\pi(\Utrans\RR)\oplus\pi(\zeroUrot\RR)$ and $\pi(\Unewiso\RR)=\pi(\Utrans\RR)\oplus\pi(\Unewrot\RR)$.
\end{proof}
\begin{Lemma}\label{Lemma:finiterotG}
If the group $\rot(\G)$ is finite, then we have $\pi(\UIso)\subset\UPer$.
\end{Lemma}
\begin{proof}
Suppose that $\rot(\G)$ is finite.
Let $n=\abs{\rot(\G)}$.
For all $g\in\G$ we have
\begin{equation}\label{eq:helppppp}
\rot(g)^n=I_d.
\end{equation}
Choose $N=m_0 n$.
Let $u\in\pi(\UIso)$.
By definition there exist some $a\in\R^d$ and $S\in\Skew(d_1)$ such that
\[\rot(g)u(g)=a+(S\oplus0)(g\gdot x_0 - x_0)\text{ for all }g\in\G.\]
For all $g\in\G$ and $t\in\T$ we have
\begin{align*}
u(gt^N)&=\rot(gt^N)^{-1}\parens[\big]{a+(S\oplus0)((gt^N)\gdot x_0-x_0)}\\
&=\rot(t)^{-N}\rot(g)^{-1}\parens[\big]{a+(S\oplus0)(g\gdot (\rot(t)^Nx_0)-x_0)+(S\oplus0)\rot(g)\trans(t^N)}\\
&=\rot(g)^{-1}\parens[\big]{a+(S\oplus0)(g\gdot x_0-x_0)}\\
&=u(g),
\end{align*}
where we used \eqref{eq:helppppp}, that $\rot(\G)<\O(d_1)\oplus\O(d_2)$ and that $\trans(\G)\subset\{0_{d_1}\}\times\R^{d_2}$ in the second to last step.
Thus, $u$ is $\T^N$-periodic and we have $u\in\UPer$.
\end{proof}
The following theorem characterizes the kernel of the seminorm $\norm\fdot_\RR$.
\begin{Theorem}\label{Theorem:kernelseminorm}
Suppose that $\RR\subset\G$ is an admissible neighborhood range of $\id$.
Then we have
\[\ker(\norm\fdot_{\RR})=\UIso\cap\UPer.\]
\end{Theorem}
\begin{proof}
First we show that $\UIso\cap\UPer\subset\ker(\norm\fdot_{\RR})$:\\
Let $u\in\UIso\cap\UPer$.
There exist some $a\in\R^d$ and $S\in\Skew(d)$ such that
\[\rot(g)\proj{u}(g)=a+S(g\cdot x_0-x_0)\qquad\text{for all }g\in\G.\]

Let $g\in\G$.
For all $h\in\RR$ it holds
\begin{align*}
\rot(h)\proj{u}(gh)&=\rot(g)^Ta+\rot(g)^TS((gh)\cdot x_0-x_0)\\
&=\rot(g)^Ta+\rot(g)^TS(g\cdot x_0-x_0)+\rot(g)^TS\rot(g)(h\cdot x_0-x_0).
\end{align*}
Since $\rot(g)^TS\rot(g)\in\Skew(d)$, we have $u(g\fdot)|_\RR\in\Uiso\RR$.

Let $N\in \MM$ be such that $u$ is $\T^N$-periodic.
Since $g\in\G$ was arbitrary, we have
\[\norm u_\RR^2=\frac1{\abs{\CC_N}}\sum_{g\in\CC_N}\norm{\pi_{\Uiso\RR}(u(g\fdot)|_\RR)}^2=0.\]
Thus, we have $u\in\ker(\norm\fdot_\RR)$.

Now we show that $\ker(\norm\fdot_{\RR})\subset\UIso\cap\UPer$:\\
Let $u\in\ker(\norm\fdot_{\RR})$.
By definition of $\norm\fdot_{\RR}$ we have $u\in\UPer$.
Let $g\in\G$.
By Theorem~\ref{Theorem:equivalence} we have $u\in\ker(\norm\fdot_{\RR\cup g\G_{x_0}})$ and thus $u|_{\RR\cup g\G_{x_0}}\in\Uiso{\RR\cup g\G_{x_0}}$.
There exist some $a\in\R^d$ and $S\in\Skew(d)$ such that
\begin{equation}\label{eq:rtgv}
\rot(h)\proj{u}(h)=a+S(h\gdot x_0-x_0)\qquad\text{for all }h\in\RR\cup g\G_{x_0}.
\end{equation}
Since $\RR$ is admissible, it holds $\G_{x_0}\subset\RR$ and thus, $a=\proj{u}(\id)$.
In particular, the vector $a$ is independent of $g$.

By Lemma~\ref{Lemma:MatrixRank} there exists some $A\in\R^{\daff\times\abs\RR}$ of rank $\daff$ such that
\[(h\cdot x_0-x_0)_{h\in\RR/\G_{x_0}}=\parens[\bigg]{\begin{matrix}0_{d-\daff,\abs\RR}\\ A\end{matrix}}.\]
Since $\G\cdot x_0\subset\{0_{d-\daff}\}\times\R^\daff$, without loss of generality we may assume that
\[S=\parens[\bigg]{\begin{matrix}0&S_1\\-S_1^T&S_2\end{matrix}}\]
for some $S_1\in\R^{(d-\daff)\times \daff}$ and $S_2\in\Skew(\daff)$.
By equation~\eqref{eq:rtgv} we have
\begin{equation}\label{eq:afd}
\parens[\bigg]{\rot(h)\proj{u}(h)-a}_{h\in\RR/\G_{x_0}}=\parens[\bigg]{\begin{matrix}0&S_1\\-S_1^T&S_2\end{matrix}}\parens[\bigg]{\begin{matrix}0\\ A\end{matrix}}=\parens[\bigg]{\begin{matrix}S_1A\\S_2A\end{matrix}}.
\end{equation}
Since the rank of $A$ is equal to the number of its rows, by \eqref{eq:afd} the matrix $S$ is independent of $g$.

Since $g\in\G$ was arbitrary, we have
\begin{equation}\label{eq:erv}
\rot(g)\proj{u}(g)=a+S(g\cdot x_0-x_0)\qquad\text{for all }g\in\G.
\end{equation}
Let $C=\sup\set{\norm{u(g)}}{g\in\G}$.
Since $u$ is periodic, we have $C<\infty$.
Let $t\in\T$.
By \eqref{eq:erv} for all $n\in\N$ we have
\[n\norm{S\trans(t)}=\norm{S\trans(t^n)}=\norm[\bigg]{\rot(t^n)\proj{u}(t^n)-a-S\rot(t^n)x_0+Sx_0}\le2C+2\norm S\norm{x_0}\]
and thus, $S\trans(t)=0$.
Since $t\in\T$ was arbitrary, we have
\[Sx=0\quad\text{for all }x\in\spano\parens{\set{\trans(t)}{t\in\T}}=\{0_{d_1}\}\times\R^{d_2},\]
and thus, $S\in\Skew(d_1)\oplus\{0_{d_2,d_2}\}$.
By \eqref{eq:erv} we have $u\in\UIso$.
\end{proof}
\begin{Corollary}\label{Corollary:kernelseminorm}
Suppose that $\rot(\G)$ is finite, $\G_{x_0}$ is trivial, and $\RR\subset\G$ is an admissible neighborhood range of $\id$.
Then we have
\[\ker(\norm\fdot_{\RR})=\UIso.\]
Moreover, the map
\begin{align*}
&\R^d\times\R^{d_3\times d_4}\times\Skew(d_4)\to\ker(\norm{\fdot}_\RR)\\
&(a,A_1,A_2)\mapsto \parens[\bigg]{\G\to\R^d,g\mapsto \rot(g)^T\parens[\Big]{a+\parens[\Big]{\parens[\Big]{\begin{smallmatrix}0&A_1\\-A_1^T&A_2\end{smallmatrix}}\oplus0_{d_2,d_2}}(g\gdot x_0-x_0)}}
\end{align*}
is an isomorphism and in particular we have
\[\dim(\ker(\norm\fdot_\RR))=d+d_4(d_3+d_1-1)/2,\]
where $d_3=d-\daff$ and $d_4=\daff-d_2$.
\end{Corollary}
\begin{proof}
The assertion is clear by Theorem~\ref{Theorem:kernelseminorm}, Lemma~\ref{Lemma:finiterotG} and Proposition~\ref{Proposition:UtransUrotbig}.
\end{proof}
\begin{Corollary}
Suppose that $\G$ is a space group, $\G_{x_0}$ is trivial, and $\RR\subset\G$ is an admissible neighborhood range of $\id$.
Then we have
\[\ker(\norm\fdot_{\RR})=\UTrans.\]
\end{Corollary}
\begin{proof}
This is clear by Corollary~\ref{Corollary:kernelseminorm} and Example~\ref{Example:SpaceGroupURot}.
\end{proof}
\begin{Example}\label{Example:propertystarstar}
We present an example which shows that in Theorem~\ref{Theorem:equivalence} the premise that $\RR_1$ and $\RR_2$ are admissible neighborhood ranges of $\id$ cannot be weakened to the premise that for both $i\in\{1,2\}$ one has $\RR_i\G_{x_0}=\RR_i$, $\aff(\RR_i\cdot x_0)=\aff(\G\cdot x_0)$ and $\RR_i$ is a generating set of $\G$.

We consider the simple atomic chain from Example~\ref{ex:elementary-chains-i}\ref{ex:item-straight} with $d=2$, $d_1=1$, $d_2=1$, $t=\iso{I_2}{e_2}$, $\G=\angles t$ and $x_0=0$. 
The set $\RR_1=\{\id,t\}$ generates $\G$ and satisfies $\aff(\RR_1\cdot x_0)=\aff(\G\cdot x_0)$ but is not an admissible neighborhood range of $\id$.
The set $\RR_2=\{\id,t,t^2\}$ is admissible.
Using that the seminorms $\norm\fdot_\RR$ and $\nablanorm\fdot{\RR\setminus\{\id\}}$ are equivalent by Proposition~\ref{Proposition:nablanormequivalent-ohne-0}, it follows
\[\ker(\norm\fdot_{\RR_1})=\set{u\in\UPer}{\exists\,a\in\R\,\forall\,g\in\G : u_2(g)=a}.\]
By Corollary~\ref{Corollary:kernelseminorm} and Example~\ref{Example:SpaceGroupURot} we have
\[\ker(\norm\fdot_{\RR_2})=\UIso=\UTrans.\]
Since the kernels of $\norm\fdot_{\RR_1}$ and $\norm\fdot_{\RR_2}$ are not equal, the seminorms $\norm\fdot_{\RR_1}$ and $\norm\fdot_{\RR_2}$ are not equivalent.
\end{Example}
The following theorem summarizes the main results of this section.
\begin{Theorem}\label{Theorem:EquivalenceAll}
Suppose that $\RR_1,\RR_2\subset\G$ are admissible neighborhood ranges of $\id$.
Then the seminorms $\norm\fdot_{\RR_1}$, $\norm\fdot_{\RR_2}$, $\zeronorm\fdot{\RR_1}$, $\zeronorm\fdot{\RR_2}$, $\nablanorm\fdot{\RR_1}$, $\nablanorm\fdot{\RR_2}$, $\nablazeronorm\fdot{\RR_1}$ and $\nablazeronorm\fdot{\RR_2}$ are equivalent and their kernel is $\UIso\cap\UPer$.
\end{Theorem}
\begin{proof}
This is clear by Theorem~\ref{Theorem:equivalence}, Proposition~\ref{Proposition:nablanormequivalent-ohne-0}, Proposition~\ref{Proposition:nablanormequivalent-0}, Theorem~\ref{Theorem:Korn} and Theorem~\ref{Theorem:kernelseminorm}.
\end{proof}

%-------------------------------------------------------------------
%-------------------------------------------------------------------
\section{Stronger seminorms}
\label{section:zero}
%-------------------------------------------------------------------
%
In this section we introduce two stronger seminorms. First we consider $\newnorm\fdot\RR$ and its variant $\newnablanorm\fdot\RR$ that are defined as the averaged local distance to the spaces $\UIso$, respectively, $\URot$, introduced in Definition~\ref{Definition:Unew}. They thus measure rigidity up to local rotations about $\{0_{d_1}\}\times\R^{d_2}$. Then we define the even stronger seminorm $\norm{\nabla_\RR\fdot}_2$ as a discrete $H^1$ norm. Again we show that these seminorms are essentially independent of $\RR$ if $\RR$ is rich enough. In Corollary~\ref{Corollary:seminormequivalence} we observe that for bulk structures all seminorms (and in particular the weakest $\norm\fdot_\RR$ and strongest $\norm{\nabla_\RR\fdot}_2$) are equivalent.
\begin{Definition}
For all finite sets $\RR\subset\G$ such that $\RR\G_{x_0}=\RR$ we define the seminorms
\begin{align*}
&\begin{aligned}\newnorm{\fdot}{\RR}\colon&\UPer\to[0,\infty)\\
&u\mapsto\parens[\Big]{\frac1{\abs{\CC_N}}\sum_{g\in\CC_N}\norm{\pi_{\Unewiso\RR}(u(g\fdot)|_\RR)}^2}^{\frac12}\quad\text{if $u$ is $\T^N$-periodic,}\end{aligned}\\
\shortintertext{and}
&\begin{aligned}\newnablanorm\fdot\RR\colon&\UPer\to[0,\infty)\\
&u\mapsto\parens[\Big]{\frac1{\abs{\CC_N}}\sum_{g\in\CC_N}\norm{\pi_{\Unewrot\RR}(\nabla_\RR u(g))}^2}^{\frac12}\quad\text{if $u$ is $\T^N$-periodic,}\end{aligned}
\end{align*}
where $\pi_{\Unewiso\RR}$ and $\pi_{\Unewrot\RR}$ are the orthogonal projections on $\{u\colon\RR\to\R^d\}$ with respect to the norm $\norm\fdot$ with kernels $\Unewiso\RR$ and $\Unewrot\RR$, respectively.
\end{Definition}
\begin{Remark}\label{Remark:inequalityseminorm}
For all finite sets $\RR\subset\G$ such that $\RR\G_{x_0}=\RR$ we have $\norm\fdot_\RR\le\newnorm\fdot\RR$, but the seminorms $\norm\fdot_\RR$ and $\newnorm\fdot\RR$ need not be equivalent, see Proposition~\ref{Proposition:ExOne}.
\end{Remark}
\begin{Theorem}\label{Theorem:NewEquivalenceAll}%\label{Theorem:newequivalence}
Suppose that $\RR_1,\RR_2\subset\G$ are admissible neighborhood ranges of $\id$.
Then the seminorms $\newnorm\fdot{\RR_1}$, $\newnorm\fdot{\RR_2}$, $\newnablanorm\fdot{\RR_1}$, and $\newnablanorm\fdot{\RR_2}$ are equivalent and their kernel is $\UIso\cap\UPer$.
\end{Theorem}
\begin{proof}
The proof that the seminorms $\newnorm\fdot{\RR_1}$ and $\newnorm\fdot{\RR_2}$ are equivalent is analogous to the proof of Theorem~\ref{Theorem:equivalence}:
For all finite sets $\RR\subset\G$ such that $\RR\G_{x_0}=\RR$ we define the seminorm
\begin{align*}
\newsemip{\RR}\colon&\{u\colon\RR\to\R^d\}\to[0,\infty),\qquad
u\mapsto\norm{\pi_{\Unewiso\RR}(u)}
\end{align*}
on $(\R^d)^\RR$ whose kernel is $\Unewiso\RR$.
Moreover, for all finite sets $\RR_1,\RR_2\subset\G$ such that $\RR\G_{x_0}=\RR$ we define the seminorm
\begin{align*}
\newsemiq{\RR_1}{\RR_2}\colon\{u\colon\RR_1\RR_2\to\R^d\}\to[0,\infty)\qquad
u\mapsto \parens[\bigg]{\sum_{g\in\RR_1}\newsemip[2]{\RR_2}\parens[\big]{u(g\fdot)|_{\RR_2}}}^{\frac12}
\end{align*}
on $(\R^d)^{\RR_1\RR_2}$.
Analogously to Lemma~\ref{Lemma:hRSIequivalence} for all $\RR_1,\RR_2\subset\G$ such that $\RR_1$ is finite and $\RR_2$ is an admissible neighborhood range of $\id$ there exists a finite set $\RR_3\subset\G$ such that $\RR_1\subset \RR_3\RR_2$ and the seminorms $\newsemip{\RR_3\RR_2}$ and $\newsemiq{\RR_3}{\RR_2}$ are equivalent.
As in the proof of Theorem~\ref{Theorem:equivalence} this implies that the seminorms $\newnorm\fdot{\RR_1}$ and $\newnorm\fdot{\RR_2}$ are equivalent.

Analogously to the proof of Proposition~\ref{Proposition:nablanormequivalent-ohne-0}, the seminorms $\newnorm\fdot\RR$ and $\newnablanorm\fdot\RR$ are equivalent for all finite sets $\RR\subset\G$ such that $\RR\G_{x_0}=\RR$ and $\G_{x_0}\subset\RR$.
In particular, if $\RR\subset\G$ is admissible, then $\newnorm\fdot\RR$ and $\newnablanorm\fdot\RR$ are equivalent.

Finally, suppose that $\RR\subset\G$ is an admissible neighborhood range of $\id$.
Analogously to the proof of Theorem~\ref{Theorem:kernelseminorm}, we have $\UIso\cap\UPer\subset\ker(\newnorm\fdot\RR)$. Since $\norm\fdot_\RR\le\newnorm\fdot\RR$, by Theorem~\ref{Theorem:kernelseminorm} we have $\ker(\newnorm\fdot\RR)\subset\UIso\cap\UPer$.
\end{proof}
For the second seminorm to be discussed in this section we first slightly extend our notion of the $\ell^2$ norm $\norm\fdot_2$. 
\begin{Definition}
For all finite sets $\RR\subset\G$ we define the norm
\begin{align*}
\norm\fdot_2\colon&\set[\big]{u\colon\G\to\{v\colon\RR\to\R^d\}}{u\text{ is periodic}}\to[0,\infty)\\
&u\mapsto\parens[\Big]{\frac1{\abs{\CC_N}}\sum_{g\in\CC_N}\norm{u(g)}^2}^{\frac12}\quad\text{if $u$ is $\T^N$-periodic}.
\end{align*}
\end{Definition}
\begin{Theorem}\label{Theorem:nablaequivalent}
Let $\RR_1,\RR_2\subset\G$ be finite generating sets of $\G$ such that $\RR_1\G_{x_0}=\RR_1$ and $\RR_2\G_{x_0}=\RR_2$.
Then the seminorms $\norm{\nabla_{\RR_1}\fdot}_2$ and $\norm{\nabla_{\RR_2}\fdot}_2$ on $\UPer$ are equivalent and their kernel is $\UTrans\cap\UPer$.
\end{Theorem}
\begin{proof}
First we show that the seminorms $\norm{\nabla_{\RR_1}\fdot}_2$ and $\norm{\nabla_{\RR_2}\fdot}_2$ are equivalent.
It suffices to show that there exists a constant $C>0$ such that $\norm{\nabla_{\RR_1}\fdot}_2\le C\norm{\nabla_{\RR_2}\fdot}_2$.
Since $\RR_2$ generates $\G$, for every $r\in\RR_1$ there exist some $n_r\in\N$ and $s_{r,1},\dots,s_{r,n_r}\in\RR_2\cup\RR_2^{-1}$ such that $r=s_{r,1}\dots s_{r,n_r}$.
Let $u\in\UPer$.
Let $N\in \MM$ be such that $u$ is $\T^N$-periodic.
Then we have
\begin{align*}
&\norm{\nabla_{\RR_1}u}_2^2=\frac1{\abs{\CC_N}}\sum_{g\in\CC_N}\norm{\nabla_{\RR_1}u(g)}^2\\
&\quad=\frac1{\abs{\CC_N}}\sum_{g\in\CC_N}\sum_{r\in\RR_1}\norm{\rot(r)\proj{u}(gr)-\proj{u}(g)}^2\\
&\quad=\frac1{\abs{\CC_N}}\sum_{g\in\CC_N}\sum_{r\in\RR_1}\norm[\bigg]{\sum_{i=1}^{n_r}\rot(s_{r,1}\dots s_{r,{i-1}})\parens[\big]{\rot(s_{r,i})\proj{u}(gs_{r,1}\dots s_{r,i})-\proj{u}(gs_{r,1}\dots s_{r,i-1})}}^2\\
&\quad\le \frac1{\abs{\CC_N}}\sum_{g\in\CC_N}\sum_{r\in\RR_1}\parens[\bigg]{\sum_{i=1}^{n_r}\norm{\rot(s_{r,i})\proj{u}(gs_{r,1}\dots s_{r,i})-\proj{u}(gs_{r,1}\dots s_{r,i-1})}}^2\\
&\quad\le \frac1{\abs{\CC_N}}\sum_{g\in\CC_N}\sum_{r\in\RR_1}n_r\sum_{i=1}^{n_r}\norm{\rot(s_{r,i})\proj{u}(gs_{r,1}\dots s_{r,i})-\proj{u}(gs_{r,1}\dots s_{r,i-1})}^2\\
&\quad\le\frac C{\abs{\CC_N}}\sum_{\tilde g\in\CC_N}\sum_{s\in\RR_2}\norm{\rot(s)\proj{u}(\tilde gs)-\proj{u}(\tilde g)}^2\\
&\quad=C\norm{\nabla_{\RR_2}u}_2^2,
\end{align*}
where $C=\sum_{r\in\RR_1}n_r^2$.
In the fifth step we used that the arithmetic mean is lower or equal than the root mean square.
In the sixth step, if $s_{r,i}\in\RR_2$, we substituted $gs_{r,1}\dots s_{r,i-1}$ by $\tilde g$, and if $s_{r,i}\in\RR_2^{-1}$, we substituted $gs_{r,1}\dots s_{r,i}$ by $\tilde g$.

Let $\RR=\RR_1$.
Now we show that $\ker(\norm{\nabla_\RR\fdot}_2)=\UTrans\cap\UPer$.
It is clear that $\UTrans\cap\UPer\subset\ker(\norm{\nabla_\RR\fdot}_2)$.
If $u\in\ker(\norm{\nabla_\RR\fdot}_2)$, then for all $g\in\G$ we have
\begin{equation}\label{eq:Equality}
0=\norm{\nabla_{\RR\cup g\G_{x_0}}u}_2\ge\norm{\rot(g)\proj{u}(g)-\proj{u}(\id)},
\end{equation}
where we used that the seminorms $\norm{\nabla_\RR\fdot}_2$ and $\norm{\nabla_{\RR\cup g\G_{x_0}}\fdot}_2$ are equivalent.
By \eqref{eq:Equality} we have $\rot(g)\proj{u}(g)=\proj{u}(\id)$ for all $g\in\G$ and thus $u\in\UTrans$.
\end{proof}
\begin{Remark}
For all finite sets $\RR\subset\G$ such that $\RR\G_{x_0}=\RR$ we have $\newnorm\fdot\RR\le\norm{\nabla_\RR\fdot}_2$, but the seminorms $\newnorm\fdot\RR$ and $\norm{\nabla_\RR\fdot}$ need not be equivalent since their kernels are not equal, see Theorem~\ref{Theorem:NewEquivalenceAll} and Theorem~\ref{Theorem:nablaequivalent}.
\end{Remark}
Theorem~\ref{Theorem:Korn} yields the following corollary.
\begin{Corollary}(A discrete Korn inequality for space groups)\label{Corollary:seminormequivalence}
Suppose that $\G$ is a space group and $\RR\subset\G$ is an admissible neighborhood range of $\id$.
Then the seminorms $\norm\fdot_{\RR}$, $\newnorm\fdot\RR$ and $\norm{\nabla_\RR\fdot}_2$ are equivalent.
\end{Corollary}
\begin{proof}
Under the assumptions made we have $\zeroUrot\RR=\Unewrot\RR\subset\Utrans\RR$ and $\nablazeronorm\fdot\RR=\newnablanorm\fdot\RR=\norm{\nabla_\RR\fdot}_2$.
With Theorem~\ref{Theorem:EquivalenceAll} and Theorem~\ref{Theorem:NewEquivalenceAll} the assertion follows.
\end{proof}

%-------------------------------------------------------------------
%-------------------------------------------------------------------
\section{Two basic examples in real and Fourier space}\label{section:examples}
%-------------------------------------------------------------------
%
We finally work out explicitly equivalent descriptions of the seminorms $\norm\fdot_\RR$ and  $\norm{\nabla_\RR\fdot}_2$ (respectively, $\newnorm\fdot\RR$) in terms of their Fourier transform for the two basic examples of atomic chains introduced in Example~\ref{ex:elementary-chains-i}: the simple one-dimensional atomic chain in $\R^2$ with $\daff=d_2=d_1=1$ is considered in Proposition~\ref{Proposition:ExOne}, the atomic chain with non-trivial bond angles and $\daff=2$, $d_2=d_1=1$ in Proposition~\ref{Proposition:ExTwo}. While the seminorms $\newnorm\fdot\RR$ and $\norm{\nabla_\RR\fdot}_2$ will be equivalent as $d_1=1$, in both examples we will see that $\norm\fdot_\RR$ and $\newnorm\fdot\RR$ are not equivalent. 
\begin{Proposition}\label{Proposition:ExOne}
Suppose that $t=\iso{I_2}{e_2}\in\E(2)$, $\G=\angles t<\E(2)$, $x_0=0\in\R^2$ and $\RR\subset\G$ is an admissible neighborhood range of $\id$, \eg $\RR=\{\id,t,t^2\}$.
Then the seminorms $\newnorm\fdot\RR$ and $\norm{\nabla_\RR\fdot}_2$ are equivalent and there exist constants $C,c>0$ such that for all $u\in\UPer$ we have
\begin{align*}
&c\norm{\nabla_\RR u}_2^2\le\sum_{k\in[0,1)\cap\Q}\abs k_1^2\norm{\fourier u(\chi_k)}^2\le C\norm{\nabla_\RR u}_2^2
\shortintertext{and}
&c\norm u_\RR^2\le\sum_{k\in[0,1)\cap\Q}\parens[\bigg]{\abs k_1^4\abs{\fourier u_1(\chi_k)}^2+\abs k_1^2\abs[\big]{\fourier u_2(\chi_k)}^2}\le C\norm u_\RR^2,
\end{align*}
% Alternativ with enumerate and \displaystyle{...}
%
where $\abs\fdot_1\colon\R\to[0,\infty)$, $k\mapsto\dist(k,\Z)$ is the \emph{distance to nearest integer function}.
\end{Proposition}
\begin{proof}
As noted in Example~\ref{ex:elementary-chains-ii}, the set $\{\id,t,t^2\}$ is an admissible neighborhood range of $\id$ and by Theorem~\ref{Theorem:nablaequivalent} and Theorem~\ref{Theorem:EquivalenceAll} without loss of generality we let $\RR=\{\id,t,t^2\}$.
Since $\Unewrot\RR=\{0\}$, we have $\newnablanorm\fdot\RR=\norm{\nabla_\RR\fdot}_2$ and thus the seminorms $\newnorm\fdot\RR$ and $\norm{\nabla_\RR\fdot}_2$ are equivalent by Theorem~\ref{Theorem:NewEquivalenceAll}.

Remark~\ref{rmk:dualspacelattice} shows that $\dual\G=\set{\chi_k}{k\in[0,1)}$, where $\chi_k\colon\G\to\C$ is given by $\chi_k(t^n)=\euler^{2\pi\iu nk}$ for all $n\in\Z$ with $k$ such that $\chi(t)=\euler^{2\pi\iu k}$ and that $\EE=\set{\chi_k}{k\in[0,1)\cap\Q}$.

There is a constant $c_T\in(0,1)$ such that for all $k\in[0,1)$ and $n\in\{1,2\}$ we have 
\begin{align}
&c_T\abs k_1\le\abs[\big]{\euler^{-2\pi\iu k}-1},\label{eq:TOneA}\\
&c_T\abs[\big]{\euler^{-2\pi\iu nk}-1}\le \abs k_1,\label{eq:TThreeA}\\
\shortintertext{and}
&c_T\abs[\big]{\euler^{-2\pi\iu nk}-1+2\pi\iu nk}\le\abs k_1^2.\label{eq:TFiveA}
\end{align}
Indeed, this is clear on compact subintervals of $(0,1)$ since $|\euler^{-2\pi\iu k}-1|, \abs k_1>0$ for $k\in(0,1)$ and it follows from a Taylor expansion of $k \mapsto \euler^{-2\pi\iu nk}$ in a neighborhood of $\{0,1\}$.
For all $u\in\UPer$ we have
\begin{align}
\norm{\nabla_\RR u}_2^2&=\sum_{\chi\in\EE}\norm[\big]{\fourier{\nabla_\RR u}(\chi)}^2\nonumber\\
&=\sum_{k\in[0,1)\cap\Q}\norm[\big]{\parens{\chi_k(h)^{-1}\fourier u(\chi_k)-\fourier u(\chi_k)}_{h\in\RR}}^2\nonumber\\
&=\sum_{k\in[0,1)\cap\Q}\sum_{n=1}^2\abs[\big]{\euler^{-2\pi\iu nk}-1}^2\norm[\big]{\fourier u(\chi_k)}^2,\label{eq:NablaOneA}
\end{align}
where we used Proposition~\ref{Proposition:TFplancherelmatrix} in the first step and Lemma~\ref{Lemma:PeriodicTranslation} in the second step.
Equations \eqref{eq:TOneA}, \eqref{eq:TThreeA} and \eqref{eq:NablaOneA} imply the first assertion.

Now we show the second assertion.
Let $\RR'=\{t,t^2\}$.
By Proposition~\ref{Proposition:nablanormequivalent-ohne-0} the seminorms $\norm\fdot_\RR$ and $\nablanorm\fdot{\RR'}$ are equivalent, \ie there exist some constants $C,c>0$ such that
\begin{equation}\label{eq:(a)A}
c\norm\fdot_\RR\le\nablanorm\fdot{\RR'}\le C\norm\fdot_\RR.
\end{equation}
We define the linear map
\begin{align*}
g_{\RR'}\colon&\Skew(2,\C)\to\C^{2\times\abs{\RR'}}\\
&S\mapsto\parens[\big]{S(h\cdot x_0-x_0)}_{h\in\RR'}.
\end{align*}
For all $u\in\UPer$ we have
\begin{align}
\nablanorm u{\RR'}^2&=\inf\set[\Big]{\norm{\nabla_{\RR'}u-g_{\RR'}\circ v}_2^2}{v\in\Per(\G,\Skew(2,\C))}\nonumber\\
&=\inf\set[\bigg]{\sum_{\chi\in\EE}\norm[\big]{\fourier{\nabla_{\RR'}u}(\chi)-g_{\RR'}\circ\tilde v(\chi)}^2}{\tilde v\in\bigoplus_{\chi\in\EE}\Skew(2,\C)}\nonumber\\
&=\sum_{\chi\in\EE}\inf\set[\Big]{\norm[\big]{\fourier{\nabla_{\RR'}u}(\chi)-g_{\RR'}(S)}^2}{S\in\Skew(2,\C)}\nonumber\\
&=\sum_{k\in[0,1)\cap\Q}\inf\set[\bigg]{\norm[\big]{\parens[\big]{\chi_k(h)^{-1}\fourier u(\chi_k)-\fourier u(\chi_k)-\parens[\big]{\begin{smallmatrix}0&-s\\s&0\end{smallmatrix}}(h\cdot x_0-x_0)}_{h\in\RR'}}^2}{s\in\C}\nonumber\\
&=\sum_{k\in[0,1)\cap\Q}\inf\set[\bigg]{\sum_{n=1}^2\norm[\big]{(\euler^{-2\pi\iu nk}-1)\fourier u(\chi_k)+nse_1}^2}{s\in\C},\label{eq:(b)A}
\end{align}
where we used Proposition~\ref{Proposition:TFplancherelmatrix} in the second step and Lemma~\ref{Lemma:PeriodicTranslation} in the fourth step.

It holds
\begin{equation}\label{eq:(c)A}
\sum_{i=1}^n a_i^2\le\parens[\bigg]{\sum_{i=1}^n a_i}^2\le n\sum_{i=1}^na_i^2
\end{equation}
for all $n\in\N$ and $a_1,\dots,a_n\ge0$.

We define the functions
\begin{align*}
\begin{aligned}
f_1\colon&[0,1)\times\C^2\times\C\to[0,\infty),\quad\!\
(k,v,s)\mapsto\sum_{n=1}^2\norm[\big]{(\euler^{-2\pi\iu nk}-1)v+nse_1}
\quad\text{and}\\
f_2\colon&[0,1)\times\C^2\to[0,\infty),\qquad\qquad
(k,v)\mapsto\abs k_1^2\abs{v_1}+\abs k_1\abs{v_2}.
\end{aligned}
\end{align*}
By \eqref{eq:(a)A}, \eqref{eq:(b)A} and \eqref{eq:(c)A} it suffices so show that there exist some constants $C,c>0$ such that for all $(k,v)\in[0,1)\times\C^2$ we have
\begin{equation}\label{eq:(star)A}
c\inf_{s\in\C}f_1(k,v,s)\le f_2(k,v)\le C\inf_{s\in\C}f_1(k,v,s).
\end{equation}
First we show the left inequality of \eqref{eq:(star)A}.
By \eqref{eq:TThreeA} and \eqref{eq:TFiveA} for all $(k,v)\in[0,1)\times\C^2$ we have
\begin{align*}
\inf_{s\in\C}f_1(k,v,s)&\le f_1(k,v,2\pi\iu kv_1)\\
&\le\sum_{n=1}^2\parens[\Big]{\abs[\big]{\euler^{-2\pi\iu nk}-1+2\pi\iu nk}\abs{v_1}+\abs[\big]{\euler^{-2\pi\iu nk}-1}\abs{v_2}}\\
&\le\tfrac2{c_T}f_2(k,v).
\end{align*}
Now we show the right inequality of \eqref{eq:(star)A}.
Let $(k,v,s)\in[0,1)\times\C^2\times\C$.
By \eqref{eq:TOneA} we have
\begin{align}
f_1(k,v,s)&\ge\abs[\big]{\euler^{-2\pi\iu k}v_1-v_1+s}+\tfrac12\abs[\big]{\euler^{-4\pi\iu k}v_1-v_1+2s}\nonumber\\
&\ge\tfrac12\abs[\big]{2(\euler^{-2\pi\iu k}v_1-v_1+s)-(\euler^{-4\pi\iu k}v_1-v_1+2s)}\nonumber\\
&=\tfrac12\abs[\big]{\euler^{-2\pi\iu k}-1}^2\abs{v_1}\nonumber\\
&\ge \tfrac{c_T^2}2\abs k_1^2\abs{v_1}\label{eq:(2a)A}
\end{align}
and
\begin{equation}\label{eq:(2)A}
f_1(k,v,s)\ge\abs[\big]{\euler^{-2\pi\iu k}-1}\abs{v_2}\ge c_T\abs k_1\abs{v_2}.
\end{equation}
By \eqref{eq:(2a)A} and \eqref{eq:(2)A} we have
\[f_1(k,v,s)\ge\tfrac{c_T^2}4 f_2(k,v).\qedhere\]
\end{proof}
\begin{Proposition}\label{Proposition:ExTwo}
Suppose that $t=\iso{\parens{\begin{smallmatrix}-1&0\\0&1\end{smallmatrix}}}{e_2}\in\E(2)$, $\G=\angles t<\E(2)$, $x_0=e_1\in\R^2$ and $\RR\subset\G$ is an admissible neighborhood range of $\id$, \eg $\RR=\{t^0,\dots,t^3\}$.
Then the seminorms $\newnorm\fdot\RR$ and $\norm{\nabla_\RR\fdot}_2$ are equivalent and there exist constants $C,c>0$ such that for all $u\in\UPer$ we have
\begin{align*}
&c\norm{\nabla_\RR u}_2^2\le\sum_{k\in[0,1)\cap\Q}\parens[\bigg]{\abs{k-\tfrac12}_1^2\abs{\fourier u_1(\chi_k)}^2+\abs k_1^2\abs{\fourier u_2(\chi_k)}^2}\le C\norm{\nabla_\RR u}_2^2
\intertext{and}
&c\norm u_\RR^2\le\sum_{k\in[0,1)\cap\Q}\parens[\bigg]{\abs{k-\tfrac12}_1^4\abs{\fourier u_1(\chi_k)}^2+\abs k_1^2\abs[\big]{2\pi\iu(k-\tfrac12)\fourier u_1(\chi_k)-\fourier u_2(\chi_k)}^2}\le C\norm u_\RR^2,
\end{align*}
% Alternativ with enumerate and \displaystyle{...}
%
where $\abs\fdot_1\colon\R\to[0,\infty)$, $k\mapsto\dist(k,\Z)$ is the \emph{distance to nearest integer function}.
\end{Proposition}
\begin{proof}
As noted in Example~\ref{ex:elementary-chains-ii}, the set $\{t^0,\dots,t^3\}$ is an admissible neighborhood range of $\id$ and by Theorem~\ref{Theorem:nablaequivalent} and Theorem~\ref{Theorem:EquivalenceAll} without loss of generality we let $\RR=\{t^0,\dots,t^3\}$.
Since $\Unewrot\RR=\{0\}$, we have $\newnablanorm\fdot\RR=\norm{\nabla_\RR\fdot}_2$ and thus the seminorms $\newnorm\fdot\RR$ and $\norm{\nabla_\RR\fdot}_2$ are equivalent by Theorem~\ref{Theorem:NewEquivalenceAll}.

As in the previous example we have $\EE=\set{\chi_k}{k\in[0,1)\cap\Q}$, where $\chi_k\in\dual\G$ is given by $\chi_k(t^n)=\euler^{2\pi\iu nk}$ for all $n\in\Z$, cf.\ Remark~\ref{rmk:dualspacelattice}. Since $\set{k\in[0,1)}{\euler^{-2\pi\iu k}=1}=\{0\}$, $\set{k\in[0,1)}{\euler^{-2\pi\iu k}=-1}=\{\frac12\}$ and by Taylor's theorem, there exists a constant $c_T\in(0,1)$ such that for all $k\in[0,1)$ and $n\in\{1,2,3\}$ we have
\begin{align}
&c_T\abs k_1\le\abs[\big]{\euler^{-2\pi\iu k}-1},\label{eq:TOne}\\
&c_T\abs[\big]{k-\tfrac12}_1\le\abs[\big]{\euler^{-2\pi\iu k}+1},\label{eq:TTwo}\\
&c_T\abs[\big]{\euler^{-2\pi\iu nk}-1}\le \abs k_1,\label{eq:TThree}\\
&c_T\abs[\big]{\euler^{-2\pi\iu nk}-(-1)^n}\le \abs[\big]{k-\tfrac12}_1,\label{eq:TFour}
\shortintertext{and}
&c_T\abs[\big]{\euler^{-2\pi\iu nk}-(-1)^n+(-1)^n2\pi\iu n(k-\tfrac12)}\le\abs[\big]{k-\tfrac12}_1^2.\label{eq:TFive}
\end{align}
For all $u\in\UPer$ we have
\begin{align}
\norm{\nabla_\RR u}_2^2&=\sum_{\chi\in\EE}\norm[\big]{\fourier{\nabla_\RR u}(\chi)}^2\nonumber\\
&=\sum_{k\in[0,1)\cap\Q}\norm[\big]{\parens{\chi_k(h)^{-1}\fourier u(\chi_k)-\rot(h)^T\fourier u(\chi_k)}_{h\in\RR}}^2\nonumber\\
&=\sum_{k\in[0,1)\cap\Q}\sum_{n=1}^3\norm[\Big]{\euler^{-2\pi\iu nk}\fourier u(\chi_k)-\parens[\big]{\begin{smallmatrix}-1&0\\0&1\end{smallmatrix}}^n\fourier u(\chi_k)}^2\nonumber\\
&=\sum_{k\in[0,1)\cap\Q}\sum_{n=1}^3\parens[\Big]{\abs[\big]{\euler^{-2\pi\iu nk}-(-1)^n}^2\abs[\big]{\fourier u_1(\chi_k)}^2+\abs[\big]{\euler^{-2\pi\iu nk}-1}^2\abs[\big]{\fourier u_2(\chi_k)}^2},\label{eq:NablaOne}
\end{align}
where we used Proposition~\ref{Proposition:TFplancherelmatrix} in the first step and Lemma~\ref{Lemma:PeriodicTranslation} in the second step.
Equations \eqref{eq:TOne}, \eqref{eq:TTwo}, \eqref{eq:TThree}, \eqref{eq:TFour} and \eqref{eq:NablaOne} imply the first assertion.

Now we show the second assertion.
Let $\RR'=\{t^1,t^2,t^3\}$.
By Proposition~\ref{Proposition:nablanormequivalent-ohne-0} the seminorms $\norm\fdot_\RR$ and $\nablanorm\fdot{\RR'}$ are equivalent, \ie there exist some constants $C,c>0$ such that
\begin{equation}\label{eq:(a)}
c\norm\fdot_\RR\le\nablanorm\fdot{\RR'}\le C\norm\fdot_\RR.
\end{equation}
We define the linear map
\begin{align*}
g_{\RR'}\colon&\Skew(2,\C)\to\C^{2\times\abs{\RR'}}\\
&S\mapsto\parens[\big]{\rot(h)^TS(h\cdot x_0-x_0)}_{h\in\RR'}.
\end{align*}
For all $u\in\UPer$ we have
\begin{align}
&\nablanorm u{\RR'}^2=\inf\set[\Big]{\norm{\nabla_{\RR'}u-g_{\RR'}\circ v}_2^2}{v\in\Per(\G,\Skew(2,\C))}\nonumber\\
&\quad=\inf\set[\bigg]{\sum_{\chi\in\EE}\norm[\big]{\fourier{\nabla_{\RR'}u}(\chi)-g_{\RR'}\circ\tilde v(\chi)}^2}{\tilde v\in\bigoplus_{\chi\in\EE}\Skew(2,\C)}\nonumber\\
&\quad=\sum_{\chi\in\EE}\inf\set[\Big]{\norm[\big]{\fourier{\nabla_{\RR'}u}(\chi)-g_{\RR'}(S)}^2}{S\in\Skew(2,\C)}\nonumber\\
&\quad=\sum_{k\in[0,1)\cap\Q}\inf_{s\in\C}\norm[\big]{\parens[\big]{\chi_k(h)^{-1}\fourier u(\chi_k)-\rot(h)^T\fourier u(\chi_k)-\rot(h)^T\parens[\big]{\begin{smallmatrix}0&-s\\s&0\end{smallmatrix}}(h\cdot x_0-x_0)}_{h\in\RR'}}^2\nonumber\\
&\quad=\sum_{k\in[0,1)\cap\Q}\inf\set[\bigg]{\sum_{n=1}^3\norm[\Big]{\euler^{-2\pi\iu nk}\fourier u(\chi_k)-\parens[\big]{\begin{smallmatrix}-1&0\\0&1\end{smallmatrix}}^n\fourier u(\chi_k)-\parens[\Big]{\begin{smallmatrix}(-1)^{n+1}ns\\((-1)^n-1)s\end{smallmatrix}}}^2}{s\in\C},\label{eq:(b)}
\end{align}
where we used Proposition~\ref{Proposition:TFplancherelmatrix} in the second step and Lemma~\ref{Lemma:PeriodicTranslation} in the fourth step.

It holds
\begin{equation}\label{eq:(c)}
\sum_{i=1}^n a_i^2\le\parens[\bigg]{\sum_{i=1}^n a_i}^2\le n\sum_{i=1}^na_i^2
\end{equation}
for all $n\in\N$ and $a_1,\dots,a_n\ge0$.

We define the functions
\begin{align*}
f_1\colon&[0,1)\times\C^2\times\C\to[0,\infty)\\
&(k,v,s)\mapsto\sum_{n=1}^3\norm[\Big]{\euler^{-2\pi\iu nk}v-\parens[\big]{\begin{smallmatrix}-1&0\\0&1\end{smallmatrix}}^nv-\parens[\Big]{\begin{smallmatrix}(-1)^{n+1}ns\\((-1)^n-1)s\end{smallmatrix}}}
\shortintertext{and}
f_2\colon&[0,1)\times\C^2\to[0,\infty)\\
&(k,v)\mapsto\abs{k-\tfrac12}_1^2\abs{v_1}+\abs k_1\abs{2\pi\iu(k-\tfrac12)v_1-v_2}.
\end{align*}
By \eqref{eq:(a)}, \eqref{eq:(b)} and \eqref{eq:(c)} it suffices so show that there exist some constants $C,c>0$ such that for all $(k,v)\in[0,1)\times\C^2$ we have
\begin{equation}\label{eq:(star)}
c\inf_{s\in\C}f_1(k,v,s)\le f_2(k,v)\le C\inf_{s\in\C}f_1(k,v,s).
\end{equation}
First we show the right inequality of \eqref{eq:(star)}.
Let $c_R>0$ be small enough, \eg $c_R=\tfrac{c_T^3}{400}$.
Let $(k,v,s)\in[0,1)\times\C^2\times\C$.
By \eqref{eq:TOne} and \eqref{eq:TTwo} we have
\begin{align}
f_1(k,v,s)&\ge\abs[\big]{\euler^{-2\pi\iu k}v_1+v_1-s}+\tfrac12\abs[\big]{\euler^{-4\pi\iu k}v_1-v_1+2s}\nonumber\\
&\ge\tfrac12\abs[\big]{2(\euler^{-2\pi\iu k}v_1+v_1-s)+\euler^{-4\pi\iu k}v_1-v_1+2s}\nonumber\\
&=\tfrac12\abs[\big]{\euler^{-2\pi\iu k}+1}^2\abs{v_1}\nonumber\\
&\ge \tfrac{c_T^2}2\abs{k-\tfrac12}_1^2\abs{v_1}\label{eq:(2a)}
\end{align}
and
\begin{align}
f_1(k,v,s)&\ge\sum_{n\in\{1,3\}}\abs[\big]{\euler^{-2\pi\iu nk}v_2-v_2+2s}\nonumber\\
&\ge\abs[\big]{\euler^{-2\pi\iu k}v_2-v_2+2s-(\euler^{-6\pi\iu k}v_2-v_2+2s)}\nonumber\\
&=\abs[\big]{\euler^{-2\pi\iu k}+1}\abs[\big]{\euler^{-2\pi\iu k}-1}\abs{v_2}\nonumber\\
&\ge c_T^2\abs k_1\abs{k-\tfrac12}_1\abs{v_2}.\label{eq:(2)}
\end{align}
\begin{case}[Case 1: $k\in\bracks{0,\tfrac14}\cup[\tfrac34,1)$] \ \\
Since $k\in\bracks{0,\tfrac14}\cup[\tfrac34,1)$, we have $\abs{k-\tfrac12}_1\ge\tfrac14$.
By \eqref{eq:(2a)} and \eqref{eq:(2)} we have
\[f_1(k,v,s)\ge c_R\abs{k-\tfrac12}_1^2\abs{v_1}+\pi c_R\abs{v_1}+c_R\abs k_1\abs{v_2}\ge c_Rf_2(k,v),\]
where in the last step we used the triangle inequality.
\end{case}
\begin{case}[Case 2: $k\in(\frac14,\tfrac34)$] \ \\
Since $k\in(\tfrac14,\tfrac34)$, we have $\abs k_1\ge\tfrac14$.
By \eqref{eq:TOne} and \eqref{eq:TFive} we have
\begin{align}
f_1(k,v,s)&\ge\abs[\big]{(\euler^{-4\pi\iu k}-1)v_1+2s}+\abs[\big]{(\euler^{-2\pi\iu k}-1)v_2+2s}\nonumber\\
&\ge\abs[\big]{(\euler^{-4\pi\iu k}-1)v_1+2s-((\euler^{-2\pi\iu k}-1)v_2+2s)}\nonumber\\
&=\abs[\big]{\euler^{-2\pi\iu k}-1}\abs[\big]{(\euler^{-2\pi\iu k}+1)v_1-v_2}\nonumber\\
&\ge\tfrac{c_T}4\abs[\big]{(\euler^{-2\pi\iu k}+1)v_1-v_2}\nonumber\\
&\ge\tfrac{c_T}4\abs[\big]{2\pi\iu(k-\tfrac12)v_1-v_2}-\tfrac{c_T}4\abs[\big]{\euler^{-2\pi\iu k}+1-2\pi\iu(k-\tfrac12)}\abs{v_1}\nonumber\\
&\ge\tfrac{c_T}4\abs[\big]{2\pi\iu(k-\tfrac12)v_1-v_2}-\tfrac14\abs[\big]{k-\tfrac12}_1^2\abs{v_1}.\label{eq:(7)}
\end{align}
By \eqref{eq:(2a)} and \eqref{eq:(7)} we have $f_1(k,v,s)\ge c_Rf_2(k,v)$.
\end{case}
Now we show the left inequality of \eqref{eq:(star)}.
Let $C_L>0$ be large enough, \eg $C_L=\tfrac{120}{c_T}$.
Let $(k,v)\in[0,1)\times\C^2$.
We have
\begin{equation}\label{eq:(A)}
f_2(k,v)\ge\abs k_1\abs[\big]{k-\tfrac12}_1\abs[\big]{2\pi\iu(k-\tfrac12)v_1-v_2}\ge\abs k_1\abs[\big]{k-\tfrac12}_1\abs{v_2}-\pi\abs[\big]{k-\tfrac12}_1^2\abs{v_1}.
\end{equation}
By \eqref{eq:(A)} and the definition of $f_2$, we have
\begin{equation}\label{eq:(C)}
f_2(k,v)\ge \tfrac15\abs k_1\abs[\big]{k-\tfrac12}_1\abs{v_2}.
\end{equation}
\begin{case}[Case 1: $k\in\bracks{0,\tfrac14}\cup[\tfrac34,1)$] \ \\
Since $k\in[0,\tfrac14]\cup[\tfrac34,1)$, we have $\abs{k-\tfrac12}_1\ge\tfrac14$.
We have
\begin{align*}
\inf_{s\in\C}f_1(k,v,s)&\le f_1(k,v,0)\\
&\le 6\abs{v_1}+\abs{v_2}\sum_{n=1}^3\abs[\big]{\euler^{-2\pi\iu nk}-1}\\
&=6\abs{v_1}+\abs[\big]{\euler^{-2\pi\iu k}-1}\abs{v_2}\sum_{n=1}^3\abs[\bigg]{\sum_{m=0}^{n-1}\euler^{-2\pi\iu mk}}\\
&\le6\abs{v_1}+\tfrac6{c_T}\abs k_1\abs{v_2}\\
&\le C_Lf_2(k,v),
\end{align*}
where we used \eqref{eq:TThree} in the second to last step and \eqref{eq:(C)} in the last step.
\end{case}
\begin{case}[Case 2: $k\in(\frac14,\tfrac34)$] \ \\
Since $k\in(\tfrac14,\tfrac34)$, we have $\abs k_1\ge\tfrac14$.
By \eqref{eq:TFive} and \eqref{eq:TFour} we have
\begin{align}
&\inf_{s\in\C}f_1(k,v,s)\le f_1(k,v,v_2)\nonumber\\
&\quad\le\sum_{n=1}^3\parens[\Big]{\abs[\big]{(\euler^{-2\pi\iu nk}-(-1)^n)v_1+(-1)^nnv_2}+\abs[\big]{\euler^{-2\pi\iu nk}-(-1)^n}\abs{v_2}}\nonumber\\
&\quad\le\sum_{n=1}^3\parens[\Big]{\abs[\big]{\euler^{-2\pi\iu nk}-(-1)^n+(-1)^n2\pi\iu n(k-\tfrac12)}\abs{v_1}+n\abs[\big]{2\pi\iu(k-\tfrac12)v_1-v_2}\nonumber\\
&\quad\qquad+\abs[\big]{\euler^{-2\pi\iu nk}-(-1)^n}\abs{v_2}}\nonumber\\
&\quad\le \tfrac6{c_T}\parens[\Big]{\abs[\big]{k-\tfrac12}_1^2\abs{v_1}+\abs[\big]{2\pi\iu(k-\tfrac12)v_1-v_2}+\abs[\big]{k-\tfrac12}_1\abs{v_2}}.\label{eq:(B)}
\end{align}
By \eqref{eq:(C)} and \eqref{eq:(B)} we have
\[\inf_{s\in\C}f_1(k,v,s)\le C_Lf_2(k,v).\qedhere\]
\end{case}
\end{proof}

%-------------------------------------------------------------------
%-------------------------------------------------------------------
\begin{appendix}
%-------------------------------------------------------------------
%
%-------------------------------------------------------------------
%-------------------------------------------------------------------
\section{Selected auxiliary results}
%-------------------------------------------------------------------
%
For easy reference we collect a couple of auxiliary results in this appendix. 
It is well-known that commuting orthogonal matrices are simultaneously quasidiagonalisable, see \eg \cite[Corollary 2.5.11(c), Theorem 2.5.15]{Horn2013}:
\begin{Theorem}\label{Theorem:simultaneousdiagonalisation}
Let $\mathcal S\subset\O(n)$ be a nonempty commuting family of real orthogonal matrices. Then there exist a real orthogonal matrix $Q$ and a nonnegative integer $q$ such that, for each $A\in \mathcal S$, $Q^TAQ$ is a real quasidiagonal matrix of the form
\[\Lambda(A)\oplus R(\theta_1(A))\oplus\dots\oplus R(\theta_q(A))\]
in which each $\Lambda(A)=\diag(\pm1,\dots,\pm1)\in\R^{(n-2q)\times(n-2q)}$,
$R(\theta)=(\begin{smallmatrix}\cos \theta&-\sin\theta\\ \sin\theta&\cos\theta\end{smallmatrix})$ is the rotation matrix and each $\theta_j(A)\in[0,2\pi)$.
\end{Theorem}
Let $\sigma_{\min}(M)$ and $\norm{M}$ denote the minimum singular value and the Frobenius norm of a matrix $M$, respectively. We have the following singular value inequality, see Corollary 9.6.7 in \cite{Bernstein2009}. 
\begin{Theorem}\label{Theorem:singularvalueinequality}
Suppose $A,B\in\C^{d\times d}$. Then 
\[\norm{AB}\ge \sigma_{\min}(A)\norm{B}\quad\text{and}\quad \norm{AB}\ge \norm{A}\sigma_{\min}(B). \]
\end{Theorem}
Kronecker's approximation theorem, see \eg Corollary 2 on page 20 in \cite{Hlawka1991}, reads:
\begin{Theorem}\label{Theorem:Kronecker}
For each irrational number $\alpha$ the set of numbers $\set{\alpha n\text{ reduced modulo 1}}{n\in\N}$ is dense in the whole interval $[0,1)$.
\end{Theorem}
We also need the following minimax theorem of Tur\'{a}n on generalized power sums, see Theorem 11.1 on page 126 in \cite{Turan1984}.
\begin{Theorem}\label{Theorem:TuranTheorem}
Let $b_1,\dots,b_n,z_1,\dots,z_n\in\C$.
If $m$ is a nonnegative integer and the $z_j$ are restricted by
\[\frac{\min_{\mu\neq\nu}\abs{z_\mu-z_\nu}}{\max_j\abs{z_j}}\ge\delta\,(>0),\quad z_j\neq0\]
then the inequality
\[\max_{\nu=m+1,\dots,m+n}\frac{\abs[\big]{\sum_{j=1}^nb_jz_j^\nu}}{\sum_{j=1}^n\abs{b_j}\,\abs{z_j}^\nu}\ge\frac1n\parens[\bigg]{\frac\delta2}^{n-1}\]
holds.
\end{Theorem}
Finally, we include a short argument showing that the two seminorms $\norm{\Pi_{\rm rot} \nabla \fdot}_{L^2(\Omega)}$ and $\norm{\nabla \Pi_{\rm iso} \fdot}_{L^2(\Omega)}$ considered in the introduction are equivalent: 
\begin{align}\label{eq:PirotD-DPiiso-equiv}
  \norm{\Pi_{\rm rot} \nabla u}_{L^2(\Omega)} 
  \le \norm{\nabla \Pi_{\rm iso} u}_{L^2} 
  \le C \norm{\Pi_{\rm rot} \nabla u}_{L^2(\Omega)}
\end{align}
for all $u \in H^1(\Omega,\R^d)$. 
The first inequality is clear. For the second, if $\Pi_{\rm iso} u = u - A \fdot - c$ and $\Pi_{\rm rot} \nabla u = \nabla u - A'$, then Poincar{\'e}'s inequality gives $\norm{u - A \fdot - c}_{L^2} \le \norm{u - A' \fdot - c'}_{L^2} \le C \norm{\nabla u - A'}_{L^2}$ for some $c' \in \R^d$ and hence $\norm{(A'-A) \fdot + c'-c}_{L^2} \le C \norm{\nabla u - A'}_{L^2}$, which implies $\norm{A'-A} \le C \norm{\nabla u - A'}_{L^2}$. Thus, $\norm{\nabla u - A}_{L^2} \le \norm{A' - A}_{L^2} + \norm{\nabla u - A'}_{L^2} \le C \norm{\nabla u - A'}_{L^2}$.

\end{appendix}

\bibliography{ObjStrcKornLit}

\end{document}